\documentclass[11pt]{amsart}
\usepackage{macro}
\begin{document}
\title{A mirror theorem for Gromov-Witten Theory without Convexity}
\author{Jun Wang}
\begin{abstract} We prove a genus zero Givental-style mirror theorem for all
  complete intersections in proper toric Deligne-Mumford stacks, which provides
  an explicit slice called big $I-$function on
  Givental's Lagrangian cone for such targets. In particular, we remove a
  technical assumption called convexity needed in previous quantum Lefschetz theorem. 
\end{abstract}

\maketitle
\tableofcontents
\section{Introduction}
In the past few decades, following predictions from string theory~\cite{Candelas_1991},
a series of results known as mirror theorems has been proven; an
incomplete list is~\cite{Givental_1996, Coates_2015,
coates07_quant_rieman_roch_lefsc_serre,
zinger08_reduc_genus_gromov_witten_invar, Givental_1998, Coates_2009,
Lian_1999, guo17_mirror_theor_genus_two_gromov}. These theorems reveal elegant patterns and deep structures encoded in
the collection of Gromov-Witten invariants of a given
symplectic manifold or orbifold $X$. However, the scope of these results, and much
of Gromov-Witten theory in general, is limited to the world of toric geometry;
in all cases above, $X$ is a complete intersection in a toric variety or
certain complete intersection\footnote{See the discussion of convexity below.} in a toric
stack~\cite{coates2019some}. The essential reason
for this is that one of most efficient way to compute Gromov-Witten invariants is to
utilize the technique of the localization theorem~\cite{Atiyah_1995, Graber_1999}, which
requires the targets to be carried with a good torus action (e.g. toric
variety/stack).

Smooth hypersurfaces (or complete intersections more general) in toric Deligne-Mumford stacks are the next class of
spaces to consider, but much less is known in this situation. The main difficulty
comes from that a hypersurface in a toric stack doesn't have any good torus action
in general. Hence one can't directly apply localization theorem to compute the
Gromov-Witten invariants of the toric hypersurface.
Alternatively, the usual way to compute the Gromov-Witten invariants of a given
hypersurface is to use \emph{quantum Lefschetz
principle}~\cite{kim03_funct_inter_theor_conjec_cox_katz_lee}, which relates the
twisted Gromov-Witten invariants of an ambient space $X$ to the Gromov-Witten
invariants of its hypersurface $Y$ which is the zero locus of a section of a given line bundle
$L$ on $X$. However, there is a technical assumption called \emph{convexity} for the line bundle $L$ to
apply the \emph{quantum Lefschetz principle}. The convexity says, for any stable
map $f:C\rightarrow X$, one has
$$H^{1}(C,f^{*}L)=0\ ,$$
which holds, for example, when the ambient space $X$ is a projective variety,
the source curve $C$ is of genus zero and $L$ is a positive line bundle. Even worse, a counterexample was found
in~\cite{Coates_2012} that \emph{quantum Lefschetz principle} can fail for
positive hypersurfaces in orbifolds. So there are limited methods to compute the genus zero
Gromov-Witten invariants of orbifold hypersurfaces where the convexity fails
(see~\cite{Guere:2019pjc} for a recent update for certain hypersurfaces in
weighted projective spaces), and a mirror theorem\footnote{In Givental's
  formalism, we need to construct an explicit slice on the Lagrangian cone.} for these targets is lacking
for a long time in the literature.

The aim of this paper is to prove a mirror theorem for all complete intersections
in proper toric Deligne-Mumford stacks, where the convexity is not required as
in previous such mirror theorem. Our proof of the mirror theorem relies heavily on quasimap theory, which actually corresponds to the quasimap wall-crossing
conjecture for big $\mathbb I-$function (c.f~\cite{ciocan-fontanine2016}).

\subsection{Main results and Ideas of proof}
\subsubsection{Big I-function} Let $X$ be a \emph{proper toric
  Deligne-Mumford stack} constructed by a GIT data $(W=\oplus_{\rho\in
  [n]}\mathbb C_{\rho},G=(\C^{*})^{k},\theta)$, and $Y\subset
X$ is a complete intersection with respect to a direct sum of line bundles
$\oplus_{b=1}^{c} L_{\tau_{b}}$ on $X$ (See \S \ref{sec:I-fun} for a more precise setting). The big $I-$function (or $I-$function in short) of the toric stack complete intersections can be written
down as follows, of which is obtained by quasimap theory~\cite{Ciocan_Fontanine_2014, Cheong_2015,ciocan-fontanine2016}: 
\begin{equation}\label{eq:mirror}
\begin{split}
\mathbb I(q,t,z)&=\mathrm{exp}\big(\frac{1}{z}\sum_{i=1}^{l}t_{i}u_{i}(c_{1}(L_{\pi_{j}})+\beta(L_{\pi_{j}})z)\big)\\
      &\sum_{\beta\in \mathrm{Eff}(W,G,\theta)}
      \frac{\prod_{\rho:\beta(L_{\rho})<
             0}\prod_{\beta(L_{\rho})< i<0}(D_{\rho}+(\beta(L_{\rho})-i)z)}{\prod_{\rho:\beta(L_{\rho})>
             0}\prod_{0\leq i<
             \beta(L_{\rho})}(D_{\rho}+(\beta(L_{\rho})-i)z)}\\
      &\cdot \frac{\prod_{b:\beta(L_{\tau_b})>0}\prod_{i:0\leq i<\beta(L_{\tau_b})}(c_{1}(L_{\tau_b})+(\beta(L_{\tau_b})-i)z)}{\prod_{b:\beta(L_{\tau_b})<0}\prod_{i:\beta(L_{\tau_b})<i<0}(c_{1}(L_{\tau_b})+(\beta(L_{\tau_b})-i)z)}
      i_{*}(s^{!}_{E_{\geq 0},loc}([Z^{ss}_{\beta}/(G/\langle {g^{-1}_{\beta}} \rangle)])) \ .
\end{split}
\end{equation}
We remark here $i_{*}(s^{!}_{E_{\geq 0},loc}([Z^{ss}_{\beta}/(G/\langle {g^{-1}_{\beta}}
\rangle)]))$ is an element in the twister sector
$H^{*}(\bar{I}_{g^{-1}_{\beta}}Y,\mathbb Q)$ (see Corollary \ref{cor:com-graph}).
$t=\sum_{i=1}^{l}t_{i}u_{i}(c_{1}(L_{\pi_{j}})$ is an
element in $H^{*}(Y,\mathbb Q)[t_{1},\cdots,t_{l}]$, where $\{\pi_{j},1\leq j\leq k\}$ are standard
representations of $G=(\C^{*})^{k}$. Moreover, $l$ can be any nonnegative integers and $u_{i}$ can be
any polynomial on first Chern class of line bundles $L_{\pi_{1}}, \cdots,
L_{\pi_{k}}$. Please see \S \ref{sec:I-fun} for more detail about the terminology appearing
in $\mathbb I(q,t,z)$. 

Now we state our main theorem, 
\begin{theorem}[Main Theorem]\label{thm:main}
$-z \mathbb I(q,t,-z)$ is a slice on Givental's Lagrangian cone of $Y$. More
explicitly, there exists a mirror transformation $\mu(q,t,z)$ such that we have the following identity:
  \begin{equation}\label{eq:wc} \mathbb I(q,t,z)=J(q,\mu(q,t,y),z),
  \end{equation}
  where $J(q,\mu(q,t,y),z)$ is defined by the $J-$function\footnote{Here we treat $J(q,\mathbf{t},z)$ as a
    functional, which means, after fixing the input $\mathbf{t}$, we think
    $J(q,\mathbf{t},z)$ as a formal series on the Novikov variable $q$ and the
    variable $z$, which doesn't involve the variable $y$.}
  \begin{equation*}
    \begin{split} J(q,\mathbf{t},z):=&\mathbbm{1}_{Y}+\frac{\mathbf{t}(z)}{z}\\
      &+\sum_{\beta\in \mathrm{Eff}(W,G,\theta)}\sum_{m\geq
        0}\frac{q^{\beta}}{m!}\phi^{\alpha}\langle
      \mathbf{t}(-\bar{\psi}_{1}),\cdots,\mathbf{t}(-\bar{\psi}_{m}),\frac{\phi_{\alpha}}{z(z-\bar{\psi}_{\star})}\rangle_{0,[m]\cup
        \star,\beta}\ .
    \end{split}
  \end{equation*}
  Here the input $\mathbf{t}$ is an element in $(q,t)H^{*}(\bar{I}_{\mu}Y,\mathbb
  Q)[y][[t_{1},\cdots,t_{l}]][[\mathbf{Eff}(W,G,\theta)]]$\footnote{It means that $\mathbf {t}$
    admits an expression as $\sum_{(\beta,i_{1},\cdots,i_{l})\neq
      0}q^{d}t_{1}^{i_{1}}\cdots t_{l}^{i_{l}}f_{d,i_{1},\cdots,i_{l}}$, where $f_{d,i_{1},\cdots,i_{l}}\in
    H^{*}(\bar{I}_{\mu}Y, \mathbb Q)[y]$. This choice of input $\mathbf{t}$
    gives a much less general definition of Givental's $J-$function in the usual
    literature, but it suffices for the need in this paper.}, and
  $\mathbf{t}(z)$ (resp. $\mathbf{t}(-\bar{\psi}_{i})$) means that we replace
  the variable $y$ in $\mathbf{t}$ by $z$ (resp. $-\bar{\psi}_{i}$). 
\end{theorem}
Note here any degree $\beta\in
\mathrm{Eff}(W,G,\theta)$ of $X$ (c.f.\ref{def:novi-degree}), we will denote the Gromov-Witten invariant
$$\phi^{\alpha}\langle
\mathbf{t}(-\bar{\psi}_{1}),\cdots,\mathbf{t}(-\bar{\psi}_{m}),\frac{\phi_{\alpha}}{z(z-\bar{\psi}_{\star})}\rangle_{0,[m]\cup
  \star,\beta}$$ to be
  $$\sum_{\substack{d\in \mathrm{Eff}(AY,G,\theta)\\ i_{*}(d)=\beta}}\phi^{\alpha}\langle
  \mathbf{t}(-\bar{\psi}_{1}),\cdots,\mathbf{t}(-\bar{\psi}_{m}),\frac{\phi_{\alpha}}{z(z-\bar{\psi}_{\star})}\rangle_{0,[m]\cup
    \star,d}\ ,$$ 
where $\mathrm{Eff}(AY,G,\theta)$ is semigroup of the degree of $Y$.

Now let's define precisely $\mu(q,t,z)$ appearing in our main theorem, Write $z \mathbb I(q,t,z)$ as a formal
Laurent series in variable $z$:
$$\cdots+\mathbb I_{-1}(q,t)z^{2}+\mathbb I_{0}(q,t)z+\mathbb I_{1}(q,t)+\mathcal O(z^{-1}) ,$$
and define $\mu(q,t,z)$ to be the truncation in nonnegative $z$ powers:
$$\mu(q,t,z):=[z \mathbb  I(q,t,z)-\mathbbm
1_{Y}z]_{+}=\cdots+\mathbb I_{-1}(q,t)z^{2}+(\mathbb I_{0}(q,t)-\mathbbm
1_{Y})z+\mathbb I_{1}(q,t) \ .$$ By
the definition of $\mathbb I(q,t,z)$, $z \mathbb
I(q,t,z)$ admits an asymptotic expansion in $q,t$:
$$z \mathbb I(q,t,z)=z\mathbbm{1}_{Y}+\mathcal O(q)+\mathcal O(t), $$
which implies that $\mu(q,t,z)=\mathcal O(q)+\mathcal
O(t)$.

\subsubsection{Sketch of the proof of the main theorem} Now
Before sketching the proof of the main theorem, let's analyze $\mu(q,t,z)$ in
more detail. Let $\mathbb I(q,z):=\mathbb I(q,0,z)$, we can expand $\mathbb
I(q,z)$ as
$$\mathbb I(q,z)=\sum_{\beta\in \mathrm{Eff}(W,G,\theta)}q^{\beta}\mathbb
I_{\beta}(z)\ ,$$ where $\mathbb I_{\beta}(z)\in H^{*}(\bar{I}_{\mu}Y,\mathbb
Q)[z,z^{-1}]$. Then we can decompose $\mathbb I(q,t,z)$ as a formal sum
$$\mathbb I(q,t,z)=\sum_{\beta\in \mathrm{Eff}(W,G,\theta),p\in
  \mathbb N}q^{\beta}\frac{\boldsymbol{t}^{p}}{p!z^{p}}\mathbb I_{\beta}(z)\ .$$
where
$\boldsymbol{t}=\sum_{i=1}^{l}t_{i}u_{i}(c_{1}(L_{\pi_j})+\beta(L_{\pi_j})z)$.
For nonzero pair $(\beta,p)$, set $\mu_{\beta,p}:=[\frac{t^{p}z\mathbb
  I_{\beta}(z)}{p!z^{p}}]_{+}$ as the truncation in nonnegative $z$ powers. We
note that $\mu_{\beta,p}(z)$ is a polynomial in $H^{*}(\bar{I}_{\mu}Y,\mathbb
Q)[t_{0},\cdots,t_{l},z]$ of homogeneous degree $p$ in variables
$t_{1},\cdots,t_{l}$. For convenience we set $\mu_{0}(z)=\mu_{(0,0)}(z)=0$. Then
we have write $\mu(q,t,z)$ as a sum
$$\mu(q,t,z)=\sum_{\substack{p\in \mathbb Z_{\geq 0},\; \beta\in
    \mathrm{Eff}(W,G,\theta)\\(\beta,p)\neq 0}}q^{\beta}\mu_{\beta,p}(z)\ ,$$

Multiply by $z$ on both sides of \eqref{eq:wc}, we observe that, to prove the
main theorem, it suffices to prove that, for arbitrary $(\beta,p)\in
\mathrm{Eff}(W,G,\beta)\times \mathbb N$, we have:
\begin{equation}\label{eq:coewc}
  \begin{split}
    z\frac{\boldsymbol{t}^p}{p!z^{p}}&\mathbb I_{\beta}(z)=\delta_{(\beta,p),0}z+\mu_{\beta,p}(z)\\
    &+\sum_{m=0}^{\infty}\sum_{\substack{\beta_{0}+\beta_{1}+\cdots+\beta_{m}=\beta\\
        p_{1}+\cdots+p_{m}=p}}\frac{1}{m!}\phi^{\alpha}\langle\mu_{\beta_{1},p_{1}}(-\bar{\psi}_{1}),\cdots,\mu_{\beta_{m},p_{m}}(-\bar{\psi}_{m}),\frac{\phi_{\alpha}}{z-\bar{\psi}_{\star}}
    \rangle_{0,[m]\cup \star,\beta_{0}}\ .
  \end{split}
\end{equation}
To prove \eqref{eq:wc}, it suffices to prove \eqref{eq:coewc} for every
$(\beta,p)\in \operatorname{Eff(W,G,\theta)} \!\times\! \mathbb N$. Observe that
when $(\beta,p)=0$, \eqref{eq:coewc} reduces to
$$z=z\ ,$$
which holds obviously. This will be the base case for our inductive proof of
\eqref{eq:coewc}.

To prove \eqref{eq:coewc}, observe that the nonnegative parts in $z$ on both
sides of \eqref{eq:coewc} are equal to $\mu_{\beta,p}(z)$ by the very definition
of $\mu_{\beta,p}(z)$. It follows that, in order to prove \eqref{eq:coewc}, it
suffices to show the truncations in negative $z$ powers of \eqref{eq:coewc}
\begin{equation}\label{eq:coewc2}
  \begin{split}
    [z&\frac{\boldsymbol{t}^p}{p!z^{p}}\mathbb I_{\beta}(z)]_{-}:=\\
    &\sum_{m=0}^{\infty}\sum_{\substack{\beta_{0}+\beta_{1}+\cdots+\beta_{m}=\beta\\p_{1}+\cdots+p_{m}=p}}\frac{1}{m!}\phi^{\alpha}\langle\mu_{\beta_{1},p_{1}}(-\bar{\psi}_{1}),\cdots,\mu_{\beta_{m},p_{m}}(-\bar{\psi}_{m}),\frac{\phi_{\alpha}}{z-\bar{\psi}_{\star}}
    \rangle_{0,[m]\cup \star,\beta_{0}}
  \end{split}
\end{equation} holds. Equivalently, it suffices to show, for any nonnegative
integer $c$, one has
\begin{equation}\label{eq:coewc3}
  \begin{split}
    &[z\frac{\boldsymbol{t}^p}{p!z^{p}}\mathbb I_{\beta}(z)]_{z^{-c-1}}:=\\
    &\sum_{m=0}^{\infty}\sum_{\substack{\beta_{0}+\beta_{1}+\cdots+\beta_{m}=\beta\\p_{1}+\cdots+p_{m}=p}}\frac{1}{m!}\phi^{\alpha}\langle\mu_{\beta_{1},p_{1}}(-\bar{\psi}_{1}),\cdots,\mu_{\beta_{m},p_{m}}(-\bar{\psi}_{m}),\phi_{\alpha}\bar{\psi}_{\star}^{c}
    \rangle_{0,[m]\cup \star,\beta_{0}} \ .
  \end{split}
\end{equation}

The idea to prove \eqref{eq:coewc3} is to show that both sides of
\eqref{eq:coewc3} satisfy the same recursive relations (see Theorem
\ref{thm:rec1} and Theorem \ref{thm:rec2}) by induction on
$\beta(L_{\theta})+p$. This is done by considering two master spaces (see
\S\ref{subsec:Construction of the master space} and \S\ref{subsec:Construction
  of the master space at infinity sector}), which are root stack modification of
the twisted graph spaces found
in~\cite{clader17_higher_genus_quasim_wall_cross_via_local,
  clader17_higher_genus_wall_cross_gauged}. Then we apply virtual localization
to express two auxiliary cycles (see \eqref{eq:mainintegral1} and
\eqref{eq:mainintegral2}) corresponding to two master spaces in graph sums and
extract $\lambda^{-1}$ coefficients ($\lambda$ is an equivariant parameter).
Finally, the polynomiality of the two auxiliary cycles implies that the
coefficients for $\lambda^{-1}$ term must vanish, thus they yields recursive
relations (see also Theorem \ref{thm:rec1} and Theorem \ref{thm:rec2}) to finish
the proof of the quasimap wall-crossing.


\begin{remark}
  Previously, the big $\mathbb I-$function quasimap wall-crossing conjecture was
  proved for GIT targets with good torus action~\cite{ciocan-fontanine2016},
  i.e. zero dimensional torus fixed orbits are finite and the one dimensional
  torus fixed orbits are isolated. Thus the targets we treat here provides first
  examples with no good torus action.
  
  The proof of the mirror theorem here is quite robust; the main geometrical
  construction including twisted graph space and root stack construction, and
  recursive relations can be directly generalized to all proper GIT targets
  considered in quasimap theory. Hence we expect the method developed here can
  be used to prove the genus zero quasimap wall-crossing conjecture for all
  proper GIT targets considered in quasimap theory.

  The main geometrical input in the proof of wall-crossing here is inspired by
  the twisted graph space used in~\cite{clader17_higher_genus_wall_cross_gauged,
    clader17_higher_genus_quasim_wall_cross_via_local}, where they use the genus
  zero quasimap wall-crossing for the small $I-$function as input to prove the
  high genus quasimap wall-crossing. So it may be surprising that certain
  modification of the twisted graph space can be used to prove the genus zero
  $\mathbb I-$function quasimap wall-crossing directly.
\end{remark}

\begin{remark}
  Since the appearance of the first version of this paper, it's realized by the
  author that the proof of the mirror theorem doesn't need the part of the
  computation of the $I-$function. To keep this paper short, we omit the
  part about the computation of the $I-$function here. Readers who are interested in
  learning how these $I-$functions come from are encouraged to look at the first
  version of this paper on arXiv (it's done only for positive hypersurfaces but can
  be extended to all complete intersections) or to refer to the forthcoming work by Nawaz
  Sultani and Rachel Webb, in which they obtain small $I$-functions for all
  complete intersections in GIT quotients with possible non-abelian group
  action, and their method works directly with quasimap graph space and avoids
  using $p$-fields appearing in the author's first version.

  During the preparation of this work, the author learns that Yang Zhou has used
  a totally different method to prove the quasimap wall-crossing conjecture for
  all GIT quotients and all genera~\cite{zhou19_quasim_wall_cross_git_quotien},
  which implies the mirror theorem proved in this paper without exponential
  factor. The author also learns that Felix Janda and Nawaz Sultani have a
  different way of computing the (S-extended) $I-$functions for some
  hypersurfaces in weighted projective spaces and use them to calculate
  Gromov-Witten invariants.
\end{remark}

\subsection{Outline}
The rest of this paper is organized as follows. In \S\ref{sec:quasimapwithfield}, we will recall
the quasimap theory, the author wants to
draw readers' attention to the language of $\theta'$-stable quasimaps (see Remark
\ref{rmk:theta-quasimap}), where $\theta'$ can be a \emph{rational character}, because it is more suitable
than the language of $\epsilon$-stable quasimaps for the later construction of the master space
in \S\ref{sec:twistedgraphspace}. In \S\ref{sec:I-fun}, we collect some
important facts about (rigidified) inertia stack of toric stack complete
intersections, and compare them with rigidified inertia stack of toric stacks.
Some special cycles in the inertia stack will be discussed as they will be
appeared in our $\mathbb I-$function. In \S\ref{sec:twistedgraphspace} and \S\ref{sec:masterspace-inf}, we will construct two master spaces which
carry $\C^{*}$-actions, a very explicit $\C^{*}$-localization computation which
is based on localization computations~\cite{clader17_higher_genus_quasim_wall_cross_via_local, Janda2017} will be
presented, this part is technical, we encourage the reader to skip to go to \S\ref{sec:rec-rel} first and to refer back when needed. In \S\ref{sec:rec-rel}, we will calculate two auxiliary cycles corresponding to
the two master spaces via localization, they provide recursive relations
to prove the genus zero quasimap wall-crossing conjecture for toric
stack hypersurfaces. 

\medskip
$\textbf{Notations}$: In this paper, we will always assume that all
algebraic stacks and algebraic schemes are locally of finite type over the base
field $\C$. Given a GIT target $(W,G,\theta)$, we will use symbols $\mathfrak X,
\mathfrak Y$... to mean the quotient stack $[W/G]$, symbols $X,Y$... to mean the
corresponding GIT stack quotient $[W^{ss}(\theta)/ G]$, $I_{\mu}X$,
$I_{\mu}Y$... to mean the corresponding (cyclotomic) inertia stacks, and $\bar{I}_{\mu}X$,
$\bar{I}_{\mu}Y$...to mean the corresponding rigidified inertia stacks.

We will use the following construction a lot throughout this paper. 
\begin{definition}[\textbf{Borel Construction}]\label{def:bor}
Let $G$ be a linear algebraic group and $W$ be a
variety. Fix a right $G$-action on the variety $W$. For any character $\rho$ of
$G$, we will denote $L_{\rho}$ to be the line bundle on the quotient stack
$[W/G]$ defined by
$$W \!\times\!_{G}\;\mathbb C_{\rho}:=[(W \!\times\! \C_{\rho})/G]\ ,$$
where $\C_{\rho}$ is the $1$-dimensional representation of $G$ via $\rho$ and
the action is given by
$$(x,u)\cdot g=(x\cdot g,\rho(g)u)\in W \!\times\! \mathbb C_{\rho}$$
for all $(x,u)\in W \!\times\! \mathbb C_{\rho}$ and $g\in G$. For any linear
algebraic group $T$, if we have a left $T$-action on $W$ which commutes with the
right action of $G$, we will lift the line
bundle $L_{\rho}$ defined above to be a $T$-equivariant line bundle, which is induced from the (left)
$T$ action on $W \!\times\! \C_{\rho}$ in the way that $T$ acts on $\C_{\rho}$
trivially. By abusing notations, we will use the same notation $L_{\rho}$ to mean
the corresponding invertible sheaf (or $T-$equivariant invertible sheaf)
over $[W/G]$ unless stated otherwise.
\end{definition}

Acknowledgments: First, I want to express my deep gratitude to my advisor Hsian-Hua Tseng for his support and guidance. I want to especially thank his excellent talk on relative Gromov-Witten
invariants and orbifold Gromov-Witten invariants at OSU, which draws my attention to double ramification cycles, which eventually leads to the birth of this paper.
I would like to thank the organizers of the 2017 FRG workshop ``Crossing the Walls in
Enumerative Geometry'' at Columbia, which exposed me to the most recent process
in the subject as well as many inspiring ideas. Part of the paper is based on
the author's thesis at OSU. I also want to thank Yang Zhou and Cristina Manolache
on discussions on quasimap theory and thank Honglu Fan and Fenglong You on
discussions on Givental's formalism. I want to thank Felix Janda on the presentation of the big $I-$function.

\section{Background on quasimaps}\label{sec:quasimapwithfield}
We first recall the definition of a \emph{quasimap} to a GIT target, our main reference
is~\cite{Ciocan_Fontanine_2014, Cheong_2015, ciocan-fontanine2016}. By a GIT target, we mean a triple
$(W,G,\theta)$, where $W$ is an irreducible affine variety with locally complete
intersection (l.c.i) singularity, $G$ is a reductive group equipped with a right
$G-$action on $W$ and $\theta$ is an (integral) character of $G$. Denote by
$\mathfrak X:=[W/G]$ the quotient stack. Denote by $W^{ss}$ (or
$W^{ss}(\theta)$) the semistable locus in $W$, and by $W^{s}$ (or
$W^{s}(\theta)$) the stable locus. Throughout out this paper, for a GIT target
$(W,G,\theta)$, we will always assume that $W^{ss}(\theta)=W^{s}(\theta)$ and
the \emph{GIT stack quotient}
$$X:=[W^{ss}(\theta)/G]$$
is a smooth \emph{Deligne-Mumford stack}, under which condition, $X$ is always
semi-projective, i.e. it's proper over the affine GIT quotient
$\mathrm{Spec}(\mathbb C[W]^{G})$ by the proj-construction of GIT
quotient~\cite[\S 2.2]{Cheong_2015}\cite{MR1304906}:
$$\underline{X}=\mathbf{Proj}\oplus_{n=0}^{\infty}\Gamma(W,W \!\times\! \mathbb C_{n\theta})^{G}.$$

Let $\mathbf{e}$ be the least common multiple of the exponents $|\mathrm{Aut}
(\bar{x})|$ of automorphism groups $\mathrm{Aut} (\bar{x})$ of all geometric
points $\bar{x}\rightarrow X$ of $X$. Then, for any character $\rho$ of $G$, the
line bundle $L_{\rho}^{\otimes \mathbf{e}}$ is the pullback of a line bundle
from the coarse moduli $\underline{X}$ of $X$, here the line bundle $L_{\rho}$
is defined by the Borel (mixed) construction \ref{def:bor}.


\begin{definition}\label{quasimapwithfield}
  Given a scheme $S$ over $Spec(\mathbb C)$, $f=((C,q_{1},\cdots,q_{m}),P,x)$ is
  called a $\textit{quasimap}$ over $S$ (alternatively \emph{$\theta$-quasimap}
  over $S$) of class $(g,m,\beta)$ if it consists of the following data:
  \begin{enumerate}
  \item $(C,q_{1},\cdots,q_{m})$ is a flat family of genus $g$ twist curves over
    $S$~\cite[\S 4]{Dan_Abramovich_2008}, and $m$ gerbe marked sections
    $q_{1},\cdots, q_{m}$ over $S$, here we don't require the gerbe sections to
    be trivialized;
  \item $P$ is a principal $G$-bundle on $C$;
  \item $x$ is a section of the affine $W-$bundle $(P \!\times\! W)/G$ over $C$
    so that it determines a representable morphism $[x]:C \rightarrow \mathfrak
    X=[W/G]$ as the composition
     $$\xymatrix{
       C\ar[r]^-{x}&(P\times W)/G\ar[r]&[W/G]\ . }$$ We say that the quasimap
     $f$ is of degree $\beta\in Hom_{\mathbb Z}(Pic(\mathfrak X),\mathbb Q)$ if
     $\beta(L)=deg([x]^{*}L)$ for every line bundle $L\in Pic(\mathfrak X)$;
   \item The base locus of $[x]$ defined by $[x]^{-1}(\mathfrak X\backslash X)$
     is purely of relative dimension zero over $S$.
   \end{enumerate}
 \end{definition}

 Sometimes we may also use the notation $f:(C,\boldsymbol{q}=(q_{i}))\rightarrow
 \mathfrak X$ to mean a quasimap (or
 $\theta$-quasimap). A quasimap $f$ is \emph{prestable} (or
 \emph{$\theta-$prestable}) if the base locus are away from nodes and markings.

\begin{remark}
  We can extend the definition of $\theta$-prestable quasimap to allow any
  \emph{rational character $\theta'$} such that $\theta'$-prestable quasimap is
  same as $\alpha\theta'$-prestable quasimap for any $\alpha\in \mathbb Q_{>0}$.
\end{remark}

Consider a prestable quasimap $f$, since the base point is away from nodes and
marking points, for each $q\in C$, as in \cite[Definition
7.1.1]{Ciocan_Fontanine_2014}, we define the length function $l_{\theta}(q)$ as
follows:
\begin{equation}\label{eq:quasi-stablity}
  l_{\theta}(q)=min\{\frac{([x]^{*}s)_{q}}{n}|\; s\in \Gamma(W,W \!\times\! \mathbb C_{n\theta})^{G},\,[x]^{*}s\neq
  0,\,n\in \Z_{>0}\}\ ,
\end{equation}
where $([x]^{*}s)_{q}$ is the coefficient of the divisor $([x]^{*}s)$ at $q$.
Note here the length function $l_{\theta}$ depends on the \emph{integral}
character $\theta$. We have the following important observation about the length
function $l_{\theta}$: choose $\alpha\in \mathbb Q_{> 0}$ such that
$\theta'=\frac{1}{\alpha}\theta$ is another integral character. Then
$$l_{\theta}=\alpha l_{\theta'}\ ,$$
then the length function $l_{\theta}$ can be defined for any \emph{rational
  character} $\theta'$, i.e. choose $\alpha\in \mathbb Q_{>0}$ and an integral
character $\theta$ such that $\theta'=\alpha \theta$, then we define
$$l_{\theta'}:=\alpha l_{\theta}$$
as in~\cite[Definition 2.4]{ciocan-fontanine2016}, note the definition of
$l_{\theta'}$ is independent of decomposition of $\theta'$ as a product of
positive rational number $\alpha$ and an integral character $\theta$ by the
above observation.

Fix a positive rational number $\epsilon\in \mathbb Q_{>0}$. Given a prestable
quasimap $f$ over $Spec(\mathbb C)$, we say $f$ is a \emph{$\epsilon$-stable}
quasimap to $X$ if $f$ satisfies the following stability condition:
\begin{enumerate}
\item the $\mathbb Q-$line bundle
  $(\phi_{*}([x]^{*}L_{\mathbf{e}\theta}))^{\frac{\epsilon}{\mathbf{e}}}\otimes\omega_{\underline{C}}^{log}$
  on the coarse moduli curve $\underline{C}$ of $C$ is ample. Here
  $\phi:C\rightarrow \underline{C}$ is the coarse moduli map. Note the line
  bundle $[x]^{*}L_{\mathbf{e}\theta}$ on $C$ a pullback of a line bundle on the
  coarse curve $\underline{C}$ by the choice of $\mathbf{e}$ and the prestable
  condition. Here
  $\omega_{\underline{C}}^{log}=\omega_{\underline{C}}(\sum_{i=1}^{m}
  \underline{q_{i}})$ is the log dualizing invertible sheaf of the coarse moduli
  $\underline{C}$;
\item $\epsilon l_{\theta}(q)\leq 1$ for any $q\in C$.
\end{enumerate}

\begin{remark}[\emph{$\theta'$-quasimap}]\label{rmk:theta-quasimap}
  Using the above generalization of length function $l_{\theta'}$ for a rational
  character $\theta'$, we can give the definition of \emph{$\theta'$-stable
    quasimap}: given a $\theta'$-prestable quasimap
  $f=((C,q_{1},\cdots,q_{m}),[x])$, we say $f$ is a \emph{$\theta'$-stable}
  quasimap to $\mathfrak X$ if
  \begin{enumerate}
  \item the $\mathbb Q-$line bundle
    $(\phi_{*}([x]^{*}L_{\mathbf{b}\mathbf{e}\theta'}))^{\frac{1}{\mathbf{b}\mathbf{e}}}\otimes\omega_{\underline{C}}^{log}$
    on the coarse moduli curve $\underline{C}$ of $C$ is ample. Here
    $\phi:C\rightarrow \underline{C}$ is the coarse moduli map, and $\mathbf{b}$
    is a positive integer which makes $\mathbf{b}\theta'$ an integral character.
    Note the ampleness is dependent of choice of the positive integer
    $\mathbf{b}$.
  \item $l_{\theta'}(q)\leq 1$ for any $q\in C$.
  \end{enumerate}
  
  Given a GIT target $(W,G,\theta)$, following \cite[Propsition
  2.7]{ciocan-fontanine2016}, an essentially equivalent definition about
  $\epsilon$-stable quasimaps to $X$ is, but from a different point of view, the
  concept of a \emph{$\epsilon\theta-$stable quasimap} to $\mathfrak X$. The
  concept of $\theta'$-stable quasimap will play an important role in the
  construction of master space in section \ref{sec:twistedgraphspace}. For a
  \emph{rational character $\theta'$ of $G$}, we will use the notation
  $Q^{\theta'}_{g,m}(\mathfrak X,\beta)$ to mean the moduli stack of
  $\theta'$-stable quasimaps to the quotient stack $\mathfrak X$ of class
  $(g,m,\beta)$. If we choose $\theta'=\epsilon\theta$, then the space $Q^{\theta'}_{g,m}(\mathfrak X,\beta)$ is same as the space
  $Q^{\epsilon}_{0,m}([W^{ss}(\theta)/G],\beta)$ of $\epsilon$-stable quasimaps
  we introduced before.  
\end{remark}

We call a prestable quasimap $f$ over a scheme $S$ is \emph{$\epsilon$-stable}
if for every $\mathbb C$-point $s$ of $S$, the restriction of $f$ over $s$ is
$\epsilon$-stable. We call $f$ is \emph{$0+$stable} if $f$ is $\epsilon-$stable
for every positive rational number $\epsilon\in \mathbb Q_{>0}$.

\begin{definition}\label{def:novi-degree}
  A group homomorphism $\beta \in \Hom _{\mathbb Z}(\Pic \mathfrak X , \mathbb Q
  )$ is called {\em $L_{\theta}$-effective} if it is realized as a finite sum of
  classes of some quasimaps to $X$. Such elements form a semigroup with identity
  $0$, denoted by $\mathrm{Eff}(W, G, \theta )$.
\end{definition}
We will need the following lemma proved in~\cite[Lemma 2.3]{Cheong_2015}.
\begin{lemma}\label{lem:2.3}
  If $((C, \boldsymbol{q}), [x])$ is a quasimap of degree $\beta$, then
  $\beta(L_{\theta})\geq 0$. Moreover, $\beta(L_{\theta})=0$ if and only if
  $\beta =0$, if and only if the quasimap is constant (i.e., $[x]$ is a map into
  $X$, factored through an inclusion $\mathbb B\Gamma \subset X$ of the
  classifying groupoid $\mathbb B\Gamma$ of a finite group $\Gamma$).
\end{lemma}

Now we will give the definition of \emph{semi-positive line bundles} for quasimaps
used in this paper.
\begin{definition}\label{def:pos-lin}
  Given a GIT data $(W,G,\theta)$ which gives rise to a GIT stack quotient $X$,
  let $L_{\tau}$ be the line bundle associated to a character $\tau$ of $G$.
  Then we call $L_{\tau}$ a semi-positive line bundle on $\mathfrak X$ if
  $$\beta(L_{\tau})\geq 0$$
  whenever $\beta$ is the degree of a quasimap to $X$. Sometimes if the GIT
  model of $X$ and the character $\tau$ are clear in the context, we will also
  call the restricted line bundle $L_{\tau}|_{X}$ a semi-positive line bundle on
  $X$. By abusing notations, we will denote the restricted line bundle on $X$ to
  be $L_{\tau}$ unless stated otherwise.
\end{definition}

\begin{remark}\label{rmk:pos-lin}
  Note the definition of a semi-positive line bundle on $X$ depends on the GIT model
  of $X$ and the character $\tau$. By Lemma \ref{lem:2.3}, the line bundle
  $L_{\theta}$ (or positive tensor power of $L_{\theta}$) is always a semi-positive
  line bundle. Note, using the standard GIT presentations of weighted projective
  spaces, the line bundles $\mathcal O(n)$ ($n>0$) on weighted projective spaces
  are always semi-positive line bundles by the definition above.
\end{remark}

In the following, we will give an explicit description of quasimaps to
\emph{toric Deligne-Mumford stacks}.
\begin{remark}[Quasimaps to toric stack]\label{toricquasimap}
  Recall the construction of a (semi-projective) toric Deligne-Mumford stack (or
  toric stack in short) by a GIT data $(W,G,\theta)$. Let $G=(\C^{*})^{k}$, and
  $W:=\oplus_{i=1}^{n}\C_{\rho_{i}}$ be a direct sum of $1$-dimensional
  representations of $G$ given by the characters $\rho_{i}\in \chi(G)$ for
  $1\leq i\leq n$. We will denote $[n]$ to be the collection of (\emph{not
    necessarily distinct}) characters $\rho_{i}$ of $G$ for $1\leq i\leq n$. The
  toric stack $X$ is defined to be the GIT stack quotient
 $$[W^{ss}(\theta)/ G].$$
 Since we always assume that $W^{ss}(\theta)=W^{s}(\theta)$, then $X$ is a
 \emph{semi-projective Deligne-Mumford stack}.
 
 Then in the definition of quasimaps to the toric stack $X$, we can replace the
 principal $G-$bundle $P$ by $k$ line bundles $(L_{j}:1\leq j\leq k)$ on $C$,
 and replace the section $x$ in the definition of quasimap by $n$ sections
 $$\vec{x}=(x_{i}:1\leq i\leq n)\in \oplus_{\rho\in[n]}\Gamma(C,L_{\rho})\ ,$$
 where $L_{\rho}$ is a line bundle on $C$ defined by
 $$L_{\rho}=\otimes_{j=1}^{k}L^{\otimes m_{j}}_{j}\ ,$$
 where and the numbers $(m_{j}:1\leq j\leq k)$ are determined by the unique
 relation
  $$\rho=\sum_{j=1}^{k}m_{j}\pi_{j}$$
  in the character group $\chi(G)$ of $G$. Here $(\pi_{j}:\;1\leq j\leq k)$ are
  the standard characters of $G=(\C^{*})^{k}$ by projecting to coordinates.
\end{remark}


One novel application of $\theta'-$stable quasimap for a rational character
$\theta'$ is the use of the notion of $(\theta,\boldsymbol{\varepsilon})-$stable quasimap
introduced in~\cite{ciocan-fontanine2016}. 

\begin{definition}\label{def:theta-ep-quasi}[$(\theta',\boldsymbol{\varepsilon})$-stable quasimap]
Given a tuple $\boldsymbol{\varepsilon}=(\varepsilon_{1},\cdots,
\varepsilon_{p})\in \mathbb (Q_{>0}\cap (0,1])^{p}$, we will
call prestable quasimap $\boldsymbol f:=(C,\boldsymbol{q},f:C\rightarrow [W/G]\times [\mathbb C/\C^{*}
]^{p})$ a $(\theta',\boldsymbol{\varepsilon})$-stable quasimap to $\mathfrak X$ of type
$(g,m,\beta)$ if $\boldsymbol{f}$ defines $\theta'\oplus
\bigoplus_{i=1}^{p}\varepsilon_{i}\mathrm{id}_{\C^{*}}$-stable quasimap to
$[W/G]\times [\mathbb C/\C^{*}]^{p}$ of type $(g,m,(\beta, 1,\cdots 1) )$. We
will denote $Q^{(\theta,\boldsymbol{\varepsilon})}_{g,m|p}(\mathfrak X,\beta)$ to be the
moduli stack of $(\theta,\boldsymbol{\varepsilon})$-stable quasimaps to $\mathfrak X$ of type
$(g,m,\beta)$. We call $\boldsymbol{f}$ is $(\theta,(0+)^{p})$-stable if $\boldsymbol f$ is
$(\theta,\boldsymbol{\varepsilon})$-stable for all $\boldsymbol{\varepsilon}\in \mathbb
Q_{>0}^{p}$. And we will denote $Q^{\theta,0+}_{g,m|p}(\mathfrak X,\beta)$ to be
the moduli stack of $(\theta,(0+)^{p})$-stable quasimaps to $\mathfrak X$ of type
$(g,m,\beta)$.
\end{definition}

\begin{remark}\label{rmk:equiv-theta-epl}
It's shown in~\cite{ciocan-fontanine2016} that a
$(\theta,\boldsymbol{\varepsilon})$-stable map to $\mathfrak X$ is equivalent to a
$\boldsymbol{\epsilon}$-weighted $\theta$-stable map to $\mathfrak X$, i.e. the
source curve is allowed to be a Hassett-stable curve with additional $p$ $\boldsymbol{\varepsilon}-$weighted markings. Thus the
moduli stack $Q^{\theta,\boldsymbol{\varepsilon}}_{g,m|p}(\mathfrak X,\beta)$ is equipped with $p$ additional universal evaluation maps to
$\mathfrak X$ (not only to $X$). We will denote them by

$$\hat{ev}_{j}: Q^{(\theta,\boldsymbol{\varepsilon})}_{g,m|p}(\mathfrak
X,\beta)\rightarrow \mathfrak X,\quad 1\leq j\leq p\ .$$
  
\end{remark}
\subsection{Quasimap invariants}\label{subsec:quasi-inv}
We define the quasimap invariants in this section following
\cite{ciocan-fontanine14_wall_cross_genus_zero_quasim, Ciocan_Fontanine_2014,
  Dan_Abramovich_2008, Cheong_2015}. Consider an algebraic torus $T$ action on $W$,
which commutes with
the given $G-$action on $W$, here $T$ can be the identity group. Assume further
that the $T-$fixed loci $\underline{X} _0^T$ of the affine quotient
$\underline{X}_0=\mathrm{Spec}(\mathbb C[W]^{G})$ is $0-$dimensional. We also denote
$K:=\mathbb Q(\{\lambda_i\})$ by the rational localized $T$-equivariant
cohomology of $\Spec \mathbb C$, with
$\{\lambda_1,\dots,\lambda_{\text{rank}(T)}\}$ corresponding to a basis for the
characters of $T$. Denote
\[ \Lambda _K := K [[\mathrm{Eff}(W, G, \theta ) ]] \]
to be the corresponding Novikov ring. We write $q^\beta$ for
the element corresponding to $\beta$ in $\Lambda _K$ so that $\Lambda _K$ is the
$q$-adic completion. 

Given any two elements $\alpha_{1},\alpha_{2}$ in the $T-$equivariant {\it Chen-Ruan
  cohomology} of $X$,
$$H^*_{\mathrm{CR}, T} (X, \mathbb Q):=H ^*_T(\bar{I}_{\mu}X , \mathbb Q ) \ ,$$
We can define the Poincar\'e
pairing in the {\it non-rigidified }cyclotomic inertia stack $I_\mu X$ of $X$": 
$$\langle \alpha_{1} , \alpha_{2} \rangle _{\mathrm{orb}} := \int _{\sum  _{r\in \mathbb N_{\geq 1} } r^{-1} [\bar{I}_{\mu _r}X]} \alpha_{1} \cdot \iota ^* \alpha_{2} \ .$$
Here $\iota$ is the involution of $\bar{I}_{\mu}X $ obtained from the inversion
automorphisms. Therefore, the diagonal
class $[\Delta _{\bar{I}_{\mu _r}X}] $ obtained via push-forward of the
fundamental class by $(\mathrm{id}, \iota): \bar{I}_{\mu _r}X \rightarrow
\bar{I}_{\mu _r}X \times \bar{I}_{\mu _r}X$ can be written as
$$\sum _{r=1}^{\infty} r[ \Delta _{\bar{I}_{\mu _r}X}]
= \sum _{\alpha} \phi_{\alpha} \otimes \phi^{\alpha} \text{ in }
H^*(\bar{I}_{\mu}X \times \bar{I}_{\mu}X, \mathbb Q), $$ where 
$\{ \phi_{\alpha}\}$ is a basis of $H^*_{\mathrm{CR}, T} (X,
\mathbb Q)$ with $\{\phi^{\alpha}\}$ the dual basis with respect to the Poincar\'e
pairing defined above.

Denote by $\bar{\psi}_i$ the first Chern class of the universal cotangent
line whose fiber at $((C, q_1, ..., q_m), [x])$ is the cotangent space of the
coarse moduli $\underline{C}$ of $C$ at $i$-th marking $\underline{q} _i$. For
non-negative integers $a_i$ and classes $\alpha _i \in H ^*_T (\bar{I}_{\mu}X ,
\mathbb Q)$, $\delta_{j}\in H^{*}(\mathfrak X,\mathbb Q)$, we write
\begin{align*}
  \langle \alpha _1\bar{\psi} ^{a_1}, ..., \alpha _m
  \bar{\psi}^{a_m};\delta_{1},\cdots,\delta_{l} \rangle^{\theta',\boldsymbol{\varepsilon}} _{0, m, \beta} & :=
                                                       \int _{[Q^{\theta',\boldsymbol{\varepsilon}}_{0,m|p}(X, \beta )]^{\mathrm{vir} }}  \prod _i  ev _i ^* (\alpha _i) \bar{\psi} ^{a_i}_i \prod_{j}\hat{ev}_{j}^{*}(\delta_{j})\ .
\end{align*}
When $\varepsilon$ is empty, $\theta'=\epsilon\theta$ for sufficiently large
rational number $\epsilon$, the above formula recovers the usual Gromov-Witten invariants, in which
case, we will write this as 
$$\langle \alpha _1\bar{\psi} ^{a_1}, ... , \alpha _m\bar{\psi}^{a_m}\rangle\ .$$

We will also need the quasimap Chen-Ruan classes
\begin{equation} \label{tilde-ev}(\widetilde {ev _j})_* = \iota
  _*(\boldsymbol{r} _{j}(ev_{j})_*),\end{equation} where $\boldsymbol{r} _{j}$
is the order function of the band of the gerbe structure at the marking $q_{j}$.
Define a class in $H_*^T(\bar{I}_{\mu}X )\cong H^*_T(\bar{I}_{\mu}X )$ by
\begin{align*}
  \langle \alpha _1,..., \alpha _m, - \rangle
  ^{\epsilon}_{0, \beta} &:=(\widetilde{ev_{m+1}})_* \left((\prod ev _i ^* \alpha _i) \cap [Q^{\epsilon}_{0, m}(X, \beta)]^{\mathrm{vir} }\right)\\
                         &=\sum _{\alpha}\phi^{\alpha} \langle\alpha_1,...,\alpha_m,\phi_{\alpha}\rangle^{\epsilon}_{0,m+1,\beta}\ .
\end{align*}

\section{Geometry of complete intersections in toric Deligne-Mumford
  stacks}\label{sec:I-fun}
$\textbf{From now on}$, we will fix a GIT data $(W,G,\theta)$, which represents
a proper toric Deligne-Mumford stack (or toric stack in short)
$X:=[W^{ss}(\theta)/ G]$ as in remark \ref{toricquasimap}. We will also fix a
vector bundle $E$ over $\mathfrak X:=[W/G]$ which is a direct sum of line
bundles $\oplus_{b=1}^{c}L_{\tau_{b}}$ associated to characters
$(\tau_{b})_{b=1}^{c}$ of $G$. Let $s_{b} \in \Gamma(W,W \!\times\! \mathbb
C_{\tau_{b}})^{G}$ be sections such that they cut off an irreducible complete
intersection in $W$ which is smooth in $W^{ss}(\theta)$. Denote by $AY$ to be
the zero loci of the section $s:=\oplus_{b=1}^{c} s_{b}$ and by
$AY^{ss}(\theta)$ (or $AY^{ss}$) be semistable loci, then $(AY,G,\theta)$ will
also be a GIT target. Note $AY^{ss}$ is equal to the intersection of $W^{ss}$
and $AY$. Denote $Y:=[AY^{ss}(\theta)/G]$ to be the corresponding toric stack
complete intersections inside $X$ and denote $\mathfrak Y:=[AY/G]$ to be the
corresponding quotient stack of $Y$.

It's well known the rigidified inertia stacks of $Y$ and $X$ are
$$\bar{I}_{\mu}Y=\sqcup_{g\in G}[AY^{ss}(\theta)^{g}/(G/\langle{g}  \rangle)],\; \bar{I}_{\mu}X=\sqcup_{g\in G}[W^{ss}(\theta)^{g}/(G/\langle{g}  \rangle)]\ .$$

To describe the relationship of rigidified inertia stacks of $X$ and $Y$, we
will need the following lemma:
\begin{lemma}\label{lem:inertiacomp}
  For any torsion element $g\in G$, $AY^{ss}(\theta)^{g}\subset
  W^{ss}(\theta)^{g}$ is a complete intersection with respect to the sections
  $\{s_{b}|b:\tau_{b}(g)=1\}$.
\end{lemma}
\begin{proof}
  For any point $p\in W^{ss}(\theta)^{g}$ such that $s$ vanishes on $p$, we have
  the following short exact sequence of tangent spaces
  $$0\rightarrow T_{p}AY^{ss}(\theta)\rightarrow
  T_{p}W^{ss}(\theta)\rightarrow \oplus_{b=1}^{c}\mathbb C_{\tau_{b}}\rightarrow
  0\ ,$$ which is also exact as representations of the finite group generated by
  $g$. Taking the $g$-invariant subspace of the above exact sequence, we get
    $$0\rightarrow T_{p}AY^{ss}(\theta)^{g}\rightarrow
    T_{p}W^{ss}(\theta)^{g}\rightarrow \oplus_{b:\tau_{b}(g)=1}\mathbb
    C_{\tau_{b}}\rightarrow 0\ ,$$ which imply the lemma.
  \end{proof}

  For any degree $\beta\in \mathrm{Eff}(W,G,\beta)$, we will also need an
  element $g_{\beta}\in G$, and two special sub-varieties $Y^{ss}_{\beta}\subset
  AY^{ss},$ $Z^{ss}_{\beta}\subset W^{ss}$ in the statement of the mirror
  theorem:
$$g_{\beta}:=(\mathrm{e}^{2\pi \sqrt{-1}\beta(L_{\pi_{1}})},\cdots,\mathrm{e}^{2\pi
  \sqrt{-1}\beta(L_{\pi_{k}})})\in G=(\C^{*} )^{k}\ ,$$
$$Y^{ss}_{\beta}:=(AY^{ss})^{g_{\beta}}\cap \{(x_{i})\in
W|x_{i}=0\; \forall i: \beta(L_{\rho_{i}})\in \mathbb Z_{<0} \}\ ,$$ and
$$Z^{ss}_{\beta}:=(W^{ss})^{g_{\beta}}\cap \{(x_{i})\in
W|x_{i}=0\; \forall i: \beta(L_{\rho_{i}})\in \mathbb Z_{<0} \}\ .$$ In the end
of this section, we will prove a lemma \ref{lem:comp-graph} relating the geometry
of $Y^{ss}_{\beta}$ and $Z^{ss}_{\beta}$.

The geometrical significance of introducing $Y^{ss}_{\beta}$ and
$Z^{ss}_{\beta}$ is that the quotient stacks $[Y^{ss}_{\beta}/G]$ and
$[Z^{ss}_{\beta}/G]$ describe important classes in the stacky loop spaces for
$X$ and $Y$ which we now describe.

First of all, let's recall the definition of stacky loop space into the toric
stack $X$~\cite{Cheong_2015}. Set $U=\C^{2}\backslash\{0\}$, for any positive
integer $a$, denote $\P_{a,1}$ to be the quotient stack $[U / \C^{*} ]$ defined
by the $\C^{*}$-action on $U$ with weights $[a,1]$ so that $0:=[0:1]$ is a
non-stacky point and $\infty :=[1:0]\cong \mathbb B\mu_{a}$ is a stacky point.
The stacky loop space into $X$
$$ Q_{\P_{a,1}}(X,\beta)\subset Hom^{rep}_{\beta}(\P_{a,1},\mathfrak X)$$
is defined to be the moduli stack of representable morphisms from $\P_{a,1}$ to
$\mathfrak X$ of degree $\beta$ such that the generic point of $\P_{a,1}$ is
mapped into $X$. By~\cite[Lemma 4.6]{Cheong_2015}, for such representable
morphism to exist, $a$ must be the order of the finite cyclic group generated by
$g_{\beta}$. We note $a$ is also the minimal positive integer making
$a\beta(L_{\tau})$ an integer for all character $\tau$ of $G$. We can define the
stacky loop space into $Y$ in a similar manner, denote
$$ Q_{\P_{a,1}}(Y,\beta)\subset Hom^{rep}_{\beta}(\P_{a,1},\mathfrak Y)$$
by the moduli stack of representable morphisms from $\P_{a,1}$ to $\mathfrak X$
of degree $\beta$ such that the generic point of $\P_{a,1}$ is mapped into $Y$.

Let $a$ be the integer associated to $g_{\beta}$. Let $\mathbb C[z_{1},z_{2}]$
be the polynomial ring on variables $z_{1}$ and $z_{2}$ with weights $a$ and $1$
respectively. Consider the finite dimensional vector space
$$ W_\beta := \bigoplus _{\rho \in [n]} \mathbb C[z_{1}, z_{2}] _{a\beta (L_\rho )} $$
with the $G$-action given by the direct sum of the diagonal $G$-action on
$\mathbb C[z_{1},z_{2}]_{a\beta (L_\rho)}$ by the weight $\rho$, then $\mathbb
C[z_{1},z_{2}]_{a\beta (L_\rho)}\cong \bigoplus \mathbb C _\rho$. Given any
element of $W_{\beta}$, we can naturally associate a morphism from $\P_{a,1}$ to
$\mathfrak X$ of degree $\beta$. Then we have the equivalence of the following
two stacks:
$$ \Hom _{\beta} (\P_{a,1}, \mathfrak X) \cong
\Hom^{\mathrm{rep}}_{\beta}(\P_{a,1}, \mathfrak X) \cong [W_\beta / G] \ , $$
under which correspondence, we have
$$Q_{\P_{a,1}}(X,\beta)\cong [W^{ss}_{\beta}(\theta)/G]\ .$$

Consider the $\C^{*}-$action on $P_{a,1}$ defined by
$$t(\zeta_{1},\zeta_{2})=(t\zeta_{1},\zeta_{2})\ , $$
for all $(\zeta_{1},\zeta_{2})\in U$ and $t\in \C^{*}$. This induces a
$\C^{*}-$action on $Q_{\P_{a,1}}(X,\beta)$ as well as on
$Q_{\P_{a,1}}(Y,\beta)$. Denote $F_{\beta}(X)$(resp. $F_{\beta}(Y)$) to be the
subspace of $Q_{\P_{a,1}}( X,\beta)$(resp. $Q_{\P_{a,1}}(Y,\beta)$) in which the
representable morphism $f:\P_{a,1}\rightarrow \mathfrak X$ (resp. $f:\P_{a,1}
\rightarrow \mathfrak Y$) with all the degree concentrated on the point $[0:1]$.
More explicitly, $F_{\beta}(X)$ (resp. $F_{\beta}(Y)$) is comprised of the
morphisms in the form
$$f:\P_{a,1} \rightarrow \mathfrak X\quad (\mathrm{resp}. \mathfrak Y), \quad (\zeta_{1},\zeta_{2})\mapsto
(a_{\rho}\zeta_{1}^{\beta(L_{\rho})})_{ \rho\in[n]}\ ,$$ where
$(a_{\rho}z_{1}^{\beta(L_{\rho})}:\rho\in[n])\in W^{ss}_{\beta}(\theta)$. Note
for such a map to be well-defined, $a_{\rho}$ must be $0$ when
$\beta(L_{\rho})\notin \Z_{\geq 0}$.

We can see that $F_{\beta}(X)$ is a component of the $\C^{*}-$fixed loci of
$Q_{\P_{a,1}}(X,\beta)$, which we can describe more explicitly as follows.
Define
$$Z_{\beta}:=\bigoplus_{\rho\in[n],\beta(L_{\rho})\in \Z_{\geq 0}} \mathbb C
\cdot z_{1}^{\beta(L_{\rho})} \subset W_{\beta}. $$ We have $Z^{ss}_{\beta}\cong
Z_{\beta}\cap W^{ss}_{\beta}(\theta)$, and
$$ F_{\beta}(X) \cong[Z^{ss}_{\beta}/G],\; \text{and}\; F_{\beta}(Y)\cong [Y^{ss}_{\beta}/G]\ .$$

Let $g_{\beta}$ be as above, then we have another characterization of
$Z_{\beta}$ and $Z^{ss}_{\beta}$:
$$Z_{\beta}\cong W^{g_{\beta}}\cap \bigcap_{\substack{\rho\in[n]\\ \beta(L_{\rho})\in
    \Z_{<0}}}D_{\rho}\quad \text{and} \quad
Z^{ss}_{\beta}\cong W^{ss}(\theta)^{g_{\beta}}\cap \bigcap_{\substack{\rho\in[n]\\
    \beta(L_{\rho})\in \Z_{<0}}}D_{\rho}\ .$$ Here $D_{\rho}$ is the divisor of
$W$ given by $x_{\rho}=0$.

It's clear that $Y^{ss}_{\beta}$ is the vanishing loci of the sections
$\{s_{b}|b:\beta(L_{\tau_{b}})\in \Z\}$ on $Z^{ss}_{\beta}$, but this may not be
a complete intersection. Indeed, one can show the following.
\begin{lemma}\label{lem:comp-graph}
  For any $b$ such that $\beta(L_{\tau_{b}})\in \mathbb Z_{<0}$, the section
  $s_{b}$ vanishes on $Z^{ss}_{\beta}$. Thus $Y^{ss}_{\beta}$ is merely the
  vanishing loci of sections $\{s_{b}|b:\beta(L_{\tau_{b}})\in \Z_{\geq 0}\}$ in
  $Z^{ss}_{\beta}$.
\end{lemma}
\begin{proof}
  For $b$ with $\beta(L_{\tau_{b}})\in \Z_{\leq 0}$, for any point
  $\vec{x}=(a_{\rho})_{ \rho\in [n]}\in Z^{ss}_{\beta}$, the corresponding
  morphism in $F_{\beta}(X)$ is in the form
  $$[\vec{x}]:\P_{a,1}\rightarrow \mathfrak X: [\zeta_{1},\zeta_{2}]\rightarrow (a_{\rho}\zeta_{1}^{\beta(L_{\rho})})_{\rho\in[n]}\ .$$
  Then the pull back of section $s_{b}$ to $\P_{a,1}$ becomes
  $s_{b}(\vec{x})z_{1}^{\beta(L_{\tau_{b}})}$. However as the pull-back line
  bundle $[\vec{x}]^{*}L_{\tau_{b}}$ is of degree $\beta(L_{\tau_{b}})<0$ on
  $\P_{a,1}$, hence there is no nonzero section in the line bundle $[\vec{x}]^{*}L$,
  which implies that $s_{b}(\vec{x})=0$. Now the lemma follows.
\end{proof}

Denote $E_{\geq 0}:=\oplus_{b:\beta(L_{\tau_{b}})\in \mathbb Z_{\geq
    0}}L_{\tau_{b}}$ and $s_{\geq 0}=\oplus_{b:\beta(L_{\tau_{b}})\in \mathbb
  Z_{\geq 0}} s_{b}$. Using the above lemma, we can define the Gysin morphism
$$s^{!}_{E_{\geq 0},loc}:A_{*}([Z^{ss}_{\beta}/(G/\langle g^{-1}_{\beta}\rangle)])\rightarrow
A_{*}([Y^{ss}_{\beta}/(G/\langle g^{-1}_{\beta}\rangle)])$$
as the localized top Chern class~\cite[\S14.1]{MR732620} with respect to the vector
bundle $E_{\geq 0}$ over $[Z^{ss}_{\beta}/(G/\langle
g^{-1}_{\beta}\rangle)]$ and the section $s_{\geq 0}$. Let
$i:\bar{I}_{\mu}Y\rightarrow \bar{I}_{\mu}X$ be the natural inclusion. Now we
discuss two implications of the above lemma:
\begin{corollary}\label{cor:com-graph}
  \begin{enumerate}
  \item If the set $\{b\;|\;\beta(L_{\tau_{b}})\in \mathbb Z\}$ is exactly the set
    $\{b\; |\; \beta(L_{\tau_{b}})\in \mathbb Z_{\geq 0}\}$, then we have
    $$i_{*}(s^{!}_{E_{\geq
        0},\mathrm{loc}}([Z^{ss}_{\beta}/(G/\langle
    g^{-1}_{\beta}\rangle)]))=(\prod_{\rho:\beta(L_{\rho})\in \mathbb Z_{\leq
        0}}D_{\rho})\cdot \mathbbm 1_{g^{-1}_{\beta}}$$
    in $A^{*}(\bar{I}_{g^{-1}_{\beta}}Y)$, where $\mathbbm 1_{g^{-1}_{\beta}}$ is the fundamental class of
    $\bar{I}_{g^{-1}_{\beta}}Y$, $D_{\rho}$ is the class of the hyperplane given
    by $x_{\rho}=0$.
  \item If the set $\{b\; |\; \beta(L_{\tau_{b}})\in \mathbb Z\}$ is empty, then we have
    $Y^{ss}_{\beta}=Z^{ss}_{\beta}$, and
    $\bar{I}_{g^{-1}_{\beta}}Y=\bar{I}_{g^{-1}_{\beta}}X$, and $s^{!}_{E_{\geq
        0},\mathrm{loc}}$ is the identity morphism. Thus
    $$i_{*}(s^{!}_{E_{\geq
        0},\mathrm{loc}}([Z^{ss}_{\beta}/(G/\langle
    g^{-1}_{\beta}\rangle)]))=(\prod_{\rho:\beta(L_{\rho})\in \mathbb Z_{\leq
        0}}D_{\rho})\cdot \mathbbm 1_{g^{-1}_{\beta}}$$ 
   in $A^{*}(\bar{I}_{g^{-1}_{\beta}}Y)$. 
  \end{enumerate}
\end{corollary}

\subsection{Two special cases of the mirror theorem}\label{subsec:special-case}
Using the above corollary, we get two
interesting special cases of the $I-$function. The first case is when $Y$ is a hypersurface with respect to a line bundle $L:=L_{\tau}$ for some
character $\tau$, the mirror formula \eqref{eq:mirror} becomes:
\begin{equation}
\begin{split}
  \mathbb I(q,t,z)&=\mathrm{exp}\big(\frac{1}{z}\sum_{i=1}^{l}t_{i}u_{i}(c_{1}(L_{\pi_j})+\beta(L_{\pi_j})z)\big)\\
  &\bigg(\sum_{\substack{\beta\in \operatorname{Eff(W,G,\theta)}\\ \beta(L)\geq 0}}q^{\beta}\frac{\prod_{\rho:\beta(L_{\rho})<
             0}\prod_{\beta(L_{\rho})\leq i<0}(D_{\rho}+(\beta(L_{\rho})-i)z)}{\prod_{\rho:\beta(L_{\rho})>
             0}\prod_{0\leq i<
             \beta(L_{\rho})}(D_{\rho}+(\beta(L_{\rho})-i)z)}\\
           & \times \prod_{0\leq i<
             \beta(L) }\big(c_1(L)+(\beta(L)-i)z\big)\mathbbm{1}_{g^{-1}_{\beta}}\\
&+\sum_{\substack{\beta\in \operatorname{Eff(W,G,\theta)}\\ \beta(L)\in \Z_{<0}}}q^{\beta}\frac{\prod_{\rho:\beta(L_{\rho})<
             0}\prod_{\beta(L_{\rho})< i<0}(D_{\rho}+(\beta(L_{\rho})-i)z)}{\prod_{\rho:\beta(L_{\rho})>
             0}\prod_{0\leq i<
             \beta(L_{\rho})}(D_{\rho}+(\beta(L_{\rho})-i)z)}\\
         & \times \prod_{\beta(L)< i<0}\frac{1}{\big(c_1(L)+(\beta(L)-i)z\big)}\big[[Y^{ss}_{\beta}/(G/\langle{g^{-1}_{\beta}}  \rangle)]\big]\\
&+\sum_{\substack{\beta\in \operatorname{Eff(W,G,\theta)}\\ \beta(L)\in \mathbb Q_{<0}\backslash\Z_{<0}}}q^{\beta}\frac{\prod_{\rho:\beta(L_{\rho})<
             0}\prod_{\beta(L_{\rho})\leq i<0}(D_{\rho}+(\beta(L_{\rho})-i)z)}{\prod_{\rho:\beta(L_{\rho})>
             0}\prod_{0\leq i<
             \beta(L_{\rho})}(D_{\rho}+(\beta(L_{\rho})-i)z)}\\
         & \times \prod_{\beta(L)< i<0}\frac{1}{\big(c_1(L)+(\beta(L)-i)z\big)}\mathbbm 1_{g_{\beta}^{-1}}\bigg)  \ .
\end{split}
\end{equation}

The second case is when all the line bundles $L_{\tau_{b}}$ are all semi-positive, the $I-$function specializes to:
 \begin{equation}
\begin{split} 
\mathbb I(q,t,z)&=\mathrm{exp}\big(\frac{1}{z}\sum_{i=1}^{l}t_{i}u_{i}(c_{1}(L_{\pi_j})+\beta(L_{\pi_j})z)\big)\\
  &\sum_{\beta\in \operatorname{Eff(W,G,\theta)}}q^{\beta}\frac{\prod_{\rho:\beta(L_{\rho})<
             0}\prod_{\beta(L_{\rho})\leq i<0}(D_{\rho}+(\beta(L_{\rho})-i)z)}{\prod_{\rho:\beta(L_{\rho})>
             0}\prod_{0\leq i<
             \beta(L_{\rho})}(D_{\rho}+(\beta(L_{\rho})-i)z)}\\
  &\times \prod_{b} \prod_{0\leq i<
             \beta(L_{\tau_{b}}) }\big(c_1(L_{\tau_{b}})+(\beta(L_{\tau_{b}})-i)z\big)\mathbbm{1}_{g^{-1}_{\beta}}     
\end{split}
\end{equation}

The above formula matches the formula for positive hypersurfaces in
toric stacks for which the convexity holds~\cite[\S 5]{coates14_some_applic_mirror_theor_toric_stack} and the formula for
a possibly ray divisor (given by a coordinate function corresponding
to the ray) of a toric stack for which the convexity may fail~\cite{Coates_2015,
  Cheong_2015}. See \S\ref{sec:corti-example} for a non-positive example where the convexity fails.

\section{Master space I}\label{sec:twistedgraphspace}
\subsection{Construction of master space I}\label{subsec:Construction of the
  master space}
In this section, we will construct a master space which is a root stack
modification of the twisted graph space considered
in~\cite{clader17_higher_genus_quasim_wall_cross_via_local}. Let $(AY,G,\theta)$
be the GIT data which gives rise to a hypersurface in the toric stack
$X=[W^{ss}(\theta)/ G]$ as in previous sections. \emph{Since a positive scaling
  of the stability character $\theta$ will not change the GIT quotient. Without loss of generality,
let's assume that the line bundle $L_{\theta}$ on $Y=[AY^{ss}(\theta)/G]$
is the pullback of a positive line bundle on the coarse moduli space
$\underline{Y}$ of $Y$.} First we will consider the following quotient stack
$$\P \mathfrak Y^{\frac{1}{r},p}=[(AY \!\times\! \C^{p} \!\times\!\C^2)/(G
\!\times\! (\C^{*})^{p}\!\times\!\C^{*})]$$
defined by the following (right) action
$$(\vec{x},\vec{y},z_{1},z_{2})\cdot (g,h,t)=(\vec{x}\cdot g,(h_{j}y_{j})_{j=1}^{p},\theta(g)^{-1}(\prod_{j=1}^{p}h_{j}^{-1})t^{r}z_{1},tz_{2})\ ,$$
where $(g,h=(h_{j})_{j=1}^{p},t)\in G \!\times\! (\C^{*})^{p} \!\times\!\C^{*}$
$(\vec{x},\vec{y}=(y_{j})_{j=1}^{p},z_{1},z_{2})\in AY \!\times\! \C^{p}\!\times\!
\C^{2}$. For simplicity, we will write $AY_{p}:=AY \!\times\! \C^{p}$, and
$G_{p}:=G \!\times\! (\C^{*})^{p}$ and $\theta_{p}$ as the character of $G_{p}$
defined by
$$\theta_{p}(g,h)=\theta(g)\prod_{j=1}^{p}h_{j}\; \text{for all}\; (g,h)\in G_{p}\ .$$ 


Fix a positive rational number $\epsilon\in \mathbb Q_{>0}\cap(0,1]$ and a tuple of
positive rational numbers $\boldsymbol{\varepsilon}=(\epsilon,\cdots,\epsilon)\in (\mathbb Q_{>0})^{p}$, we consider the
stability given by the rational character of $G_{p} \!\times\! \C^{*}$ defined by
$$\widetilde{\theta}(g,h,t)=\theta(g)^{\epsilon}(\prod_{i=1}^{p}h^{\epsilon}_{i})t^{3r}$$
for $(g,h,t)\in G_{p}\!\times\! \C^{*}$. Then the GIT stack quotient $[(AY_{p} \!\times\!
\C^{2})^{ss}(\widetilde{\theta})/(G_{p} \!\times\! \C^{*})]$ is the root stack of
the $\P^{1}-$bundle $\P_{Y}(\mathcal O(-D_{\theta})\oplus \mathcal O)$ over $Y$
by taking $r$-th root of the infinity divisor $D_{\infty}$ given by $z_{2}=0$.
We will denote the GIT stack quotient $[(AY_{p} \!\times\!
\C^2)^{ss}(\widetilde{\theta})/ (G_{p} \!\times\! \C^{*})]$ to be $\P
Y^{\frac{1}{r}}$, which is equipped with the infinity section $\mathcal
D_{\infty}$ given by $z_{2}=0$ and the zero section $\mathcal D_{0}$ given by
$z_{1}=0$. Note this GIT quotient is independent of the integer $p$ as the
semistable(=stable) loci $(AY_{p} \!\times\!
\C^2)^{ss}(\widetilde{\theta})=AY^{ss}(\theta) \!\times\! (\C^{*})^{p}
\!\times\! (\C^{2}\backslash \{0\})$. We will take $p=1$ as our standard GIT
reference for $\mathbb PY^{\frac{1}{r}}$, which will be canonically identified
with other GIT quotients from $\P \mathfrak Y^{\frac{1}{r},p}$ by choosing the
embedding $AY\subset AY_{p}$ as $AY\cong AY_{p}\cap \{y_{i}=1| i=1,\ldots,p \}$.

The inertia stack $I_{\mu} \P Y^{\frac{1}{r}} $ of $\P Y^{\frac{1}{r}}$ admits a
decomposition
$$I_{\mu} \P Y \bigsqcup \sqcup_{j=1}^{r-1}
\sqrt[r]{L_{\theta}/Y}\ .$$

Let $(\vec{x},(g,t))$ be a point of $I_{\mu}\P Y^{\frac{1}{r}}$, if
$(\vec{x},(g,t))$ appears in the first factor of the decomposition above, then
the automorphism $(g,t)$ lies in $G \!\times\! \{1\}$; if $(\vec{x},(g,t))$
occurs in the second factor of the decomposition above, the automorphism $(g,t)$
lies in $G \!\times\! \{\mu_{r}^{j}: 1\leq j\leq r-1\}\subset G \!\times\!
\boldsymbol{\mu}_{r}$, and the point $\vec{x}$ is in the infinity section
$\mathcal D_{\infty}$ defined by $z_{2}=0$. Here
$\mu_{r}=\mathrm{exp}(\frac{2\pi \sqrt{-1}}{r})\in \C^{*}$ and
$\boldsymbol{\mu}_{r}$ is the cyclic group generated by $\mu_{r}$.

For $(g,t)\in G \!\times\! \boldsymbol{\mu}_{r}$. we will use the notation
$\bar{I}_{(g,t)}\P Y^{\frac{1}{r}}$ to mean the rigidified inertia stack
component of $\bar{I}_{\mu}\P Y^{\frac{1}{r}} $ which has automorphism $(g,t)$.

Consider the moduli stack of $\widetilde{\theta}-$stable quasimaps to $\P
\mathfrak Y^{\frac{1}{r},p}$:
$$Q^{\widetilde{\theta}}_{0,m}(\P \mathfrak Y^{\frac{1}{r},p},(d,1^p,\frac{\delta}{r}))\ .$$

More concretely,
\[Q^{\widetilde{\theta}}_{0,m}(\P \mathfrak
  Y^{\frac{1}{r},p},(d,1^{p},\frac{\delta}{r})) = \{(C;q_1, \ldots, q_m;
  L_1,\cdots,L_{k+p}, N; \vec{x},\vec{y}, z_1, z_2)\},\] where $(C;
q_1, \ldots, q_m)$ is a $m$-pointed prestable balanced orbifold curve of genus
$0$ with possible nontrivial isotropy only at special points, i.e. marked gerbes
or nodes, the line bundles $(L_j:1\leq j\leq k+p)$ and $N$ are orbifold line
bundles on $C$ with
\begin{equation}
  \label{eq:TGSdegrees}
  \mathrm{deg}([\vec{x}])=d\in Hom(Pic(\mathfrak Y),\mathbb Q), \; \; \;  \mathrm{deg}(N)=\frac{\delta}{r}\ ,
\end{equation}
\begin{equation}
\mathrm{deg}(L_{k+j})=1,\; \; 1 \leq j\leq p\ , 
\end{equation}
and
\[(\vec{x},\vec{y},\vec{z}) := (x_1, \ldots, x_n,y_{1},\ldots,y_{p}, z_1, z_2) \in
  \Gamma\left(\bigoplus_{i=1}^{n} L_{\rho_{i}}\oplus \bigoplus_{j=1}^{p}L_{k+j} \oplus (L_{-\theta_{p}}\otimes
    N^{\otimes r}) \oplus N\right).\] Here, for $1\leq i\leq n$, the line bundle
$L_{\rho_{i}}$ is equal to
$$\otimes_{j=1}^{k}L_{j}^{m_{ij}}\ ,$$
where $(m_{ij})$ ($1\leq i\leq n$,$1\leq j\leq k+p$) is given by the relation
$\rho_{i}=\sum _{j=1}^{k}m_{ij}\pi_{j}$. The same construction applies to the
line bundle $L_{-\theta_{p}}$ on $C$. Note here $\delta$ is an integer when
$Q^{\widetilde{\theta}}_{0,m}(\P \mathfrak
Y^{\frac{1}{r},p},(d,1^{p},\frac{\delta}{r}))$ is nonempty as $N^{\otimes r}$ is the
pull-back of some line bundle on the coarse moduli curve $\underline{C}$.

We require the the following conditions are satisfied for the above data:
\begin{itemize}
\item {\it Representability}: For every $q \in C$ with isotropy group $G_q$, the
  homomorphism $\mathbb BG_q \rightarrow \mathbb B(G_{p}\times \C^{*})$ induced by
  the restriction of line bundles $(L_{j}:1\leq j\leq k+p)$ and $N$ to $q$ is
  representable, Note the image of the homomorphism lies in the subgroup $G\!\times\! \C^{*}\subset
  G_{p}\!\times\! \C^{*}$.
\item {\it Nondegeneracy}: The sections $z_1$ and $z_2$ never simultaneously
  vanish. Furthermore, for each point $q$ of $C$ at which $z_2(q) \neq 0$, the
  stability condition \ref{rmk:theta-quasimap}
  $$l_{\widetilde{\theta}}(q)\leq 1$$
  for $\widetilde{\theta}$-stable map to $\P \mathfrak Y^{\frac{1}{r},p}$
  becomes the stability condition
  \begin{equation}
    \label{eq:stab1}
    l_{\epsilon\theta\bigoplus \oplus_{j=1}^{p}\epsilon\mathrm{id}_{\C^{*}}}(q)\leq 1,
  \end{equation}
  for the prestable quasimap $[\vec{x},\vec{y}]:C\rightarrow \mathfrak
  Y\!\times\! [\C/\C^{*}]^{p}$. For each
  point $q$ of $C$ at which $z_2(q) = 0$, we have
  \begin{equation}
    \label{eq:stab2}
    \text{ord}_q(\vec{x})=\text{ord}_{q}(\vec{y}) = 0.
  \end{equation}
  We note that this can be phrased as the length condition
  \eqref{eq:quasi-stablity} bounding the order of contact of $(\vec{x},\vec{y}, \vec{z})$ with the unstable loci of $\P \mathfrak Y^{\frac{1}{r},p}$ as in
  \cite[\S 2.1]{ciocan-fontanine2016}.
\item {\it Stability}: The $\mathbb Q-$line bundle
  $$(\phi_{*}(L_{\theta}))^{\otimes \epsilon}\otimes \bigotimes_{j=1}^{p}\phi_{*}(L_{k+j})^{\otimes \epsilon}\otimes \phi_{*}(N^{\otimes 3r})\otimes \omega_{\underline{C}}^{log}$$
  on the coarse curve $\underline{C}$ is ample. Here $\phi:C\rightarrow
  \underline{C}$ is the coarse moduli map. Note here, the line bundles
  $L_{\theta}$, $(L_{k+j})_{j=1}^{p}$ and $N^{\otimes 3r}$ are the pull back of line bundles on the coarse moduli of $\underline{C}$.
\item{\it Vanishing}: The image of $[\vec{x}]:C\rightarrow \mathfrak X$ lies in
  $\mathfrak Y$.
\end{itemize}
Let $\vec{m}=(v_{1},\cdots,v_{m})\in (G \!\times\! \boldsymbol{\mu}_{r})^{m}$,
we will denote $Q^{\widetilde{\theta}}_{0,\vec{m}}(\P \mathfrak
Y^{\frac{1}{r},p},(d,1^p,\frac{\delta}{r}))$ to be:
$$Q^{\widetilde{\theta}}_{0,m}(\P \mathfrak
Y^{\frac{1}{r},p},(d,1^p,\frac{\delta}{r}))\cap ev_{1}^{-1}(\bar{I}_{v_{1}}\P
Y^{\frac{1}{r}})\cap \cdots \cap ev_{m}^{-1}(\bar{I}_{v_{m}}\P Y^{\frac{1}{r}} )
\ ,$$
where
\begin{equation*}
  ev_i: Q^{\widetilde{\theta}}_{0,\vec{m}}(\P \mathfrak Y^{\frac{1}{r},p},(\beta,1^p,\frac{\delta}{r}))\rightarrow \bar{I}_{\mu}\P Y^{\frac{1}{r}}
\end{equation*}
are natural evaluation maps as before, by evaluating the sections
$(\vec{x},\vec{z})$ at $i$th marking $q_i$. Using the vanishing loci of the section $y_{j}$ of the
degree one line bundle $L_{k+j}$ for $1\leq j\leq p$, which corresponds to a
smooth non-orbifold point on $C$, one has another tuple of
evaluation maps 
\begin{equation}
  \label{eq:hatev}
  \hat{ev}_{j}: Q^{\widetilde{\theta}}_{0,\vec{m}}(\P \mathfrak
Y^{\frac{1}{r},p},(\beta,1^p,\frac{\delta}{r}))\rightarrow
\mathfrak Y  \ ,
\end{equation}
for $1\leq j\leq p$.

Because $Q^{\widetilde{\theta}}_{0,\vec{m}}(\P \mathfrak
Y^{\frac{1}{r},p},(\beta,1^{p},\frac{\delta}{r}))$ is a moduli space of stable quasimaps to
a proper lci GIT quotient, it is
a proper Deligne-Mumford stack equipped with a natural perfect obstruction
theory relative to the stack $\mathfrak M^{tw}_{0,m}$ of
twisted curves by~\cite{Ciocan_Fontanine_2014}. This obstruction theory has the form
\begin{equation}
  \label{pot}
  \mathbb E:=R^{\bullet}\pi_{*}(f^*\mathbb{T}_{\P \mathfrak Y^{\frac{1}{r},p}})\ .
\end{equation}
Here, we denote the universal family over $Q^{\widetilde{\theta}}_{0,\vec{m}}(\P
\mathfrak Y^{\frac{1}{r},p},(\beta,1^{p},\frac{\delta}{r}))$ by
\[\xymatrix{
    \mathcal{C} \ar^{\pi}[d]\ar^{f}[r] &\P \mathfrak Y^{\frac{1}{r},p}\\
     Q^{\widetilde{\theta}}_{0,\vec{m}}(\P \mathfrak
     Y^{\frac{1}{r},p},(\beta,1^{p},\frac{\delta}{r}))\ . }\]
The obstruction theory \eqref{pot} can written as cone of the morphism of complexes
\begin{equation}\label{eq:fac-pot}
R^{\bullet}\pi_{*}(\mathcal O_{\mathcal C}\otimes \mathfrak g_{r,p})\rightarrow
R^{\bullet}\pi_{*}\big(\mathcal V\oplus(\oplus_{j=1}^{p}\mathcal L_{k+j})\oplus (\mathcal L_{-\theta_{p}}\otimes\mathcal N^{\otimes r})\oplus \mathcal N\big)\ .
\end{equation}
induced from the distinguished triangle of the tangent complex $\mathbb T_{\P
  \mathfrak Y^{\frac{1}{r},p}}$ of $\P \mathfrak Y^{\frac{1}{r},p}$ 
$$\mathfrak g_{r,p}\times_{G}\mathcal O_{W_{r,p}}\rightarrow
W_{r,p}\times_{G}\mathbb T_{W_{r,p}}\rightarrow
\mathbb T_{\P \mathfrak Y^{\frac{1}{r},p}}\ .$$
Here we use the GIT representation $\P \mathfrak
Y^{\frac{1}{r},p}=[W^{r,p}/G^{r,p}]$ as constructed above. Here $\mathcal L_{\rho_{i}}$ ($1\leq i\leq n$), $\mathcal
L_{j}$ ($1\leq j\leq k$) and $\mathcal N$ are the
universal line bundles and
\begin{equation*}
  \mathcal{V} \subset \oplus_{i=1}^{n}\mathcal L_{\rho_{i}} 
\end{equation*}
is the subsheaf of sections taking values in the affine cone of $Y$. Somewhat
more explicitly, the sub-obstruction-theory
$\mathbb{E}_{\mathrm{rel}}:=R^{\bullet}\pi_{*}(\mathcal V)$ comes from the deformations and
obstructions of the sections $\vec{x}$, and $\mathbb E_{\mathrm{rel}}$ fits into the following
distinguished triangle:
\begin{equation}\label{eq:rel-pot}
  \xymatrix{
\mathbb E_{\mathrm{rel}} \ar[r] &R^{\bullet}\pi_{*}(\oplus_{i=1}^{n}\mathcal
L_{\rho_{i}})\ar[r]^{ds} &R^{\bullet}\pi_{*}(\oplus_{b=1}^{c} \mathcal
L_{\tau_{b}})\ar[r] &\ . }
\end{equation}
We note we can interpret $R^{\bullet}\pi_{*}(\mathcal O_{\mathcal C}\otimes
\mathfrak g_{r,p})$ as the deformation theory of line bundles
$(L_{j})_{j=1}^{k+p}$ and $N$, and interpret
$R^{\bullet}\pi_{*}\big(\oplus_{j=1}^{p}\mathcal L_{k+j}\oplus (\mathcal
L_{-\theta_{p}} \otimes\mathcal N^{\otimes r})\oplus \mathcal N\big)$ as the
deformation theory of sections $\vec{y}$ and $z_{1},z_{2}$.
\subsection{$\C^{*}$-action and fixed loci}\label{subsec:fixed-loci1}
Consider the (left) $\C^{*}$-action on $AY_{p} \!\times\! \C^{2}$ defined by:
$$\lambda(\vec{x},\vec{y},z_{1},z_{2})=(\vec{x},\vec{y},\lambda z_{1},z_{2})\ ,$$
this action descents to be an action on $\P \mathfrak Y^{\frac{1}{r},p}$. We
will denote $\lambda$ to be the equivariant class corresponding to the
$\C^{*}$-action of weight 1. Let's first state a criteria for a morphism to $\P
\mathfrak Y^{\frac{1}{r},p}$ to be $\C^{*}$-equivariant (see also~\cite[\S
2.2]{chang16_effec_theor_gw_fjrw_invar}), which will be important in the
analysis of localization computations.

\begin{remark}{(\it Equivariant morphism to $\P \mathfrak Y^{\frac{1}{r},p}$)}\label{rmk:equivmap}
  Fix a stack $S$ over $Spec(\mathbb C)$ with a left $\C^{*}$-action, then a
  $\C^{*}$-equivariant morphism from $S$ to $\P \mathfrak Y^{\frac{1}{r},p}$ is
  equivalent to the following data: there exists $k+p+1$ $\C^{*}$-equivariant
  line bundles on $S$
  $$L_{1},\cdots, L_{k+p},N$$
  together with $\C^{*}$-invariant sections
\begin{align*}
  (\vec{x},\vec{y}, \vec{z}) &:= (x_1, \ldots, x_n,y_{n+1},\ldots,y_{n+p}, z_1,
                               z_2)\\
  &\in \Gamma\left(\oplus_{i=1}^{n} L_{\rho_{i}}\oplus(\oplus_{j=1}^{p}L_{k+j}) \oplus (L_{-\theta_{p}}\otimes N^{\otimes
  r}\otimes \C_{\lambda}) \oplus N\right)^{\C^{*}}.
\end{align*}
 Here $L_{\rho_{i}}$ $(1\leq i\leq n)$ and $L_{-\theta_{p}}$ are constructed from $(L_{j})_{1\leq j\leq
    k+p}$ as explained before, $\mathbb C_{\lambda}$ is the trivial line bundle
  over $S$ with $\C^{*}-$linearization of weight 1. These sections should also
  satisfy the vanishing condition imposed by the cone of $Y$ as above.
\end{remark}

Fix a nonzero degree $\beta\in \mathrm{Eff}(W,G,\theta)$ and a tuple of
nonnegative integers $(\delta_{1},\cdots,\delta_{m})\in \mathbb N^{m}$. Consider
the tuple of multiplicities $\vec{m}=(v_{1},\cdots,v_{m})\in (G\!\times\!
\boldsymbol{\mu}_{r})^{m}$, where $v_{i}=(g_{i},\mu_{r}^{\delta_{i}})$, we will
denote $Q^{\widetilde{\theta}}_{0,\vec{m}}(\P \mathfrak
Y^{\frac{1}{r},p},(\beta,\frac{\delta}{r}))$ to be
$$\bigsqcup_{\substack{d\in \mathrm{Eff}(AY,G,\theta)\\ (i_{\mathfrak Y})_{*}(d)=\beta }}Q^{\widetilde{\theta}}_{0,\vec{m}}(\P \mathfrak
Y^{\frac{1}{r},p},(d,1^{p},\frac{\delta}{r}))\ ,$$ where $i_{\mathfrak Y}:
\mathfrak Y\rightarrow \mathfrak X$ is the inclusion morphism. Thus
$Q^{\widetilde{\theta}}_{0,\vec{m}}(\P \mathfrak
Y^{\frac{1}{r},p},(\beta,1^{p},\frac{\delta}{r}))$ inherits a $\C^{*}$-action as
above.

Follow the presentation of~\cite{clader17_higher_genus_wall_cross_gauged,
  clader17_higher_genus_quasim_wall_cross_via_local}, we can index the
components of $\C^{*}-$fixed loci of $Q^{\widetilde{\theta}}_{0,\vec{m}}(\P
\mathfrak Y^{\frac{1}{r}},(\beta,1^{p},\frac{\delta}{r}))$ by decorated graphs. A
decorated graph $\Gamma$ consists of vertices, edges, and $m$ legs, and we
decorate it as follows:
\begin{itemize}
\item Each vertex $v$ is associated with an index $j(v) \in \{0, \infty\}$, a
  degree $\beta(v) \in \mathrm{Eff}(W,G,\theta)$ and a subset $J_{v}\in \{1,\cdots,p\}$.
\item Each edge $e=\{h,h'\}$ is equipped with a degree $\beta(e)\in
  \mathrm{Eff}(W,G,\theta)$, a subset $J_{e}\subset \{1,\cdots,p\}$ and $\delta(e) \in \mathbb{N}$.
\item Each half-edge $h$ and each leg $l$ has an element (called multiplicity)
  $m(h)$ or $m(l)$ in $G \!\times\! \boldsymbol{\mu}_{r}$.
\item The legs are labeled with the numbers $\{1, \ldots, m\}$.
\end{itemize}
By the ``valence" of a vertex $v$, denoted $\text{val}(v)$, we mean the total
number of incident half-edges, including legs.

The fixed locus in $Q^{\widetilde{\theta}}_{0,\vec{m}}(\P \mathfrak
Y^{\frac{1}{r},p},(\beta,1^{p},\frac{\delta}{r}))$ indexed by the decorated graph
$\Gamma$ parameterizes quasimaps of the following type:
\begin{itemize}
\item Each edge $e$ corresponds to a genus-zero component $C_e$ on which
  $\deg(N|_{C_{e}}) = \frac{\delta(e)}{r}$ ($\delta(e)>0$), 
  $deg(L_{j}|_{C_{e}})=\beta(e)(L_{\pi_{j}})$ ($1\leq j\leq k$), and
  $deg(L_{k+j}|_{C_{e}})=1$ if and only if $j\in J_{e}$. We
  denote $1^{J_{e}}$ to be the degree coming from the lines bundles
  $(L_{k+j}:1\leq j\leq p)$. There are two distinguished
  points $q_{0}$ and $q_{\infty}$ on $C_{e}$ such that $q_{\infty}$ is the only
  point on $C_{e}$ at which $z_{2}$ vanishes, and $q_{0}$ is the only point on
  $C_{e}$ determined by the following conditions:
  \begin{itemize}
  \item if $C_{e}$ has base points, $q_{0}$ is the only base point on $C_{e}$;
  \item if $C_{e}$ does not have base points on it, $q_{0}$ is the only point on
    $C_{e}$ at which $z_{1}$ vanishes.
  \end{itemize}
  We call them the $\text{ramification points}$\footnote{The definition of the
    ramification point here is different from the definition in \cite[Page
    13]{clader17_higher_genus_quasim_wall_cross_via_local}, where they claim
    that $z_{1}$ or $z_{2}$ each vanish at exactly one point on $C_{e}$. We find
    that there is a missing case when $q_{0}$ is a base point and
    $deg(L_1|_{C_{e}})=deg{L_{2}|_{C_{e}}}=\delta(e)$ in their setting, then
    $z_{1}|_{C_{e}}\equiv 1$, which does not vanish anywhere on $C_{e}$. But the
    author finds this missing case does not affect their main result
    in~\cite{clader17_higher_genus_quasim_wall_cross_via_local}.},
  and all of degree $(\beta(e),1^{J_{e}})$ is concentrated at the ramification point
  $q_{0}$. That is,
  \[\text{when} \; x_{i}|_{C_{e}}\neq
    0,\; \text{we have}\;
    \text{ord}_{q_{0}}(x_{i})=\beta(e)(L_{\rho_{i}}),\; \text{and}\;
    \text{ord}_{q_{0}}(y_{i})=1, \text{only when}\; j\in J_{e}\
    .\]
\item Each vertex $v$ for which $j(v) = 0$ (with unstable exceptional cases
  noted below) corresponds to a maximal sub-curve $C_v$ of $C$ over which $z_1
  \equiv 0$, and each vertex $v$ for which $j(v) = \infty$ (again with unstable
  exceptions) corresponds to a maximal sub-curve over which $z_2 \equiv 0$. The
  label $\beta(v)$ denotes the degree coming from the restriction map
  $[\vec{x}]|_{C_{v}}$, note here we count the degree $\beta(v)$ in
  $\mathrm{Eff}(W,G,\theta)$, but not in $\mathrm{Eff}(AY,G,\theta)$. The subset
  $J_{v}$ is equal to the set where $deg(L_{k+j}|_{C_{v}})=1$ for $1\leq j\leq
  p$. We
  denote $1^{J_{e}}$ to be the degree coming from the lines bundles $(L_{k+j}|_{C_{v}})_{j=1}^{p}$.
\item A vertex $v$ is {\it unstable} if stable quasimap of the type described
  above do not exist (where, as always, we interpret legs as marked points and
  half-edges as half-nodes). In this case, $v$ corresponds to a single point of
  the component $C_e$ for each adjacent edge $e$, which may be a node at which
  $C_e$ meets $C_{e'}$, a marked point of $C_e$, an unmarked point, or a
  basepoint on $C_e$ of order $\beta(v)$, note the base point only appears as a
  vertex over $0$ due to the nondegeneracy condition.
\item The index $m(l)$ on a leg $l$ indicates the rigidified inertia stack
  component $\bar{I}_{m(l)}\P Y^{\frac{1}{r}}$ of $\P Y^{\frac{1}{r}}$ on which
  the marked point corresponding to the leg $l$ is evaluated, this is determined
  by the multiplicity of $L_{1},\cdots, L_{k}, N$ at the corresponding marked
  point.
\item A half-edge $h$ incident to a vertex $v$ corresponds to a node at which
  components $C_e$ and $C_v$ meet, and $m(h)$ indicates the rigidified inertia
  component $\bar{I}_{m(h)}\P Y^{\frac{1}{r}}$ of $\P Y^{\frac{1}{r}}$ on which
  the node on $C_{v}$ corresponding to $h$ is evaluated. If $v$ is unstable,
  then $h$ corresponds to a single point on a component $C_e$, then $m(h)$ is
  the {\it inverse} in $G\times \boldsymbol{\mu}_{r}$ of the multiplicity of
  $L_1,\cdots, L_{k},N$ at this point.
\end{itemize}
In particular, we note that the decorations at each stable vertex $v$ yield a
tuple
$$\vec{m}(v) \in (G\times \boldsymbol{\mu}_{r})^{\text{val}(v)}$$
recording the multiplicities of $L_1,\cdots, L_{k},N$ at every special point of
$C_v$. We have the following remarks:
\begin{remark}\label{rmk:quasivertex}
  The crucial observation, now, is the following. For a stable vertex $v$ such
  that $j(v) = 0$, we have $z_1|_{C_v} \equiv 0$, so the stability condition
  \eqref{eq:stab1} implies that $l_{\epsilon\theta\oplus
    \bigoplus_{j=1}^{p}\epsilon\mathrm{id}_{\C^{*}}}(q)\leq 1$ for each $q \in
  C_v$. That is, the restriction of $(C; q_1, \ldots, q_m;
  L_1,\cdots,L_{k+p};\vec{x},\vec{y})$ to $C_v$ gives rise to a $\epsilon\theta\oplus \bigoplus_{j=1}^{p}\epsilon\mathrm{id}_{\C^{*}}$-stable
  quasimap to the quotient stack $\mathfrak Y_{p}:=[AY/G] \times [\C/\C^{*}]^{p} $ (c.f. \ref{def:theta-ep-quasi}) in
  $$Q^{(\epsilon\theta,\boldsymbol{\varepsilon})}_{0,\vec{m}(v)||J_{v}|}(\mathfrak
  Y_{p},(\beta(v),1^{J_{v}})):=\bigsqcup_{\substack{d\in
      \mathrm{Eff}(AY,G,\theta)\\(i_{\mathfrak
        Y})_{*}(d)=\beta(v)}}Q^{(\epsilon\theta,\boldsymbol{\varepsilon})}_{0,\vec{m}(v)||J_{v}|}(\mathfrak Y_{p},(d,1^{J_{v}}))\
  .$$
  In this case, let $j\in J_{v}$, the evaluation
  map considered in \eqref{eq:hatev} coincides with $\hat{ev}_{j}$ for $Q^{(\epsilon\theta,\boldsymbol{\varepsilon})}_{0,\vec{m}(v)||J_{v}|}(\mathfrak
  Y_{p},(\beta(v),1^{J_{v}}))$ in Remark \ref{rmk:equiv-theta-epl}.  
  On the other hand, for a stable vertex $v$ such that $j(v) = \infty$, we
  have $z_2|_{C_v} \equiv 0$, so the stability condition \eqref{eq:stab2}
  implies that $\text{ord}_q(\vec{x}) =\text{ord}_{q}(\vec{y})=0$ for each $q \in C_v$. Thus, the
  restriction of $(C; q_1, \ldots, q_m; L_1,\cdots,L_{k};\vec{x})$ to $C_v$
  gives rise to a usual twisted stable map in
  $$\mathcal
  K_{0,\vec{m}(v)}(\sqrt[r]{L_{\theta}/Y},\beta(v)):=\bigsqcup_{\substack{d\in
      \mathrm{Eff}(AY,G,\theta)\\(i_{\mathfrak Y})_{*}(d)=\beta(v)}}\mathcal
  K_{0,\vec{m}(v)}(\sqrt[r]{L_{\theta}/Y},d)\ .$$ Here $\sqrt[r]{L_{\theta}/Y}$
  is the root gerbe of $Y$ by taking $r$-th root of $L_{\theta}$.
\end{remark}

\begin{remark}\label{rmk:edge1}
  For each edge $e$, the restriction of $(\vec{x},\vec{y})$ to $C_e$ defines a constant
  map to $Y$ (possibly with an additional basepoint at the ramification point
  $q_{0}$). So if there is no basepoint on $C_{e}$, $(\vec{x},\vec{y},\vec{z})$ defines
  a representable map
  $$C_{e}\rightarrow \mathbb B G_{y}\times \P_{r,1}$$
  where $y\in Y$ comes from $\vec{x}$, $G_{y}$ is the isotropy group of $y\in
  Y$. Then we have $m(q_{0})=(g^{-1},1)$ and
  $m(q_{\infty})=(g,\mu_{r}^{\delta(e)})$ for some $g\in G_{y}$. Note when $r$
  is a sufficiently large prime comparing to $\delta(e)$, assuming that the
  order of $g$ is equal to $a$, we have $C_{e}\cong \P^{1}_{ar,a}$ and the
  ramification point $q_{\infty}$ must be a special point. Here $\P^{1}_{ar,a}$
  is the unique Deligne-Mumford stack with coarse moduli $\P^{1}$, isotropy
  group $\boldsymbol{\mu}_{a}$ at $0\in \P^{1}$, isotropy group
  $\boldsymbol{\mu}_{ar}$ at $\infty\in \P^{1}$, and generic trivial stabilizer.

  If $q_{0}$ is a basepoint of degree $(\beta,1^{J_{e}})$, the ramification point $q_{0}$
  can't be an orbifold point, thus $m(q_{0})=(1,1)\in G\!\times\!
  \boldsymbol{\mu}_{r}$. In this case, by the representable condition, we have
  $C_{e}\cong \P_{ar,1}$ and $m(q_{\infty})=(g_{\beta},\mu_{r}^{\delta(e)})$ if
  $r$ is a sufficiently large prime. Here $a$ is minimal positive integer
  associated to $g_{\beta}$ as in \S \ref{sec:I-fun}.
\end{remark}
\begin{remark}\label{rmk:edge2}
  If there is a basepoint on the edge curve $C_{e}$, then the degree
  $(\beta(e),1^{J_{e}}, \frac{\delta(e)}{r})$ on $C_{e}$ must satisfy the relation
  $\delta(e)\geq \beta(e)(L_{\theta })+|J_e|$. Otherwise we have $z_{1}|_{C_{e}}\equiv
  0$, given the fact $z_{2}$ vanishes at $q_{\infty}$, this will violate the
  nondegeneracy condition for $z_{1}$ and $z_{2}$.
\end{remark}

\subsection{Localization analysis}\label{subsec:local-ana1}
Fix $\beta\in \mathrm{Eff}(W,G,\theta)$ and $\delta\in \Z_{\geq 0}$, we will consider
the space $Q^{\widetilde{\theta}}_{0,\vec{m}}(\P \mathfrak
Y^{\frac{1}{r},p},(\beta,1^{p},\frac{\delta}{r}))$. The reason why we assume that the
second degree is $\frac{\delta}{r}$ is that
$Q^{\widetilde{\theta}}_{0,\vec{m}}(\P \mathfrak
Y^{\frac{1}{r},p},(\beta,1^{p},\frac{\delta}{r}))$ corresponds to
$Q^{\widetilde{\theta}}_{0,\vec{m}}(\P \mathfrak Y,(\beta,\delta)) $, here $\P
\mathfrak Y$ is equal to $\P \mathfrak Y^{\frac{1}{r},p}$ for $r=1$ and $p=1$. In the
remaining section, we will always assume that $r$ is a \emph{sufficiently large prime}.

By virtual localization formula of Graber--Pandharipande~\cite{Graber_1999}, we
can write
$$[Q^{\widetilde{\theta}}_{0,\vec{m}}(\P \mathfrak
Y^{\frac{1}{r},p},(\beta,1^{p},\frac{\delta}{r}))]^{ \mathrm{vir}}\ ,$$
in terms of contributions from each decorated graph $\Gamma$:
\begin{equation}
  \label{localization}
  [Q^{\widetilde{\theta}}_{0,\vec{m}}(\P \mathfrak Y^{\frac{1}{r},p},(\beta,1^{p},\frac{\delta}{r}))]^{\mathrm{vir}} = \sum_{\Gamma} \frac{1}{\mathbb A_{\Gamma}}\iota_{\Gamma*} \left(\frac{[F_{\Gamma}]^{\mathrm{vir}}}{e^{\C^{*}}(N_{\Gamma}^{\mathrm{vir}})}\right)\ .
\end{equation}
Here, for each graph $\Gamma$, $[F_{\Gamma}]^{\mathrm{vir}}$ is obtained via the $\C^*$-fixed part of the
restriction to the fixed loci of the obstruction theory on
$Q^{\widetilde{\theta}}_{0,\vec{m}}(\P \mathfrak
Y^{\frac{1}{r},p},(\beta,1^{p},\frac{\delta}{r}))$, and $N_{\Gamma}^{\mathrm{vir} }$ is
the equivariant Euler class of the $\C^*$-moving part of this restriction. Besides, $\mathbb
A_{\Gamma}$ is the automorphism factor for the graph $\Gamma$, which represents
the degree of $F_{\Gamma}$ into the corresponding open and closed $\C^{*}$-fixed
substack in $Q^{\widetilde{\theta}}_{0,\vec{m}}(\P \mathfrak
Y^{\frac{1}{r},p},(\beta,1^{p},\frac{\delta}{r}))$.

We will do an explicit computation for the contributions of each graph
$\Gamma$ in the following. As for the contribution of a graph $\Gamma$ to
\eqref{localization}, one can first apply the normalization exact sequence to
the relative obstruction theory \eqref{pot} and \eqref{eq:fac-pot}, which decomposes the contribution from
$\Gamma$ to \eqref{localization} into contributions from vertex, edge, and node
factors. This includes all but the automorphisms and deformations within
$\mathcal M^{tw}_{0,\vec{m}}$. The latter are distributed in the
vertex, edge, and node factors as deformations of
the vertex components, deformations of the edge
components, and deformations of smoothing the nodes,
respectively. We also include automorphisms of the source curve in
the edge contributions as part of gerbe structure of the edge moduli
$\mathcal M_{e}$, then an additional factor from the gerbe structure of each edge
moduli will be included in the automorphism factor $\mathbb A_{\Gamma}$ (see \eqref{eq:auto-loc1} for the localization
contribution of graph $\Gamma$).
\subsubsection{Vertex contributions}\label{subsubsec:ver-contr1}
First of all, consider the stable vertex $v$ over $\infty$, this vertex moduli
$\mathcal M_{v}$ corresponds to the moduli stack
$$\mathcal K_{0,\vec{m}(v)}(\sqrt[r]{L_{\theta}/Y},\beta(v)):=\bigsqcup_{\substack{d\in
    \mathrm{Eff}(AY,G,\theta)\\(i_{\mathfrak Y})_{*}(d)=\beta(v)}}\mathcal
K_{0,\vec{m}(v)}(\sqrt[r]{L_{\theta}/Y},d)\ ,$$ which parameterizes twisted
stable maps to the root gerbe $\sqrt[r]{L_{\theta}/Y}$ over $Y$.

Let $$\pi:\mathcal C_{\infty}\rightarrow \mathcal
K_{0,\vec{m}(v)}(\sqrt[r]{L_{\theta}/Y},\beta(v))$$ be the universal curve over
$\mathcal K_{0,\vec{m}(v)}(\sqrt[r]{L_{\theta}/Y},\beta(v))$. In this case, on
$\mathcal C_{\infty}$, we have $\mathcal L_{-\theta }\otimes \mathcal
N^{\otimes r}\otimes \mathbb C_{\lambda}\cong \mathcal O_{\mathcal C_{\infty}}$
as $z_{1}|_{\mathcal C_{\infty}}\equiv 1$, hence we have $\mathcal N\cong
\mathcal L_{\theta}^{\frac{1}{r}}\otimes \C_{-\frac{\lambda}{r}}$, here
$\mathcal L_{\theta}^{\frac{1}{r}}$ is the line bundle over $\mathcal
C_{\infty}$ that is the pull back of the universal root bundle over
$\sqrt[r]{L_{\theta}/Y}$ along the universal map $f:\mathcal
C_{\infty}\rightarrow \sqrt[r]{L_{\theta}/Y}$. The movable part of the perfect
obstruction theory comes from the deformation of $z_{2}$, thus the \emph{inverse
  of Euler class} of the virtual normal bundle is equal to
\begin{equation*}
  e^{\C^{*}}((-R^{\bullet}\pi_{*}\mathcal L_{\theta}^{\frac{1}{r}})\otimes \C_{-\frac{\lambda}{r}}).
\end{equation*}
When $r$ is a sufficiently large prime,
following~\cite{janda18_doubl_ramif_cycles_with_target_variet}, the above Euler
class has a representation
$$\sum_{d\geq 0}c_{d}(-R^{\bullet}\pi_{*}\mathcal
L_{\theta}^{\frac{1}{r}})(\frac{-\lambda}{r})^{|E(v)|-1-d} \ .$$ Here the
virtual bundle $-R^{\bullet}\pi_{*}\mathcal L_{\theta}^{\frac{1}{r}}$ has virtual rank $|E(v)|-1$, where $|E(v)|$ is the number of edges incident to
the vertex $v$. The fixed part of
the perfecct obstruction theory contributes to the virtual cycle
$$[\mathcal K_{0,\vec{m}(v)}(\sqrt[r]{L_{\theta}/Y},\beta(v))]^{\mathrm{vir}}\ .$$

For the stable vertex $v$ over $0$, the vertex moduli $\mathcal M_{v}$ corresponds
to the moduli space
$$Q^{\epsilon\theta, \boldsymbol{\varepsilon}}_{0,\vec{m}(v)||J_{v}|}(\mathfrak Y,(\beta(v),1^{J_{v}})):=\bigsqcup_{\substack{d\in
    \mathrm{Eff}(AY,G,\theta)\\(i_{\mathfrak
      Y})_{*}(d)=\beta(v)}}Q^{\epsilon\theta,\boldsymbol{\varepsilon} }_{0,\vec{m}(v)||J_{v}|}(\mathfrak Y,(d,1^{J_{v}}))\
.$$

Let $\pi:\mathcal C_{0}\rightarrow Q^{\epsilon\theta,\boldsymbol{\varepsilon}}_{0,\vec{m}(v)||J_{v}|}(\mathfrak
Y,(\beta(v),1^{J_{v}}))$ be the universal curve over
$Q^{\epsilon\theta,\boldsymbol{\varepsilon}}_{0,\vec{m}(v)||J_{v}|}(\mathfrak Y,(\beta,1^{J_{v}}))$. In this case, the fixed
part of the obstruction theory of the vertex moduli over $0$ yields the
virtual cycle
$$[Q^{\epsilon\theta,\boldsymbol{\varepsilon}}_{0,\vec{m}(v)||J_{v}|}(\mathfrak Y,(\beta(v),1^{J_{v}}))]^{\mathrm{vir}}\ .$$
Note $\mathcal N|_{\mathcal C_{0}}=\mathcal O_{\mathcal C_{0}}$ as
$z_{2}|_{\mathcal C_{0}}\equiv 1$, therefore the virtual normal comes from the
movable part of the infinitesimal deformations of the section $z_{1}$, which is
a section of the line bundle $\mathcal L_{-\theta_{p}}$ over $\mathcal C_{0}$,
whose Euler class is equal to
$$e^{\C^{*}}((R^{\bullet}\pi_{*}\mathcal L_{-\theta_{p}})\otimes \C_{\lambda})\ .$$
\subsubsection{Edge contributions: basepoint case}\label{subsubsec:edgebase}
When there is a base point on the edge curve, it has degree
$(\beta(e),1^{J_{e}},\frac{\delta(e)}{r})$ with $\beta(e)\neq 0$ and $\delta(e)\geq
\beta(e)(L_{\theta })+|J_{e}|$ by Remark \ref{rmk:edge2}, we will write
$\beta(e)$ as $\beta$ only in this subsection for simplicity unless stated otherwise.
Then the multiplicity at $q_{\infty}\in C_{e}$ is equal to
$(g,\mu_{r}^{\delta(e)})\in G \!\times\! \boldsymbol{\mu}_{r}$, where
$g=g_{\beta}$ is defined in \S\ref{sec:I-fun}. Let $a$ be the
minimal positive integer associated to $\beta$ as in \S\ref{sec:I-fun},
which is also the order of $g_{\beta}$. When $r$ is sufficiently large, due to
Remark \ref{rmk:edge1}, $C_{e}$ must be isomorphic to $\P^{1}_{ar,1}$ where the
ramification point $q_{0}$ for which $z_{1}=0$ is an ordinary point, and the
ramification point $q_{\infty}$ for which $z_{2}=0$ must be a special point,
which is isomorphic to $\mathbb B\boldsymbol{\mu}_{ar}$.

Recall that
$$[Y^{ss}_{\beta}/G]\cong[(Z^{ss}_{\beta}\cap AY)/G]$$
in \S\ref{sec:I-fun}. We now define the edge moduli $\mathcal M_{e}$
to be
$$\sqrt[a\delta(e)]{L_{-\theta}/[Y^{ss}_{\beta}/G]}\ ,$$
which is the root gerbe over the stack $[Y^{ss}_{\beta}/G]\subset
[AY^{ss}(\theta)^{g}/G]\subset I_{\mu} Y$ by taking $a\delta(e)$th root of the
line bundle $L_{-\theta}$ on $[Y^{ss}_{\beta}/G]$.

The root gerbe $\sqrt[a\delta(e)]{L_{-\theta}/[Y^{ss}_{\beta}/G]}$
admits a representation as a quotient stack:
\begin{equation*}\label{eq:edge}
  [(Y^{ss}_{\beta} \!\times\! \C^{*}) / (G \!\times\! \C^{*}_{w})],
\end{equation*}
where the (right) action is defined by
$$(\vec{x},v)\cdot (g,w)=(\vec{x}\cdot g,\theta(g)vw^{a\delta(e)})\ ,$$
for all $(g,w)\in G \!\times\! \C^{*}_{w}$ and $(\vec{x},v)\in A(Y)^{g}
\!\times\! \C^{*}$. Here $ \vec{x}\cdot g$ is given by the action as in the
definition of $[AY/G]$. For every character $\rho$ of $G$, we can define a new
character of $G \!\times\! \C^{*}_{w}$ by composing the projection map
$\mathrm{pr}_{G}:G \!\times\! \C^{*}_{w} \rightarrow G$. By an abuse of
notation, we will continue to use the notation $\rho$ to name the new character
of $G \!\times\! \C^{*}_{w} $. Then the new character $\rho$ will determines a
line bundle $L_{\rho}:=[(Y^{ss}_{\beta}\!\times\! \C^{*} \times \mathbb
C_{\rho}) /(G \!\times\!\C^{*}_{w})]$ on
$\sqrt[a\delta(e)]{L_{-\theta}/[Y^{ss}_{\beta}/G]}$. 

By virtue of its universal
property of the root gerbe $\sqrt[a\delta(e)]{L_{-\theta}/[Y^{ss}_{\beta}/G]}$, there is a universal line bundle $\mathcal R$ that is the
$a\delta(e)$th root of line bundle $L_{-\theta}$ over the root gerbe. This root line bundle
$\mathcal R$ can also constructed by the Borel construction, i.e. $\mathcal R$
is associated to the character $p_{2}$:
$$\mathrm{pr}_{\C^{*}_{w}}:G \!\times\! \C^{*}_{w} \rightarrow \C^{*}_{w} \quad (g,w)\in G \!\times\!
\C^{*}_{w} \mapsto w\in \C^{*}_{w}\ .$$ We have the relation
$$L_{-\theta}=\mathcal R^{a\delta(e)}\ .$$
Then the coordinate functions $(\vec{x},v)\in Y^{ss}_{\beta} \!\times\! \C^{*} $
descents to be tautological sections of vector bundle
$\bigoplus_{i=1}^{n}L_{\rho_{i}}\oplus (L_{\theta}\otimes \mathcal
R^{a\delta(e)})$ on $\sqrt[a\delta(e)]{L_{-\theta}/[Y^{ss}_{\beta}/G]}$.

We will construct a universal family of $\C^{*}-$fixed quasimaps to $\P \mathfrak Y^{\frac{1}{r},p}$
of degree $(\beta,1^{J_{e}},\frac{\delta(e)}{r})$ over the edge moduli $\mathcal
M_{e}$, which takes the form
\begin{equation*}
  \xymatrix{
    \mathcal C_e := \P_{ar,1}(\mathcal R^{\otimes a} \oplus \O_{\mathcal M_{e}}) \ar^-{ev}[r]\ar_{\pi}[d] & \P \mathfrak Y^{\frac{1}{r},p}\\
    \mathcal M_{e}:=\sqrt[a\delta(e)]{L_{-\theta}/[Y^{ss}_{\beta}/G]}\ .
  }
\end{equation*}

The universal curve $\mathcal C_{e}$ over the edge moduli $\mathcal M_{e}$ is constructed as a
quotient stack:
$$\mathcal C_{e}=[(Y^{ss}_{\beta} \!\times\! \C^{*} \!\times\! U)/ (G \!\times\! \C^{*}_{w}
\!\times\! \C^{*}_{t})]\ ,$$
where the right action is defined by:
$$(\vec{x},v,x,y)\cdot (g,w,t)=(\vec{x}\cdot g,\theta(g)vw^{a\delta(e)},w^{a}t^{ar}x,ty)\ ,$$
for all $(g,w,t)\in G \!\times\! \C^{*}_{w} \!\times\! \C^{*}_{t}$ and $(\vec{x},v,(x,y))=((x_{1},\cdots, x_{n}),v,(x,y))\in Y^{ss}_{\beta} \!\times\!
\C^{*} \!\times\! U$.

The universal map $ev$ from $\mathcal C_{e}$ to $\P \mathfrak Y^{\frac{1}{r},p}$
can be presented as follows:
$$\tilde{ev}: Y^{ss}_{\beta} \!\times\! \C^{*} \!\times\! U\rightarrow AY_{p} \!\times\! U\ ,$$
defined by:
\begin{equation}
  \begin{split}
    (\vec{x},v,(x,y))&\in Y^{ss}_{\beta} \!\times\! \C^{*} \!\times\! U \mapsto\\
    &\big(\big(x_{1}x^{\beta(L_{\rho_{1}})},\cdots,x_{n}x^{\beta(L_{\rho_{n}})}\big),(x)_{j\in J_{e}},v^{-1}x^{\delta(e)-\beta(L_{\theta})-|J_{e}|},y^{a\delta(e)}\big)\in
    AY_{p} \!\times\! U\ .
  \end{split}
\end{equation}
Here $(x)_{j\in J_{e}}$ are elements belonging to $\C^{p}$ so that the $j-$th
component is $1$ if $j\notin J_{e}$ and all the other components are $x$. 
Note that when $\beta(L_{\rho_{i}})\notin \Z_{\geq 0}$ for some $i$, we
must have $x_{i}=0$ as $\vec{x}\in Y^{ss}_{\beta}$, so the $\tilde{ev}$ is well
defined. Then $\tilde{ev}$ is equivariant with respect to the group homomorphism from $G
\!\times\! \C^{*}_{w} \!\times\! \C^{*}_{t} $ to $G_{p} \!\times\! \C^{*}$ defined
by:
\begin{equation}
  \begin{split}
    (g,w,t)&\in G \!\times\! \C^{*}_{w} \!\times\! \C^{*}_{t} \mapsto\\
    &\big(g\big(t^{ar\beta(L_{\pi_{1}})}w^{a\beta(L_{\pi_{1}})},\cdots,
    t^{ar\beta(L_{\pi_{k}})}w^{a\beta(L_{\pi_{k}})}\big),(w^{a}t^{ar})_{j\in J_{e}},t^{a\delta(e)}\big)\in
    G_{p} \!\times\! \C^{*}\ .
  \end{split}
\end{equation}
Here $(w^{a}t^{ar})_{j\in J_{e}}$ is the element belonging to $(\C^{*})^{p}$ so that the
$j-$th component is $1$ if $j\notin J_{e}$ and all the other components are $w^{a}t^{ar}$. This gives the universal morphism $f$ from $\mathcal C_{e}$ to $\P \mathfrak
Y^{\frac{1}{r},p}$ by descent.

There is a tautological line bundle $\mathcal O_{C_{e}}(1)$ on $C_{e}$
associated to the character $\mathrm{pr}_{\C^{*}_{t}}$ of $G \!\times\!
\C^{*}_{w} \!\times\! \C^{*}_{t}$ by the Borel construction. Here
$\mathrm{pr}_{\C^{*}_{t}}$ is the projection map from $G \!\times\! \C^{*}_{w}
\!\times\! \C^{*}_{t}$ to $\C^{*}_{t}$. 

We will define a (quasi\footnote{This means we allow
  $\C^{*}-$action on $C_{e}$ with fractional weight. See a similar discussion
  in~\cite[\S 2.2]{chang16_effec_theor_gw_fjrw_invar}.} left) $\C^{*}-$action on $\mathcal C_{e}$ such that the
map $ev$ constructed above is $\C^{*}-$equivariant. Define a (left)
$\C^{*}-$action on $\mathcal C_{e}$ which is induced from the $\C^{*}-$action on $Y^{ss}_{\beta} \!\times\! \C^{*} \!\times\! U$:
$$m:\C^{*}\times Y^{ss}_{\beta} \!\times\! \C^{*} \!\times\! U\rightarrow Y^{ss}_{\beta} \!\times\! \C^{*} \!\times\! U\ ,$$
$$t\cdot (x,v,(x,y))= (x,v,(x,t^{\frac{-1}{ar\delta(e)}}y))\ .$$
Note the morphism $\pi$ is also $\C^{*}$-equivariant, where $\mathcal M_{e}$ is
equipped with trivial $\C^{*}$-action. By the universal
property of the projectivized bundle $\mathcal C_e$ over
$\mathcal M_{e}$, the line bundle
$\mathcal O_{\mathcal C_e}(1)$ is equipped with a tautological section
\begin{equation*}\label{eq:tautosec}
  (x, y) \in H^0\big(\big(\O_{\mathcal C_e}(ar) \otimes \mathcal \pi^{*}\mathcal R^{\otimes a}\big) \oplus (\O_{\mathcal C_e}(1)\otimes \C_{\frac{-1}{ar\delta(e)}})\big)\ ,
\end{equation*}
which is also a $\C^{*}-$invariant section. Here $\mathcal O_{\mathcal
  C_{e}}(1)$ is the standard $\C^{*}$-equivariant line bundle on $\mathcal C_{e}$ by the Borel construction.

Now we can check that $ev$ is a $\C^{*}-$equivariant morphism from $\mathcal
C_{e}$ to $\P \mathfrak Y^{\frac{1}{r},p}$ with respect to the $\C^{*}-$actions
for $\mathcal C_{e}$ and $\P \mathfrak Y^{\frac{1}{r},p}$. According to Remark
\ref{rmk:equivmap}, $ev$ is equivalent to the following data:
\begin{enumerate}\label{equivedge} 
\item $k+p+1$ $\C^{*}$-equivariant line bundles on
  $\mathcal C_{e}$:
  $$\mathcal L_{j} := \pi^*L_{\pi_{j}} \otimes \O_{\mathcal
    C_e}(ar\beta(L_{\pi_{j}}))\otimes \pi^{*}\mathcal R^{\otimes a\beta(L_{\pi_{j}})},
  1\leq j\leq k\ ,$$
 $$\mathcal L_{k+j}:=\pi^{*}\mathcal R^{\otimes a}\otimes \mathcal O_{\mathcal C_{e}}(ar), j\in J_{e},\; \mathcal L_{k+j}:=\C,\; j\notin J_{e}$$
  and
  $$ \mathcal N:=\mathcal O_{\mathcal C_{e}}(a\delta(e))\otimes \mathbb
  C_{\frac{-\lambda}{r}}\ ,$$ where the line bundles $L_{\pi_{j}}$, $\mathcal R$
  are the standard $\C^{*}$-equivariant line bundle on
  $\mathcal M_{e}$ by the Borel construction;
  \item a universal section
  \begin{equation}
    \begin{split}
      \big(\vec{x},\vec{y},(\zeta_1, \zeta_2)\big) :=&
      \big((x_{1}x^{\beta(L_{\rho_{1}})},\cdots,x_{n}x^{\beta(L_{\rho_{n}})}),(x)_{J_{e}},(v^{-1}x^{\delta(e)
        - \beta(L_{\theta})-|J_{e}|}, y^{a\delta(e)})\big)\\
      &\in H^0\big(\mathcal C_{e}, (\oplus_{i=1}^{n} \mathcal L_{\rho_{i}})\oplus(\oplus_{j=1}^{p}\mathcal L_{k+j})\oplus
      (\mathcal L_{-\theta_{p}}\otimes \mathcal N^{\otimes r}\otimes \mathbb
      C_{\lambda})\oplus \mathcal N\big)^{\C^{*}}\ ,
    \end{split}
  \end{equation}
where the line bundles $\mathcal L_{-\theta_{p}}$ and $\mathcal L_{\rho_{i}}$ are
constructed from line bundles $\mathcal L_{j}$ as before.
\end{enumerate}

Equipped with these notations, now we compute the localization contribution from
$\mathcal M_{e}$. Based on the perfect obstruction theory for quasimaps in
$Q^{\widetilde{\theta}}_{0,\vec{m}}(\P \mathfrak
Y^{\frac{1}{r},p},(\beta,1^{p},\frac{\delta}{r}))$, the restriction of the prefect
obstruction theory to $\mathcal M_{e}$ decomposes into three parts: (1) the
deformation theory of source curve $\mathcal C_{e}$; (2) the deformation theory
of the lines bundles $(\mathcal L_{j})_{1\leq j\leq k+p}$ and $\mathcal N$; (3)
the deformation theory for the section
$$(\vec{x},\vec{y},(\zeta_{1},\zeta_{2}))\in \Gamma\left(\oplus_{i=1}^{n} \mathcal
  L_{\rho_{i}}\oplus (\oplus_{j=1}^{p}\mathcal L_{k+j}) \oplus (\mathcal L_{-\theta_{p}}\otimes \mathcal N^{\otimes
    r}\otimes \mathbb C_{\lambda}) \oplus \mathcal N\right)\ .$$

The virtual normal bundle comes from the movable part of the three parts, and
the fixed part will contribute to the virtual cycle of $\mathcal M_{e}$. First every fiber curve $C_{e}$ in $\mathcal C_{e}$ is isomorphic to $\P_{ar,1} $,
which is rational. Then the infinitesimal deformations/obstructions of
$C_{e}$ and the line bundles $L_{j}:=\mathcal L_{j}|_{C_{e}},\; N:=\mathcal
N|_{C_{e}}$ are zero. Hence their contribution to the perfect obstruction theory
solely comes from infinitesimal automorphisms. The infinitesimal automorphisms of $C_{e}$
come from the space of vector field on $C_{e}$ that vanishes on special points.
Thus the $\C^{*}-$fixed part of the infinitesimal automorphisms of $C_{e}$ comes from the
$1-$dimensional subspace of vector fields on $C_{e}$ which vanish on the two
ramification points, which, together with the infinitesimal automorphisms of
line bundle $N$, will be canceled with the fixed part of infinitesimal
deformation of sections $(z_{1},z_{2}):=(\zeta_{1},\zeta_{2})|_{C_{e}}$. The
movable part of infinitesimal automorphisms of $C_{e}$ is nonzero only if at
least one of ramification points on $C_{e}$ is not a special point. By Remark
\ref{rmk:edge1}, the ramification $q_{\infty}$ must be a special point since it
has nontrivial stacky structure when $r$ is sufficiently large, and the
ramification point $q_{0}$ is not a special point, then the movable part of
infinitesimal automorphisms of $C_{e}$ contributes
$$\frac{\delta(e)}{\lambda-D_{\theta}}$$
to the virtual normal bundle.

Now let's turn to the localization contribution from sections. As for the deformations of $z_{2}$, we continue to use the tautological section $(x,y)$ in
\eqref{eq:tautosec}. Sections of $N$ is spanned by monomials
$(x^{m}y^{n})|_{C_{e}}$ with $arm+n=a\delta(e)$ and $m,n\in \Z_{\geq 0}$. Note
$x^{m}y^{n}$ may not be a global section of $\mathcal N$ but always a global
section of the line bundle $R^{\otimes am}\otimes \mathcal N\otimes
\C_{\frac{m}{\delta(e)}\lambda}$. Then $R^{\bullet}\pi_{*}\mathcal N$ will decompose as
a direct sum of line bundles, each corresponds to the monomial $x^{m}y^{n}$, whose
first chern class is
$$c_{1}(\mathcal R^{\otimes -am} \bigotimes \mathbb C_{\frac{-m}{\delta(e)}\lambda})=\frac{m}{\delta(e)}(D_{\theta}-\lambda)\ .$$
So the total contribution is equal to
$$\prod_{m=0}^{\lfloor\frac{\delta(e)}{r}\rfloor}\bigg(\frac{m}{\delta(e)}(D_{\theta}-\lambda)\bigg)\ .$$
The term corresponding to $m=0$ in the above product is the $\C^{*}-$invariant part of $R^{\bullet}\pi_{*}\mathcal N$, it
will contribute to the virtual cycle of $\mathcal M_{e}$. The rest contributes
to the virtual normal bundle as
$$\prod_{m=1}^{\lfloor\frac{\delta(e)}{r}\rfloor}\bigg(\frac{m}{\delta(e)}(D_{\theta}-\lambda)\bigg)\ .$$
Note when $r$ is sufficiently large, the above product becomes $1$.


For the deformation of $z_{1}$, arguing in the same way as $z_{2}$, the Euler
class of $R^{\bullet}\pi_{*}(\mathcal L_{-\theta_{p}}\otimes \mathcal N^{\otimes r}\otimes
\C_{\lambda})$ is equal to
$$\prod_{m=0}^{\delta(e)-\beta(L_{\theta})-|J_e|}\bigg(\frac{m}{\delta(e)}(-D_{\theta}+\lambda)\bigg)\ .$$ 
The factor for $m=0$ appearing in the above product is the $\C^{*}-$fixed part
of $R^{\bullet}\pi_{*}(\mathcal L_{-\theta}\otimes \mathcal N^{\otimes r}\otimes \C_{\lambda})$, it will contribute to the virtual cycle of
$\mathcal M_{e}$. The rest contributes to the virtual normal bundle as 
$$\prod_{m=1}^{\delta(e)-\beta(L_{\theta})-|J_e|}\bigg(\frac{m}{\delta(e)}(-D_{\theta}+\lambda)\bigg)\ .$$ 

Finally, let's turn to the localization contribution from the sections
$\vec{x}$ and $\vec{y}$. Before that, using the
same argument above, one can prove the following lemma:
\begin{lemma}\label{lem:dirsum}
  When $n\in \Z_{\geq 0}$, we have
  $$e^{\C^{*}}\big(R^{\bullet}\pi_{*}(\mathcal O_{\mathcal C_{e}}(n))\big)=\prod_{m=0}^{\lfloor{\frac{n}{ar}}\rfloor}\bigg(\frac{m}{\delta(e)}(D_{\theta}-\lambda)+\frac{n}{ar\delta(e)}\lambda\bigg)\ .$$
  When $n\in \Z_{<0}$, we have
  $$e^{\C^{*}}\big(R^{\bullet}\pi_{*}(\mathcal O_{\mathcal
    C_{e}}(n))\big)=\prod_{\frac{n}{ar}<
    m<0}\frac{1}{\frac{m}{\delta(e)}(D_{\theta}-\lambda)+\frac{n}{ar\delta(e)}\lambda}\ .$$
\end{lemma}
Using the above lemma, we have the following description of
$e^{\C^{*}}(R^{\bullet}\pi_{*}\mathcal L_{\rho_{i}})$ for $1\leq i\leq n$. Then for each $\rho_{i}$, we have:
\begin{enumerate}\label{virnomx}
\item If $\beta(L_{\rho_{i}})\in \mathbb Q_{\geq 0}$, one has
  \begin{equation*}\label{eq:virnomx1}
    \begin{split}
      e^{\C^{*}}\big(R^{\bullet}\pi_{*}(\mathcal
      L_{\rho_{i}})\big)&=e^{\C^{*}}\big(R^{\bullet}\pi_{*}(\pi^{*}(L_{\rho_{i}})\otimes \mathcal
      O_{\mathcal C_{e}}(ar\beta(L_{\rho_{i}}))\otimes \pi^{*}(\mathcal R^{\otimes a\beta(L_{\rho_{i}})}))\big)\\
      &=e^{\C^{*}}\big(L_{\rho_{i}}\otimes \mathcal R^{\otimes a\beta(L_{\rho_{i}})} \otimes R^{0}\pi_{*}(\mathcal O_{\mathcal C_{e}}(ar\beta(L_{\rho_{i}})))\big)\\
      &=\prod_{m=0}^{\lfloor {\beta(L_{\rho_{i}})}\rfloor}\bigg( D_{\rho_{i}}+\frac{\beta(L_{\rho_{i}})(-D_{\theta})}{\delta(e)}+\frac{m}{\delta(e)}(D_{\theta}-\lambda)+\frac{\beta(L_{\rho_{i}})}{\delta(e)}\lambda\bigg)\\
      &=\prod_{m=0}^{\lfloor {\beta(L_{\rho_{i}})} \rfloor}\bigg(
      D_{\rho_{i}}+\frac{\beta(L_{\rho_{i}})-m}{\delta(e)}(\lambda-D_{\theta})\bigg)\ .
    \end{split}
  \end{equation*}
  Hence we have
  \begin{align*}\label{eq:virnomx1mov}
    e^{\C^{*}}((R^{\bullet}\pi_{*}\mathcal L_{\rho_{i}})^{\mathrm{mov}})
    &=\prod_{0\leq m< \beta(L_{\rho_{i}})}\big( D_{\rho_{i}}+\frac{\beta(L_{\rho_i})-m}{\delta(e)}(\lambda-D_{\theta})\big)\ .
  \end{align*}
  Note the invariant part of $R^{\bullet}\pi_{*}\mathcal L_{\rho_{i}}$ is nonzero
  only when $\beta(L_{\rho_{i}})\in \Z_{\geq 0}$.
\item If $\beta(L_{\rho_{i}})\in \mathbb Q_{< 0}$, one has
  \begin{equation*}\label{eq:virnomx2}
    \begin{split}
      e^{\C^{*}}\big(R^{\bullet}\pi_{*}\mathcal
      L_{\rho_{i}}\big)&=e^{\C^{*}}\big(R^{\bullet}\pi_{*}\big(\pi^{*}L_{\rho_{i}}\otimes
      \mathcal
      O_{\mathcal C_{e}}(ar\beta(L_{\rho_{i}}))\otimes \pi^{*}\mathcal R^{\otimes a\beta(L_{\rho_{i}})}\big)\big)\\
      &=\frac{1}{e^{\C^{*}}\big(L_{\rho_{i}}\otimes \mathcal R^{\otimes a\beta(L_{\rho_{i}})} \otimes R^{1}\pi_{*}(\mathcal O_{\mathcal C_{e}}(ar\beta(L_{\rho_{i}})))\big)}\\
      &=\prod_{\beta(L_{\rho_{i}})<m<0}\frac{1}{ D_{\rho_{i}}+\frac{\beta(L_{\rho_{i}})(-D_{\theta})}{\delta(e)}+\frac{m}{\delta(e)}(D_{\theta}-\lambda)+\frac{\beta(L_{\rho_{i}})}{\delta(e)}\lambda}\\
      &=\prod_{\beta(L_{\rho_{i}})<m<0}\frac{1}{
        D_{\rho_{i}}+\frac{\beta(L_{\rho_{i}})-m}{\delta(e)}(\lambda-D_{\theta})}\ ,
    \end{split}
  \end{equation*}
  which implies that
  \begin{align*}\label{eq:virnomx2mov}
    e^{\C^{*}}\big((R^{\bullet}\pi_{*}\mathcal L_{\rho_{i}})^{\mathrm{mov}}\big)&=e^{\C^{*}}\big(R^{\bullet}\pi_{*}\mathcal
      L_{\rho_{i}}\big)\\
    &=\prod_{\beta(L_{\rho_{i}})<m<0}\frac{1}{ D_{\rho_{i}}+\frac{\beta(L_{\rho_{i}})-m}{\delta(e)}(\lambda-D_{\theta})}\ .
  \end{align*}
\end{enumerate}
The movable part of deformation of $\vec{y}$ contributes
$$e^{\C^{*}}(R^{\bullet}\pi_{*}(\oplus_{j=1}^{p}\mathcal
L_{k+j}))=(\frac{\lambda-D_{\theta}}{\delta(e)})^{|J_{e}|}\ .$$
to the virtual normal bundle and the fixed part of deformation of $\vec{y}$ will
be canceled with the automorphisms of line bundles $(L_{k+j}:1\leq j\leq p)$. 

Recall that the complete intersection $Y$ is cut off by the section $s:=\oplus_{b=1}^{c}s_{b}$ of direct sum
of the line bundles
$E=\oplus_{b=1}^{c}L_{\tau_{b}}$ on $X$ associated to the characters $\tau_{b}$. There is also an obstruction corresponding to the infinitesimal deformations of
$\vec{x}$ being moved away from $[AY^{ss}(\theta)/ G]\subset [W^{ss}(\theta)/
G]$, which contributes to the virtual normal bundle as the movable part of 
\begin{align*}
  e^{\C^{*}}\big(-(\oplus_{b}R^{\bullet}\pi_{*}\mathcal L_{\tau_{b}})\big)&=\frac{e^{\C^{*}}\big(R^{1}\pi_{*}\oplus_{b:\beta(L_{\tau_{b}})<0}\mathcal L_{\tau_{b}})\big)}{e^{\C^{*}}\big(R^{0}\pi_{*}\oplus_{b:\beta(L_{\tau_{b}})\geq 0}\mathcal L_{\tau_{b}}\big)}\\
                                                     &=\frac{\prod_{b:\beta(L_{\tau_{b}})< 0}\prod_{\beta(L_{\tau_b})<m<0}\big( c_1(L_{\tau_b})+\frac{\beta(L_{\tau_b})-m}{\delta(e)}(\lambda-D_{\theta})\big)}{\prod_{b:\beta(L_{\tau_{b}})\geq 0}\prod_{0\leq m\leq \beta(L_{\tau_b})}\big( c_1(L_{\tau_b})+\frac{\beta(L_{\tau_b})-m}{\delta(e)}(\lambda-D_{\theta})\big)}\ .
\end{align*}
Here $m$ are all integers.

One can see that the fixed part only comes from the
summand corresponding to the terms $b$ with $\beta(L_{\tau_{b}})\in \mathbb
Z_{\geq 0}$, for which there is one
dimensional $\C^{*}-$fixed piece to each $-R^{\bullet}\pi_{*}\mathcal L_{\tau_b}$, which contributes to the
virtual cycle of $\mathcal M_{e}$.

Now let's move the virtual cycle of $\mathcal M_{e}$. Consider $E_{\geq 0}:=\oplus_{b:\beta(L_{\tau_{b}})\in \mathbb Z_{\ geq
    0}}L_{\tau_{b}}$ the vector bundle $[Z^{ss}_{\beta}/G]$ and $s_{\geq 0}=\oplus_{b:\beta(L_{\tau_{b}})\in \mathbb
  Z_{\geq 0}} s_{b}$ a section of $E_{\geq 0}$. Using Lemma \ref{lem:comp-graph}, we can define the Gysin morphism
$$s^{!}_{E_{\geq 0},loc}:A_{*}([Z^{ss}_{\beta}/G])\rightarrow
A_{*}([Y^{ss}_{\beta}/G])$$
as the localized top Chern class~\cite[\S14.1]{MR732620}.

\begin{lemma}\label{lem:edgevir2}
  We have the following:
  $$[\mathcal M_{e}]^{\mathrm{vir}}=i_{\mathcal
      M_{e}}^{*}\big(s_{E_{\geq 0},loc}^{!}([Z^{ss}_{\beta}/G])\big)\ .$$
  Here $i_{\mathcal M_{e}}:\mathcal M_{e}\rightarrow [Y^{ss}_{\beta}/G]$ is the
  natural $\acute{e}tale$ morphism by forgetting gerbe structure.
\end{lemma}

\begin{remark}\label{rmk:viredge}
We note that $s_{\geq0}$ and $E_{\geq 0}$ can descend to a section and a vector
bundle over $[Y^{ss}_{\beta}/G]$, thus $s_{E_{\geq
    0},loc}^{!}([Z^{ss}_{\beta}/(G/\langle g^{-1}_{\beta}\rangle)])$ is also well defined. 
\end{remark}
\begin{proof}
Now combing the discussion above, the virtual cycle structure of $\mathcal
M_{e}$ solely comes automorphisms of line bundles $(\mathcal
L_{j})_{j=1}^{k}$, fixed part of deformations/obstructions of the sections
$\vec{x}$. Using the distinguished triangle \ref{eq:rel-pot} and \ref{eq:fac-pot} in \S\ref{subsec:fixed-loci1}, the $\C^{*}-$fixed part of the obstruction theory $\mathbb
E^{\mathrm{fix}}$ over $\mathcal M_{e}$ fits into the following distinguished triangle
$$\xymatrix{\mathbb E^{\mathrm{fix}}\ar[r] &\mathbb T_{[Z^{ss}_{\beta}/G]}|_{\mathcal
  M_{e}}\ar[r]^-{ds_{\geq
    0}}&E_{\geq 0}:=\oplus_{b:\beta(L_{\tau_{b}})\in \mathbb Z_{\geq 0}}L_{\tau_{b}}}\ .$$
Here $ds_{\geq 0}$ is the differential of the section $s_{\geq
  0}:=\oplus_{b:\beta(L_{\tau_{b}})\in \mathbb Z_{\geq 0}}s_{b}$, which is inside the
vector bundle $E_{\geq 0}$ over $\mathcal M_{e}$, and $\mathbb T_{[Z^{ss}_{\beta}/G]}|_{\mathcal
  M_{e}}$ is the pull back of the tangent bundle $\mathbb
T_{[Z^{ss}_{\beta}/G]}$ along the morphism
$$\mathcal M_{e}\rightarrow
\sqrt[a\delta(e)]{L_{-\theta}/[Z^{ss}_{\beta}/G]}\rightarrow
[Z^{ss}_{\beta}/G]\ ,$$
where the first arrow is the inclusion and the second arrow is the natural etale
morphism by forgetting gerbe structure.

By Lemma \ref{lem:comp-graph}, $\mathcal M_{e}$ is the zero loci of the section $s_{\geq
  0}$ of the
vector bundle $E_{\geq 0}:=\oplus_{b:\beta(L_{\tau_{b}})\in \mathbb Z_{\geq
    0}}L_{\tau_{b}}$ over $\sqrt[a\delta(e)]{L_{-\theta}/[Z^{ss}_{\beta}/G]}$.
Thus the virtual cycle $[\mathcal M_{e}]^{vir}$ with respect to the perfect
obstruction theory $\mathbb E^{\mathrm{fix}}$ can be obtained by the standard
model associated with the vector bundle $E_{\geq 0}$ over $\sqrt[a\delta(e)]{L_{-\theta}/[Z^{ss}_{\beta}/G]}$ and the section
$s_{\geq 0}$(c.f.\cite[\S Appendix]{pandharipande201413}). Now the Lemma
\ref{lem:edgevir2} is immediate as the distinguished triangle is also pull-back of a similar
distinguished triangle on $[Y^{ss}_{\beta}/G]$.
\end{proof}
We have the expression of virtual normal bundle from the movable part of curves,
line bundles and sections as follows:
\begin{align*}
  e^{\C^{*}}(N^{\mathrm{vir}})&=\frac{\prod_{\rho:\beta(L_{\rho})> 0}\prod_{0\leq i< \beta(L_{\rho})}(D_{\rho}+(\beta(L_{\rho})-i)\frac{\lambda-D_{\theta}}{\delta(e)})}{\prod_{\rho:\beta(L_{\rho})< 0}\prod_{\lfloor {\beta(L_{\rho})+1 } \rfloor\leq i<0}(D_{\rho}+(\beta(L_{\rho})-i)\frac{\lambda-D_{\theta}}{\delta(e)})}\cdot \frac{\delta(e)}{\lambda-D_{\theta}}\\
                              &\cdot \frac{\prod_{b:\beta(L_{\tau_{b}})< 0}\prod_{\beta(L_{\tau_b})<m<0}\big( c_1(L_{\tau_b})+\frac{\beta(L_{\tau_b})-m}{\delta(e)}(\lambda-D_{\theta})\big)}{\prod_{b:\beta(L_{\tau_{b}})\geq 0}\prod_{0\leq m< \beta(L_{\tau_{b}})}\big( c_1(L_{\tau_b})+\frac{\beta(L_{\tau_b})-m}{\delta(e)}(\lambda-D_{\theta})\big)}\cdot \prod_{m=1}^{\delta(e)-\beta(L_{\theta})-|J_e|}\bigg(\frac{m}{\delta(e)}(-D_{\theta}+\lambda)\bigg)\ .
\end{align*} 
We observe that, after taking the push-forward $i:\mathcal M_{e}\rightarrow \sqrt[a\delta(e)]{L_{-\theta}/I_{g_{\beta}}Y}$, the localization contribution from edge moduli with a basepoint yields:
\begin{lemma}\label{lem:edgeloc}
  $$i_{*}(Cont_{\mathcal M_{e}})=(\bar{i}_{})^{*}\left(\iota_{*}\frac{(\frac{z}{z^{|J_{e}|}}\mathbb I_{\beta}(q,z))|_{z=\frac{\lambda-D_{\theta}}{\delta(e)}}}{\prod_{m=1}^{\delta(e)-\beta(L_{\theta})-|J_e|}\big(\frac{m}{\delta(e)}(-D_{\theta}+\lambda\big)}\right)\ ,$$
  where $\bar{i}_{I_{g_{\beta}}Y}:\sqrt[a\delta(e)]{L_{-\theta}/I_{g_{\beta}}Y}\rightarrow \bar{I}_{g_{\beta}}Y:=[Y^{ss}_{\beta}/(G/\langle
  g_{\beta}\rangle)]$ is the
  natural structure map by forgetting gerbe structure and taking rigidification, $\iota$ is the involution
  of $\bar{I}_{\mu}Y$ obtained from taking the inverse of the band, and
  $\mathbb I_{\beta}$ is the coefficient of $q^{\beta}$ of $\mathbb
  I(q,0,z)$ defined in the introduction. 
\end{lemma}

\subsubsection{Edge contributions: without basepoint case}
The contribution from an edge without basepoint will not appear in the later
analysis in \S \ref{sec:rec-rel}. However we include the
discussion for this case here for completeness. The reader is encouraged to skip
this part in the first reading. In this case, $J_{e}$ is empty. Assume that the multiplicity at $q_{\infty}\in C_{e}$ is equal to
$(g,\mu_{r}^{\delta(e)})\in G \!\times\! \boldsymbol{\mu}_{r}$ and $a_{e}$ (or
$a$ for simplicity) is the
order of $g$. When $r$ is sufficiently large, due to Remark \ref{rmk:edge1},
$C_{e}$ must be isomorphic to $\P^{1}_{ar,a}$ where the ramification point
$q_{0}$ for which $z_{1}=0$ is isomorphic to $\mathbb B\boldsymbol{\mu}_{a}$,
and the ramification point $q_{\infty}$ for which $z_{2}=0$ must be a special
point and is isomorphic to $\mathbb B\boldsymbol{\mu}_{ar}$. The restriction of
degree $(\beta,\frac{\delta}{r})$ from $C$ to $C_{e}$ is equal to
$(0,\frac{\delta(e)}{r})$, which is equivalent to:
$$deg(L_{j}|_{C_{e}})=0 \quad \textit{for}\;1\leq j\leq k,\quad deg(N|_{C_{e}})=\frac{\delta(e)}{r}\ .$$

Recall that the inertia stack component $I_{g}Y$ of $I_{\mu}Y$ is isomorphic to
the quotient stack
$$[AY^{ss}(\theta)^{g}/G]\ .$$
We construct the edge moduli $\mathcal M_{e}$ as
$$\mathcal M_{e}:=\sqrt[a\delta(e)]{L_{-\theta}/I_{g}Y}\ ,$$
which is the root gerbe over the stack $I_{g}Y$ by taking the $a\delta(e)$th
root of the line bundle $L_{-\theta}$.

The root gerbe $\sqrt[a\delta(e)]{L_{-\theta}/I_{g}Y}$ admits a representation
as a quotient stack:
\begin{equation}\label{eq:edge0}
  [(AY^{ss}(\theta)^{g} \!\times\! \C^{*}) / (G \!\times\! \C^{*}_w)]\ ,
\end{equation}
where the (right) action is defined by:
$$(\vec{x},v)\cdot (g,w)=(\vec{x}\cdot g,\theta(g)vw^{a\delta(e)})\ ,$$
for all $(g,w)\in G \!\times\! \C^{*}_w$ and $(\vec{x},v)\in AY^{ss}(\theta)^{g}
\!\times\! \C^{*}$. Here $\vec{x}\cdot g$ is given by the action as in the
definition of $[AY/G]$, the torus $\C^{*}_{w}$ is isomorphic to $\C^{*}$ with
variable $w$. For any character $\rho$ of $G$, define a new character of $G
\!\times\! \C^{*}_w$ by composing the projection map $\mathrm{pr}_{G}:G
\!\times\! \C^{*}_{w} \rightarrow G$. By an abuse of notation, we will continue
to use the notation $\rho$ to mean the new character of $G \!\times\! \C^{*}_w
$. Then $\rho$ will determines a line bundle
$L_{\rho}:=[(AY^{ss}(\theta)^{g}\!\times\! \C^{*} \!\times\! \mathbb C_{\rho})
/(G \!\times\!\C^{*}_{w})]$ on $\sqrt[a\delta(e)]{L_{-\theta}/I_{g}Y}$ by the
Borel construction.

By virtue of the universal property of root gerbe, on $\mathcal
M_{e}=\sqrt[a\delta(e)]{L_{-\theta}/I_{g}Y}$, there is a universal line bundle
$\mathcal R$ that is the $a\delta(e)$th root of the line bundle $L_{-\theta}$.
The root bundle $\mathcal R$ is associated to the character
$$\mathrm{pr}_{\C^{*}}:G \!\times\! \C^{*}_w \rightarrow \C^{*}_{w}, \quad (g,w)\in G \!\times\!
\C^{*}_{w} \mapsto w\in \C^{*}_{w}$$ by the Borel construction. We have the
relation
$$L_{-\theta}=\mathcal R^{a\delta(e)}\ .$$
The coordinate functions $\vec{x}$ and $v$ of
$AY^{ss}(\theta)^{g}\times \C^{*}$ descents to be universal sections of line
bundles $\oplus_{\rho\in [n]}L_{\rho}$ and $L_{\theta}\otimes \mathcal
R^{\otimes a\delta(e)}$ over $\mathcal M_{e}$, respectively.

We will construct a universal family of $\C^{*}-$fixed quasimaps to $\P \mathfrak Y^{\frac{1}{r},p}$
of degree $(0,1^{\emptyset},\frac{\delta(e)}{r})$ over $\mathcal M_{e}$:
\begin{equation*}
  \xymatrix{
    \mathcal C_e := \P_{ar,a}(\mathcal R \oplus \mathcal O_{\mathcal M_{e}}) \ar^-{f}[r]\ar_{\pi}[d] & \P \mathfrak Y^{\frac{1}{r},p}\\
    \mathcal M_{e}:=\sqrt[a\delta(e)]{L_{-\theta}/I_{g}Y}.
  }
\end{equation*}
Then the universal curve $\mathcal C_{e}$ over $\mathcal M_{e}$ can be
represented as a quotient stack:
$$\mathcal C_{e}=[(AY^{ss}(\theta)^{g} \!\times\! \C^{*} \!\times\! U )/ (G \!\times\! \C^{*}_w
\!\times\! T)]\ ,$$ where $T=\{(t_{1},t_{2})\in (\C^{*})^{2}|\quad
t_{1}^{a}=t_{2}^{ar}\}$. The (right) action is defined by:
$$(\vec{x},v,x,y)\cdot (g,w,(t_{1},t_{2}))=(\vec{x}\cdot g,\theta(g)vw^{a\delta(e)},wt_{1}x,t_{2}y)\ ,$$
for all $(g,w,(t_{1},t_{2}))\in G \!\times\! \C^{*}_w \!\times\! T$ and
$(\vec{x},v,(x,y))\in AY^{ss}(\theta)^{g} \!\times\! \C^{*} \!\times\! U$. Then
$\mathcal C_{e}$ is a family of orbifold $\P_{ar,a}$ parameterized by $\mathcal
M_{e}$.

There are two standard characters $\chi_{1}$ and $\chi_{2}$ of $T$:
$$\chi_{1}:(t_{1},t_{2})\in T\mapsto t_{1}\in \C^{*}, \quad
\chi_{2}:(t_{1},t_{2})\in T\mapsto t_{2}\in \C^{*}\ .$$ We can lift them to be
new characters of $G \!\times\! \C^{*}_w \!\times\! T$ by composing the
projection map $\mathrm{pr}_{T}:G \!\times\! \C^{*}_w \!\times\! T\rightarrow
T$. By an abuse of notation, we continue to use $\chi_{1},\chi_{2}$ to denote
the new characters. Then $\chi_{1},\chi_{2}$ defines two line bundles
$$M_{1}:=(AY^{ss}(\theta)^{g} \!\times\! \C^{*} \!\times\! U) \times _{G \!\times\! \C^{*}_w
\!\times\! T}\; \C_{\chi_{1}}$$ and
$$M_{2}:=(AY^{ss}(\theta)^{g} \!\times\! \C^{*} \!\times\! U)\times _{G \!\times\! \C^{*}_w
\!\times\! T} \;\C_{\chi_{2}}$$ on $\mathcal C_{e}$ by the Borel construction, respectively. We
have the relation $M_{1}^{\otimes a}=M_{2}^{\otimes ar}$ on $\mathcal C_{e}$.
The universal map $f$ from $\mathcal C_{e}$ to $\P \mathfrak Y^{\frac{1}{r},p}$
can be constructed as follows: let
$$\tilde{f}: AY^{ss}(\theta)^{g} \!\times\! \C^{*} \!\times\! U\rightarrow AY \!\times\! U$$
be the morphism defined by:
\begin{equation}
  \begin{split}
    (\vec{x},v,x,y)&\in AY^{ss}(\theta)^{g} \!\times\! \C^{*} \!\times\! U \mapsto\\
    &((x_{1},\cdots,x_{n}),v^{-1}x^{a\delta(e)},y^{a\delta(e)})\in AY \!\times\!
    U\
  \end{split}
\end{equation}
Then $\tilde{f}$ is equivariant with respect to the group homomorphism from $G
\!\times\! \C^{*}_{w} \!\times\! T $ to $G \!\times\! \C^{*}$ defined by:
\begin{equation}
  \begin{split}
    (g,w,(t_{1},t_{2}))&\in G \!\times\! \C^{*}_w \!\times\! T \mapsto\\
    &(g((t_{1}^{-1}t_{2}^{r})^{p_{1}},\cdots,(t_{1}^{-1}t_{2}^{r})^{p_{k}}),t_{2}^{a\delta(e)})\in
    G \!\times\! \C^{*}\ ,
  \end{split}
\end{equation}
where the tuple $(p_{1},\cdots, p_{k})\in \mathbb N^{k}$ satisfies that
$g=(\mu_{a}^{p_{1}},\cdots,\mu_{a}^{p_{k}})\in G$. Note $\tilde{f}$ is well
defined for $\chi^{-1}_{1}\chi_{2}^{r}$ is a torsion character of $T$ of order
$a$. The above construction gives the universal morphism $f$ from $\mathcal
C_{e}$ to $\P \mathfrak Y^{\frac{1}{r},p}$ by descent.

Now we define a (quasi left) $\C^{*}-$action on $\mathcal C_{e}$ such that $f$ is
$\C^{*}-$equivariant. The $\C^{*}$-action on $\mathcal C_{e}$ is induced by
the $\C^{*}-$action on $AY^{ss}(\theta)^{g} \!\times\! \C^{*} \!\times\! U$:
$$m:\C^{*}\times AY^{ss}(\theta)^{g} \!\times\! \C^{*} \!\times\! U\rightarrow AY^{ss}(\theta) \!\times\! \C^{*} \!\times\! U\ ,$$
$$t\cdot (\vec{x},v,(x,y))= (\vec{x},v,(x,t^{\frac{-1}{ar\delta(e)}}y))\ .$$
Note then $\pi$ is $\C^{*}$-equivariant map, where $\mathcal M_{e}$ is equipped with
the trivial $\C^{*}$-action. By the universal property of the projectivized bundle $\mathcal C_e$ over $\mathcal
M_{e}$, one has a tautological section
$$ (x, y) \in H^0(\mathcal C_{e},(M_{1} \otimes \mathcal \pi^{*}\mathcal
 R) \oplus (M_{2}\otimes \C_{\frac{-1}{ar\delta(e)}}))\ ,$$
which is also a $\C^{*}-$invariant section. 

Now we can check that $f$ is a $\C^{*}-$equivariant morphism from $\mathcal
C_{e}$ to $\P \mathfrak Y^{\frac{1}{r},p}$ with respect to the $\C^{*}-$actions
for $\mathcal C_{e}$ and $\P \mathfrak Y^{\frac{1}{r},p}$. Using Remark
\ref{rmk:equivmap}, $f$ is given by the following data:
\begin{enumerate}\label{equivedge0} 
\item $k+p+1$ $\C^{*}-$equivariant line bundles $\mathcal C_{e}$:
  $$\mathcal L_{j} := \pi^*L_{\pi_{j}} \otimes (M^{\vee}_{1}\otimes M_{2}^{\otimes r})^{p_{j}},
  1\leq j\leq k\ ,$$
$$\mathcal L_{k+j}:=\mathbb C,\; 1\leq j\leq p$$
  and
  $$\mathcal N:=M_{2}^{a\delta(e)}\otimes \mathbb C_{\frac{-\lambda}{r}}\ ,$$
  where $(L_{\pi_{j}})_{1\leq j\leq k}$ are the standard $\C^{*}$-equivariant
  line bundles on $\mathcal M_{e}$ by the Borel contribution, $M_{1},M_{2}$ are
  the standard $\C^{*}$-equivariant line bundles on $C_{e}$ by the Borel
  construction;
  \item a universal section
  \begin{equation}
    \begin{split}
      (\vec{x},\vec{y},(\zeta_1, \zeta_2)) :=&
      ((x_{1},\cdots,x_{n}),1^{p},(v^{-1}x^{a\delta(e)}, y^{a\delta(e)}))\\
      &\in H^0(\mathcal C_{e}, \oplus_{i=1}^{n} \mathcal L_{\rho_{j}}\oplus (\oplus_{j=1}^{p}\mathcal L_{k+j})\oplus
      (\mathcal L_{-\theta_{p}}\otimes \mathcal N^{\otimes r}\otimes \mathbb
      C_{\lambda})\oplus \mathcal N)^{\C^{*}}\ .
    \end{split}
  \end{equation}
\end{enumerate}
Use the similar analysis as in the last section, we have that $[\mathcal
M_{e}]^{vir}=[\mathcal M_{e}]$ and the virtual normal bundle
$$e^{\C^{*}}(N^{\mathrm{vir}})=\prod_{m=1}^{\delta(e)}\bigg(\frac{m}{\delta(e)}(-D_{\theta}+\lambda)\bigg)\ .$$

when $r$ is a sufficiently large prime with the exception that the movable part of
infinitesimal automorphisms of $C_{e}$ contributes
$$\frac{\delta(e)}{\lambda-D_{\theta}}$$
to the virtual normal bundle when $a=1$.

\subsubsection{Node contributions}\label{subsubsec:node-cntr2}
The deformations in $Q^{\widetilde{\theta}}_{0,\vec{m}}(\P \mathfrak
Y^{\frac{1}{r},p},(\beta,\frac{\delta}{r}))$ smoothing a node contribute to the
Euler class of the virtual normal bundle as the first Chern class of the tensor
product of the two cotangent line bundles at the branches of the node. For nodes
at which a component $C_e$ meets a component $C_v$ over the vertex $0$, this
contribution is
\begin{equation*}
  \frac{\lambda - D_{\theta}}{a\delta(e)} - \frac{\bar{\psi}_v}{a} \ ;
\end{equation*}
for nodes at which a component $C_e$ meets a component $C_v$ over the vertex
$\infty$, this contribution is
\begin{equation*}
  \frac{-\lambda + D_{\theta }}{ar\delta(e)} -\frac{\bar{ \psi}_v}{ar} \ ;
\end{equation*}
for nodes at which two edge component $C_e$ meets with a vertex $v$ over $0$,
the node-smoothing contribution is
\begin{equation*}
  \frac{\lambda - D_{\theta}}{a\delta(e)} + \frac{\lambda - D_{\theta }}{a\delta(e')} \ .
\end{equation*}
The nodes at which
two edge component $C_e$ meets with a vertex $v$ over $\infty$ will not occur using a similar argument in \cite[Lemma
6]{Janda2017} when $r$ is sufficiently large. 


As for the node contributions from the normalization exact sequence of relative
obstruction theory \eqref{pot}, each node
$q$ (specified by a vertex $v$) contributes the inverse of Euler class of
\begin{equation}
  \label{eq:node2}
  (R^0\pi_*(\mathcal L_{\theta}^{\vee} \otimes \mathcal N^{\otimes r} \otimes \C_{\lambda})|_{q})^{\mathrm{mov}} \oplus (R^0\pi_*\mathcal N|_{q})^{\mathrm{mov}}
\end{equation}
to the Euler class of the virtual normal bundle. Note here we use the fact that
the node can't be a base point, which implies that $\mathcal
L_{\theta_{p}}|_{q}=\mathcal L_{\theta}|_{q}$.

In the case where $j(v) = 0$,  $z_2|_{q}=1$ gives a
trivialization of $\mathcal N$ at $q$. Thus, the second factor in \eqref{eq:node2}
is trivial, while the Euler class of the first factor equals
$$\frac{1}{\lambda - D_{\theta}}\ .$$

In the case where $j(v) = \infty$, $z_1|_{q}=1$ gives
a trivialization of the fiber $(\mathcal L_{\theta}^{\vee}\otimes \mathcal N^{\otimes
  r}\otimes \mathbb C_{\lambda})|_{q}$. Hence we have
$\mathcal N|_{q}\cong \mathcal L_{\theta}^{\frac{1}{r}}|_{q}\otimes \mathbb
C_{-\frac{\lambda}{r}}$, this implies that it $R^{0}\pi_{*}(\mathcal N|_{q})=0$
because of the nontrivial stacky structure when $r$ is sufficiently large. Thus there is no localization contribution from the normalization sequence at the node over $\infty$.

\subsection{Total localization contributions}
For each decorated graph $\Gamma$, denote the moduli $F_{\Gamma}$ to be the
fiber product
$$\prod_{v:j(v)=0}\mathcal M_{v}\times_{\bar{I}_{\mu}Y}\prod_{e\in E}\mathcal
M_{e}\times_{\bar{I}_{\mu}\sqrt[r]{L_{\theta}/Y}}
\prod_{v:j(v)=\infty} \mathcal M_{v}
$$ of the following
diagram:
$$\xymatrix{
  F_{\Gamma}\ar[r]\ar[d] &\prod\limits_{v:j(v)=0} \mathcal M_{v}\times
  \prod\limits_{e\in E}\mathcal M_{e}\times \prod\limits_{v:j(v)=\infty}
  \mathcal M_{v}
  \ar[d]^-{ev_{h_{0}},ev_{p_{0}}, ev_{p_{\infty}}, ev_{h'_{\infty}}}\\
  \prod\limits_{E}(\bar{I}_{\mu}Y \times
  \bar{I}_{\mu}\sqrt[r]{L_{\theta}/Y})\ar[r]^-{(\Delta\times
    \Delta^{\frac{1}{r}})^{|E|}} &\prod\limits_{E}(\bar{I}_{\mu}Y)^{2}\times
  (\bar{I}_{\mu}\sqrt[r]{L_{\theta}/Y})^{2}\ , }
$$
where $\Delta=(id,\iota)$(resp. $\Delta^{\frac{1}{r}}=(id,\iota)$) is the
diagonal map of $\bar{I}_{\mu}Y$ (resp. $\bar{I}_{\mu}\sqrt[r]{L_{\theta }/Y}$),
where $\iota$ is the involution on $\bar{I}_{\mu}Y$ (resp. $\bar{I}_{\mu}
\sqrt[r]{L_{\theta}/Y}$) by taking the inverse of the band of gerbe structure. Here when $v$ is a stable
vertex, the vertex moduli $\mathcal M_{v}$ is described in
\ref{subsubsec:ver-contr1}; when $v$ is
an unstable vertex over $0$, we treat
$\mathcal M_{v}:=\bar{I}_{m(h)}Y$ with the identical virtual cycle, where $h$ is
the half-edge incident to $v$; when $v$ is an unstable vertex over $\infty$, We treat
$\mathcal M_{v}:=\bar{I}_{m(h)}\sqrt[r]{L_{\theta}/Y}$ with the identical
virtual cycle, where $h$ is the half-edge incident to $v$.

We define $[F_{\Gamma}]^{\mathrm{vir}}$ to be:
\begin{equation*}
\begin{split}
\prod_{v:j(v)=0}[\mathcal M_{v}]^{\mathrm{vir}}
\times_{\bar{I}_{\mu}Y}\prod_{e\in E}[\mathcal
M_{e}]^{\mathrm{vir}}\times_{\bar{I}_{\mu}\sqrt[r]{L_{\theta}/Y}}
\prod_{v:j(v)=\infty}[\mathcal M_{v}]^{\mathrm{vir}}\ .
\end{split}
\end{equation*}
where the fiber product over $\bar{I}_{\mu}Y$ and
$\bar{I}_{\mu}\sqrt[r]{L_{\theta}/Y}$ imposes that the evaluation maps at the
two branches of each node (Here we adopt the convention that a node can link a unstable vertex and an edge.) agree.
Then the contribution of decorated graph $\Gamma$ to the virtual localization
is:
\begin{equation}\label{eq:auto-loc1}
Cont_{\Gamma}=\frac{\prod_{e\in
    E}a_{e}}{\text{Aut}(\Gamma)}(\iota_{\Gamma})_{*}\left(\frac{[F_{\Gamma}]^{\mathrm{vir}}}{e^{\C^{*}}(N^{\mathrm{vir}}_{\Gamma})}\right)\ .
\end{equation}
Here $\iota_{F}:F_{\Gamma}\rightarrow Q^{\widetilde{\theta}}_{0,\vec{m}}(\P \mathfrak
Y^{\frac{1}{r},p},(\beta,\frac{\delta}{r}))$ is a finite etale map of degree
$\frac{\text{Aut}(\Gamma)}{\prod_{e\in E}a_{e}}$ into the
corresponding $\C^{*}$-fixed loci in twisted graph space.
The virtual normal bundle
$e^{\C^{*}}(N^{\mathrm{vir}}_{\Gamma})$ is the
product of virtual normal bundles from vertex contributions, edge contributions
and node contributions. 

\begin{remark}
Let $u$ be a polynomial on $c_{1}(L_{\pi_{1}}),\cdots,c_{1}(L_{\pi_{k}})$. In the contribution
from the graph $\Gamma$, assume that
$j\in J_{e}$ for some edge $e$, then $\hat{ev}_{j}|_{F_{\Gamma}}$ factors through the
projection from $F_{\Gamma}$ to $\mathcal M_{e}$. By abusing notations, we
denote $\hat{ev}_{j}:\mathcal M_{e}\rightarrow \mathfrak Y$. Thus when we want to
apply virtual localization to
$\prod_{j=1}^{p}\hat{ev}^{*}_{j}(u(c_{1}(L_{\pi_{k}})))$, we replace $\mathbb I_{\beta}(q,z)$ in the edge contribution in Lemma \ref{lem:edgeloc} by
$u(c_{1}(L_{\pi_{k}})+\beta(L_{\pi_{k}})z)^{|J_{e}|}\mathbb I_{\beta}(q,z)$. Indeed, use the setting in
\S\ref{subsubsec:edgebase}, let $\underline{ev}:=\mathrm{pr}_{r,p}\circ
ev:\mathcal C_{e}\rightarrow \mathfrak Y$, where $\mathrm{pr}_{r,p}:\mathbb P
\mathfrak Y^{\frac{1}{r},p}\rightarrow \mathfrak Y$ is the
natural projection map. Then we have
$$\underline{ev}^{*}(L_{\tau})=\pi^{*}(L_{\tau}\otimes \mathcal
R^{a\beta(L_{\tau})})\otimes \mathcal O_{\mathcal C_{e}}(ar\beta(L_{\tau}))$$
for any character $\tau$ of $G$. Let $D_{0}$ be the zero section of $\mathcal
C_{e}$ over $\mathcal M_{e}$ given by $x=0$. Then
$\hat{ev}_{j}=\underline{ev}|_{D_{0}}$. Note that $\mathcal O_{\mathcal
  C_{e}}(1)|_{D_{0}}=\mathbb C_{\frac{1}{ar\delta(e)}}$. Using the fact
$\mathcal R^{a\delta(e)}=L_{-\theta}$, we have
$c_{1}(\hat{ev}^{*}(L_{\tau}))=c_{1}(L_{\tau})+\frac{\beta(L_{\tau})(\lambda-D_{\theta})}{\delta(e)}$. 
\end{remark}

\section{Master space II}\label{sec:masterspace-inf}
\subsection{Construction of master space II}\label{subsec:Construction of the master space at infinity sector}
Fix two different primes $r,s\in \mathbb N$, let $\theta$ be as in previous
section, let $\P Y_{r,s}$ be the root stack
of the $\P^{1}$ bundle $\P_{Y}(\mathcal O(-D_{\theta})\oplus \mathcal O)$ over
$Y$ by taking the $\operatorname{s-th}$ root of the zero section ($z_{1}=0$) and
$\operatorname{r-th}$ root of the infinity section ($z_{2}=0$). Then the zero
section $\mathcal D_{0}\subset \P Y_{r,s}$ is isomorphic to
$\sqrt[s]{L_{-\theta}/Y}$, and the infinity section $\mathcal D_{\infty}\subset
\P Y_{r,s}$ is isomorphic to $\sqrt[r]{L_{\theta}/Y}$.

We give a more concrete presentation of $\P Y_{r,s}$ as a quotient stack:
$$\P Y_{r,s}=[(\C^{*} \!\times\! AY^{ss}(\theta) \!\times\! U)/ 
(G \!\times\! \C^{*}_{\alpha} \!\times\! \C^{*}_{t}) ]\ ,$$ where the (right) $G
\!\times\! \C^{*}_{\alpha} \!\times\! \C^{*}_{t}$-action on $\C^{*} \!\times\!
AY^{ss}(\theta) \!\times\! U$ is given by:
$$(u,\vec{x},z_{1},z_{2})\cdot(g,\alpha,t)=( \alpha^{-s}\theta(g)^{-1}t^{r}
u, \vec{x}g,\alpha z_{1},tz_{2})\ ,$$ for $(g,\alpha,t)\in G
\!\times\!\C^{*}_{\alpha} \!\times\! \C^{*}_{t} ,\text{and}\; (u,
\vec{x},z_{1},z_{2})\in \C^{*} \!\times\! AY^{ss}(\theta)\!\times\! U$. Here
$U=\C^{2}\backslash \{0\}$. This quotient stack presentation of $\P Y_{r,s}$ comes from the root
stack construction in \cite[Appendix B]{Dan_Abramovich_2008} after
some simplification.

The inertia stack $I_{\mu} \P Y_{r,s} $ of $\P Y_{r,s}$ admits a decomposition
$$I_{\mu} \P Y\sqcup\bigsqcup_{i=1}^{s-1} \sqrt[s]{L_{-\theta}/Y} \sqcup\bigsqcup_{j=1}^{r-1}
\sqrt[r]{L_{\theta}/Y}\ .$$

Let $(\vec{x},(g,\alpha,t))$ be a point of the inertia stack $I_{\mu}\P
Y_{r,s}$, if the point $(\vec{x},(g,\alpha,t))$ appears in the first factor of
the decomposition above, then the automorphism $(g,\alpha,t)$ lies in $G
\!\times\! \{1\}\!\times\! \{1\}$; if the point $(\vec{x},(g,\alpha,t))$ occurs
in the second factor of the decomposition above, then the automorphism
$(g,\alpha,t)$ lies in $G \!\times\! \{\mu_{s}^{i}:1\leq i\leq s-1\} \!\times\!
\{1\}\subset G \!\times\! \C^{*}_{\alpha} \!\times\! \C^{*}_{t}$, and the point
$\vec{x}$ is in the zero section $\mathcal D_{0}$ defined by
$z_{1}=0$; finally if the point $(\vec{x},(g,\alpha,t))$ shows in the third
factor of the decomposition above, then the automorphism $(g,\alpha,t)$ lies in
$G \!\times\! \{1\} \!\times\! \{\mu_{r}^{j}: 1\leq i\leq r-1\}\subset G
\!\times\! \C^{*}_{\alpha} \!\times\! \C^{*}_{t}$, and $\vec{x}$
is in the infinity section $\mathcal D_{\infty}$ defined by $z_{2}=0$. Here
$\mu_{r}=\mathrm{exp}(\frac{2\pi \sqrt{-1}}{r})\in \C^{*}$ and
$\mu_{s}=\mathrm{exp}(\frac{2 \pi \sqrt{-1}}{s})\in \C^{*}$.

Fix $(g, \alpha,t)\in G \!\times\! \boldsymbol{\mu}_{s} \!\times\!
\boldsymbol{\mu}_{r}$, we will use the notation $\bar{I}_{(g,\alpha,t)}\P
Y_{r,s}$ to mean the rigidified inertia stack component of $\bar{I}_{\mu}\P
Y_{r,s} $ which has automorphism $(g,\alpha,t)$. Note if $\alpha$ and $t$ are
not equal to $1$ simultaneously, then the corresponding rigidified inertia stack
component is empty.

Let $\mathcal K_{0,m}(\P Y_{r,s}, (d,\frac{\delta}{r}))$ be the moduli stack of
$m$-pointed twisted stable maps to $\P Y_{r,s}$ of degree $(d,\frac{\delta}{r})$.
More concretely, More concretely,
\[\mathcal K_{0,m}(\P Y_{r,s},(d,\frac{\delta}{r})) = \{(C;q_1, \ldots,
  q_m; L_1,\cdots,L_{k}, N_{1},N_{2}; u,\vec{x}:= (x_1, \ldots, x_m), z_1, z_2)\},\]
where $(C; q_1, \ldots, q_m)$ is a $m$-pointed prestable balanced twisted curve
of genus $0$ with nontrivial isotropy only at special points, $(L_j:1\leq j\leq
k)$ and $N_{1},N_{2}$ are orbifold line bundles on $C$ with
\begin{equation*}
  \label{eq:TGSdegrees3}
  deg([\vec{x}])=d\in Hom(Pic(\mathfrak Y),\mathbb Q), \; \; \; \mathrm{deg}(N_{2})=\frac{\delta}{r},
\end{equation*}
and
\[(u, (\vec{x}, \vec{z})) := (u, x_1, \ldots, x_n, z_1, z_2) \in
  \Gamma\left(\big((N^{\vee}_{1})^{\otimes s}\!\otimes\! L_{-\theta}\!\otimes\!
    N_{2}^{\otimes r}\big)\oplus \bigoplus_{i=1}^{n} L_{\rho_{i}} \oplus N_{1}\oplus
    N_{2}\right).\] Here, for $1\leq i\leq n$, the line bundle $L_{\rho_{i}}$ is
equal to
$$\otimes_{j=1}^{k}L_{j}^{m_{ij}}\ ,$$
where $(m_{ij})_{1\leq i\leq n,\;1\leq j\leq k }$ is given by the relation
$\rho_{i}=\sum _{j=1}^{k}m_{ij}\pi_{j}$. The same construction applies to the
line bundle $L_{-\theta}$ on $C$. Note here $\delta$ is an integer when
$\mathcal K_{0,m}(\P Y_{r,s},(d,\frac{\delta}{r}))$ is nonempty as $N_{2}^{\otimes
  r}$ is the pullback of some line bundle on the coarse moduli curve
$\underline{C}$.

We require this data to satisfy the following conditions:
\begin{itemize}
\item {\it Representability}: For every $q \in C$ with isotropy group $G_q$, the
  homomorphism $\mathbb B G_q \rightarrow \mathbb B(G \!\times\! \C^{*}_{\alpha} \!\times\! \C^{*}_{t})$ given by the restriction of line bundles
  $(L_{j}:1\leq j\leq k)$ and $N_{1},N_{2}$ on $q$ is representable.
\item {\it Nondegeneracy}: The sections $z_1$ and $z_2$ never simultaneously
  vanish, and we have
  \begin{equation}
    \label{eq:stab3}
    \text{ord}_q(\vec{x}) = 0.
  \end{equation}
  for all $q\in C$. Furthermore, the section $u$ never vanish, so we have
  $(N^{\vee}_{1})^{\otimes s}\!\otimes\! L_{-\theta}\!\otimes\! N_{2}^{\otimes
    r}\cong \mathcal O_{C}$.
\item {\it Stability}: the map
  $[u,\vec{x},\vec{z}]:(C,q_{1},\cdots,q_{m})\rightarrow \P Y_{r,s}$ satisfies
  the usual stability condition defined by a twisted stable map; 
\item{\it Vanishing}: The image of $[\vec{x}]:C\rightarrow \mathfrak X$ lies in
  $\mathfrak Y$.
\end{itemize}

Let $\vec{m}=(v_{1},\cdots,v_{m})\in (G \!\times\! \boldsymbol{\mu}_{s}
\!\times\! \boldsymbol{\mu}_{r})^{m}$, we will denote $\mathcal K_{0,\vec{m}}(\P
Y_{r,s},(d,\frac{\delta}{r}))$ to be:
$$\mathcal K_{0,m}(\P Y_{r,s},(d,\frac{\delta}{r}))\cap ev_{1}^{-1}(\bar{I}_{v_{1}}\P
 Y_{r,s})\cap \cdots \cap ev_{m}^{-1}(\bar{I}_{v_{m}}\P Y_{r,s} ) \ ,$$
where 
\begin{equation*}
  ev_i: \mathcal K_{0,\vec{m}}(\P Y_{r,s},(d,\frac{\delta}{r})) \rightarrow \bar{I}_{\mu}\P Y_{r,s},
\end{equation*}
are natural evaluation maps as before, by evaluating the sections $(u,\vec{x},\vec{z})$ at $q_i$.
\subsection{$\C^{*}$-action and fixed loci}
Define a (left) $\C^{*}$-action on $\C^{*} \!\times\! AY^{ss}(\theta) \!\times\!
U$ given by
$$t\cdot (u, \vec{x},(z_{1},z_{2}))=(tu, \vec{x},(z_{1},z_{2}))\ .$$
This action descends to be a (left) $\C^{*}$-action on $\P Y_{r,s}$, which
induces a $\C^{*}$-action on $\mathcal K_{0,\vec{m}}(\P
Y_{r,s},(d,\frac{\delta}{r}))$. The reason why we define this action is that
this definition lifts the $\C^{*}$-action on $\P Y$ defined in \S\ref{subsec:Construction of the master space} along the canonical structure map
$\pi_{r,s}: \P Y_{r,s}\rightarrow \P Y$. We will let $\lambda$ to be equivariant
parameter corresponding to the action of weight $1$.

First we state a similar criteria for maps to $\P Y_{r,s}$ to be
$\C^{*}$-equivariant as in Remark \ref{rmk:equivmap}.
\begin{remark}\label{rmk:equivmap3}
  Given a stack $S$ over $Spec(\mathbb C)$ equipped with a (left)
  $\C^{*}$-action, then a $\C^{*}$-equivariant morphism from $S$ to $\P Y_{r,s}$
  is equivalent to that the following: there exists $k+2$ $\C^{*}$-equivariant
  line bundles on $S$
  $$L_{1},\cdots, L_{k},N_{1},N_{2}$$
  together with sections invariant under the $\C^{*}$-action;
  \begin{equation}
    \begin{split}
      (u, \vec{x}, \vec{z}) := &(u, x_1, \ldots, x_n, z_1, z_2)\\
      &\in \Gamma\left(\big((N^{\vee}_{1})^{\otimes s}\otimes
        L_{-\theta}\otimes N_{2}^{\otimes r}\otimes \C_{\lambda}\big)\oplus
        \bigoplus_{i=1}^{n} L_{\rho_{i}} \oplus N_{1} \oplus
        N_{2}\right)^{\C^{*}}.
    \end{split}
  \end{equation}
  Here the line bundles $L_{\rho_{i}}$ and $L_{-\theta}$ are defined similarly
  as before. The sections should satisfy the vanishing condition imposed by the
  affine cone of $Y$ in the definition of stable maps to $\P Y_{r,s}$.
\end{remark}

Fix a nonzero degree $\beta\in \mathrm{Eff}(W,G,\theta)$, and two tuples of
nonnegative integers
$$(\delta_{1},\cdots,\delta_{m})\in \mathbb N^{m}$$
and
$$(\delta'_{1},\cdots,\delta'_{m})\in \mathbb N^{m}\ .$$
Consider the tuple of multiplicities $\vec{m}=(v_{1},\cdots,v_{m})\in (G\times
\boldsymbol{\mu}_{r})^{m}$, where
$v_{i}=(g_{i},\mu^{\delta'_{i}}_{s},\mu_{r}^{\delta_{i}})$, we will denote
$\mathcal K_{0,\vec{m}}(\P Y_{r,s},(\beta,\frac{\delta}{r}))$ to be
$$\bigsqcup_{\substack{d\in \mathrm{Eff}(AY,G,\theta)\\ (i_{\mathfrak
      Y})_{*}(d)=\beta }}\mathcal K_{0,\vec{m}}(\P
Y_{r,s},(d,\frac{\delta}{r}))\ ,$$ where $i_{\mathfrak Y}: \mathfrak
Y\rightarrow \mathfrak X$ is the inclusion morphism. Thus $\mathcal
K_{0,\vec{m}}(\P Y_{r,s},(\beta,\frac{\delta}{r}))$ inherits a $\C^{*}$-action
as above.

We will follow the presentation
of~\cite{clader17_higher_genus_wall_cross_gauged,
  clader17_higher_genus_quasim_wall_cross_via_local} to describe the virtual localization for $\mathcal K_{0,\vec{m}}(\P
Y_{r,s}, (\beta,\frac{\delta}{r}))$ similar to
$Q^{\widetilde{\theta}}_{0,\vec{m}}(\P \mathfrak
Y^{\frac{1}{r},p},(\beta,\frac{\delta}{r}))$, but the edge contribution is easier
to analyze as there is no basepoint occurring for twisted stable maps.

We index the components of $\C^{*}-$fixed loci of $\mathcal K_{0,\vec{m}}(\P
Y_{r,s},(\beta,\frac{\delta}{r}))$ by
decorated graphs. A decorated graph 
$\Gamma$ consists of vertices, edges, and $m$ legs with the following decorations it:
\begin{itemize}
\item Each vertex $v$ is associated with an index $j(v) \in \{0, \infty\}$, and a degree
  $\beta(v) \in \mathrm{Eff}(W,G,\theta)$.
\item Each edge $e=\{h,h'\}$ is equipped with a degree $\delta(e) \in \mathbb{N}$.
\item Each half-edge $h$ and each leg $l$ has an element $m(h)$ or $m(l)$ in $G \!\times\! \boldsymbol{\mu}_{s}\times \boldsymbol{\mu}_{r}$.
\item The legs are labeled with the numbers $\{1, \ldots, m\}$.
\end{itemize}
By the ``valence" of a vertex $v$, denoted $\text{val}(v)$, we mean the total
number of incident half-edges, including legs.

The fixed locus in $\mathcal K_{0,\vec{m}}(\P Y_{r,s},(\beta,\frac{\delta}{r}))$
indexed by the decorated graph $\Gamma$ parameterizes stable map of the following
type:
\begin{itemize}
\item Each edge $e$ corresponds to a genus-zero component $C_e$ on which
  $\deg(N_{2}) = \frac{\delta(e)}{r}$ for some integer $\delta(e)\in \Z_{>0}$,
  where there are two distinguished points $q_{0}$ and $q_{\infty}$ on $C_{e}$
  satisfying that $z_2|_{q_{\infty}}=0$ and $z_1|_{q_{0}}=0$, respectively. We
  call them the ``ramification points". Note that we have
  $deg(L_{j}|_{C_{e}})=0$ for all $1\leq j\leq k$.
\item Each vertex $v$ for which $j(v) = 0$ (with unstable exceptional cases noted below) corresponds to a maximal sub-curve $C_v$ of $C$ over which $z_1
  \equiv 0$, then the restriction of $(C; q_1, \ldots, q_m;
  L_1,\cdots,L_{k};\vec{x})$ to $C_{v}$ defines a twisted stable map in
  $$\mathcal
  K_{0,val(v)}(\sqrt[s]{L_{-\theta}/Y},\beta(v)):=\bigsqcup_{\substack{d\in
      \mathrm{Eff}(AY,G,\theta)\\(i_{\mathfrak Y})_{*}(d)=\beta(v)}}\mathcal
  K_{0,val(v)}(\sqrt[s]{L_{-\theta}/Y},d)\ .$$
  Each vertex $v$ for which $j(v)
  = \infty$ (again with unstable exceptions) corresponds to a maximal sub-curve
  for which $z_2 \equiv 0$, then the restriction of $(C; q_1, \ldots, q_m;
  L_1,\cdots,L_{k};\vec{x})$ to $C_{v}$ defines a twisted stable map in
  $$\mathcal
  K_{0,val(v)}(\sqrt[r]{L_{\theta}/Y},\beta(v)):=\bigsqcup_{\substack{d\in
      \mathrm{Eff}(AY,G, \theta)\\(i_{\mathfrak Y})_{*}(d)=\beta(v)}}\mathcal
  K_{0,val(v)}(\sqrt[r]{L_{\theta}/Y},d)\ .$$
  The label $\beta(v)$ denotes the
  degree coming from the restriction $[x]|_{C_{v}}:C_{v}\rightarrow \mathfrak X$. Note here we
  count the degree $\beta(v)$ in $\mathrm{Eff}(W,G,\theta)$, but not in
  $\mathrm{Eff}(AY,G,\theta)$.
\item A vertex $v$ is {\it unstable} if stable twisted maps of the type
  described above do not exist (where, as always, we interpret legs as marked
  points and half-edges as half-nodes). In this case, $v$ corresponds to a
  single point of the component $C_e$ for each adjacent edge $e$, which may be a
  node at which $C_e$ meets $C_{e'}$, a marked point of $C_e$, or an unmarked
  point.
\item The index $m(l)$ on a leg $l$ indicates the rigidified inertia stack
  component $\bar{I}_{m(l)}\P Y_{r,s}$ of $\P Y_{r,s}$ on which the marked point
  corresponding to the leg $l$ is evaluated, this is determined by the
  multiplicity of $L_{1},\cdots, L_{k}, N_{1},N_{2}$ at the corresponding marked
  points.
\item A half-edge $h$ incident to a vertex $v$ corresponds to a node at which
  components $C_e$ and $C_v$ meet, and $m(h)$ indicates the rigidified inertia
  component $\bar{I}_{m(h)}\P Y_{r,s}$ of $\P Y_{r,s}$ on which the node on
  $C_{v}$ is evaluated, this is determined by the multiplicity of
  $L_{1},\cdots, L_{k}, N_{1},N_{2}$ at the corresponding node. If $v$
  is unstable and hence $h$ corresponds to a single point on a component $C_e$,
  then $m(h)$ is the {\it inverse} in $G \!\times\!
  \boldsymbol{\mu_{s}}\!\times\! \boldsymbol{\mu}_{r}$ of the multiplicity of
  $L_1,\cdots, L_{k},N_{1},N_{2}$ at this point.
\end{itemize}
In particular, we note that the decorations at each stable vertex $v$ yield a
vector
$$\vec{m}(v) \in (G\times \boldsymbol{\mu}_{s}\times \boldsymbol{\mu}_{r})^{\text{val}(v)}$$
recording the multiplicities of $L_1,\cdots, L_{k},N_{1},N_{2}$ at every special
point of $C_v$.

\begin{remark}\label{rmk:edge3*}
  For each edge $e$, the restriction of $\vec{x}$ to $C_e$ defines a constant
  map to $Y$. So the restriction of $(u, \vec{x},\vec{z})$ to $C_{e}$ defines a
  representable map
  $$f: C_{e}\rightarrow  \mathbb B G_{y}\times \P_{r,s}$$
  where $y\in Y$ comes from $\vec{x}$ and $G_{y}$ is the isotropy group of $y\in
  Y$. Then we have $m(q_{0})=(g^{-1},\mu_{s}^{\delta(e)},1)$ and
  $m(q_{\infty})=(g,1,\mu_{r}^{\delta(e)})$ for some $g\in G_{y}$. Denote $a$ to
  be the order of element $g\in G$. Note when $r$ and $s$ are sufficiently large
  primes comparing to $\delta(e)$, we must have $C_{e}\cong \P^{1}_{ar,as}$ and
  $q_{0}$ and $q_{\infty}$ are special points as they are nontrivial stacky
  points. Here $\P^{1}_{ar,as}$ is the unique Deligne-Mumford stack with coarse
  moduli $\P^{1}$, isotropy group $\boldsymbol{\mu}_{as}$ at $0\in \P^{1}$,
  isotropy group $\boldsymbol{\mu}_{ar}$ at $\infty\in \P^{1}$, and generic
  trivial stabilizer.
\end{remark}


\subsection{Localization analysis}
Fix $\beta\in \operatorname{Eff(W,G,\theta)},\delta\in \Z_{\geq 0}$ and $\vec{m}=(v_{1},\cdots,v_{m})\in (G \!\times\! \boldsymbol{\mu}_{s} \!\times\!
\boldsymbol{\mu}_{r})^{m}$, we will consider the space $\mathcal
K_{0,\vec{m}}(\P Y_{r,s},(\beta,\frac{\delta}{r}))$. The reason why we assume
that the second degree is $\frac{\delta}{r}$ is that $\mathcal K_{0,[m]}(\P
Y_{r,s},(\beta,\frac{\delta}{r}))$ admits a natural morphism to $\mathcal
K_{0,[m]}(\P Y,(\beta,\delta))$(c.f.\cite{Andreini_2015, tang16_quant_leray_hirsc_theor_banded_gerbes}). Here $\P Y$ is equal to $\P Y_{r,s}$ for $r=s=1$. In this section, we will always assume that $r$ and $s$ are \emph{sufficiently
  large primes}. 

Now we analyze the $\C^{*}-$localization contribution for $\mathcal
K_{0,\vec{m}}(\P Y_{r,s},(\beta,\frac{\delta}{r}))$ as in \S\ref{subsec:local-ana1}.  
\subsubsection{Vertex contributions}\label{subsubsec:ver-contr2}
The analysis of localization contribution for the stable vertex $v$ is similar
to the analysis in \S\ref{subsubsec:ver-contr1}.

For the stable vertex $v$ over $\infty$, the vertex moduli $\mathcal M_{v}$
corresponds to the moduli stack
$$\mathcal K_{0,\vec{m}(v)}(\sqrt[r]{L_{\theta}/Y},\beta(v)):=\bigsqcup_{\substack{d\in
    \mathrm{Eff}(AY,G,\theta)\\(i_{\mathfrak Y})_{*}(d)=\beta(v)}}\mathcal
K_{0,\vec{m}(v)}(\sqrt[r]{L_{\theta}/Y},d)\ ,$$ which parameterizes twisted
stable maps to the root gerbe $\sqrt[r]{L_{\theta}/Y}$ over $Y$.

Let
$$\pi:\mathcal C_{\infty}\rightarrow \mathcal
K_{0,\vec{m}(v)}(\sqrt[r]{L_{\theta}/Y},\beta(v))$$ be the universal curve over
$\mathcal K_{0,\vec{m}(v)}(\sqrt[r]{L_{\theta}/Y},\beta(v))$. Follow the same
discussion in \S\ref{subsubsec:ver-contr1}, the \emph{inverse of the Euler
  class} of the virtual normal bundle for the vertex moduli $\mathcal M_{v}$
over $\infty$ is equal to
\begin{equation*}
  e^{\C^{*}}((-R^{\bullet}\pi_{*}\mathcal L_{\theta}^{\frac{1}{r}})\otimes \C_{-\frac{\lambda}{r}})\ .
\end{equation*}
When $r$ is a sufficiently large prime,
following~\cite{janda18_doubl_ramif_cycles_with_target_variet}, the above Euler
class has a representation
$$\sum_{d\geq 0}c_{d}(-R^{\bullet}\pi_{*}\mathcal
L_{\theta}^{\frac{1}{r}})(\frac{-\lambda}{r})^{|E(v)|-1-d}\ .$$ Here the
virtual bundle $-R^{\bullet}\pi_{*}\mathcal L^{\frac{1}{r}}_{\theta}$ has
virtual rank $|E(v)|-1$, where $|E(v)|$ is the number of edges incident to the
vertex $v$. The fixed part of the obstruction theory contributes to the virtual
cycle
$$[\mathcal K_{0,\vec{m}(v)}(\sqrt[r]{L_{\theta}/Y},\beta(v))]^{\mathrm{vir}}\ .$$

For the stable vertex $v$ over $0$, the vertex moduli $\mathcal M_{v}$
corresponds to the moduli space
$$\mathcal K_{0,\vec{m}(v)}(\sqrt[s]{L_{-\theta}/Y},\beta(v)):=\bigsqcup_{\substack{d\in
    \mathrm{Eff}(AY,G,\theta)\\(i_{\mathfrak Y})_{*}(d)=\beta(v)}}\mathcal
K_{0,\vec{m}(v)}(\sqrt[s]{L_{-\theta}/Y},d)\ .$$

Let $$\pi:\mathcal C_{0}\rightarrow \mathcal
K_{0,\vec{m}(v)}(\sqrt[s]{L_{-\theta}/Y},\beta(v))$$ be the universal curve
over $\mathcal K_{0,\vec{m}(v)}(\sqrt[s]{L_{-\theta}/Y},\beta(v))$, and
$f:\mathcal C_{0}\rightarrow \sqrt[s]{L_{-\theta}/Y}$. In this case, the fixed
part of the perfect obstruction theory for the vertex moduli over $0$ yields the
virtual cycle
$$[\mathcal K_{0,\vec{m}(v)}(\sqrt[s]{L_{-\theta}/Y},\beta(v))]^{\mathrm{vir}}\ .$$

Note $\mathcal N_{2}|_{\mathcal C_{0}}\cong \mathcal O_{\mathcal C_{0}}$ as
$z_{2}|_{\mathcal C_{0}}\equiv 1$, the virtual normal bundle comes from the
movable part of the infinitesimal deformations of $z_{1}$, which is a section of
the line bundle $\mathcal L^{\frac{1}{s}}_{-\theta}$ over $\mathcal C_{0}$,
which is the pullback of the universal $s-$th root line bundle on
$\sqrt[s]{L_{-\theta}/Y}$ via the universal map $f$. Then the \emph{inverse of
  the Euler class} of the virtual normal bundle is equal to
$$e^{\C^{*}}((-R^{\bullet}\pi_{*}\mathcal L^{\frac{1}{s}}_{-\theta })\otimes
\C_{\frac{\lambda}{s}})\ .$$

We will simplify the above presentation when $\beta(v)\neq 0$. First, we will
state a simple vanishing lemma regarding a line bundle of negative degree on a
genus zero twisted curve, of which the proof is proceeded by induction on the
number of connected components.
\begin{lemma}\label{lem:vanishH0}
  Let $L$ be a line bundle of negative degree on a genus zero twisted curve $C$.
  Assume that the degree of the restriction of the line bundle $L|_{C_{i}}$ to
  every irreducible component $C_{i}$ is non-positive. Then we have
  $H^{0}(C,L)=0\ .$
\end{lemma}

\begin{remark}\label{rmk:vanishH0}
  For every fiber curve $C_{0}$ of the universal curve $\mathcal C_{0}$ over
  $\mathcal M_{v}$. The degree of the restricted line bundle $\mathcal L_{-\theta
  }^{\frac{1}{s}}|_{C_{0}}$ to $C_{0}$ is non-positive. Indeed, $\mathcal L_{-\theta
  }^{\frac{1}{s}}$ is the pullback of the $s$-th root of the line bundle
  $L_{-\theta}$ on $\sqrt[s]{L_{-\theta}/Y}$, where $L_{-\theta}$ is the
  pullback of an \emph{anti-ample} line bundle on the coarse moduli of
  $\sqrt[s]{L_{-\theta}/Y}$. Now assuming $\beta(v)\neq 0$, we have the degree
  of the restricted line bundle $\mathcal L_{-\theta }^{\frac{1}{s}}|_{C_{0}}$
  is negative by Lemma \ref{lem:2.3}. By the above lemma, one has
  $$R^{0}\pi_{*}\mathcal L^{\frac{1}{s}}_{-\theta }=0\ .$$
  Then we have
  $$-R^{\bullet}\pi_{*}\mathcal L^{\frac{1}{s}}_{-\theta
  }=R^{1}\pi_{*}\mathcal L^{\frac{1}{s}}_{-\theta}\ ,$$ which implies that
  $R^{1}\pi_{*}\mathcal L^{\frac{1}{s}}_{-\theta}$ is a vector bundle. When $s$
  is sufficiently large, it has rank $|E(v)|-1$ where $|E(v)|$ is the number of
  edges incident to the vertex $v$. Especially when $|E(v)|=1$, it has rank $0$,
  thus the Euler class becomes $1$, this case will be important in the later
  simplification of the localization contribution in \S\ref{subsec:intr2}.
\end{remark}

\subsubsection{Edge contributions}
Assume that the multiplicity at $q_{\infty}\in C_{e}$ is equal to
$(g,\mu_{s}^{\delta(e)},1)\in G \!\times\! \boldsymbol{\mu}_{s} \!\times\!
\boldsymbol{\mu}_{r}$ and $a$ (or $a_{e}$) is the order of $g\in G$. When $r,s$ is
sufficiently large primes, due to Remark \ref{rmk:edge3*}, $C_{e}$ must be
isomorphic to $\P^{1}_{ar,as}$ where the ramification point $q_{0}$ for which
$z_{1}=0$ is isomorphic to $\mathbb B\mu_{as}$, and the ramification point
$q_{\infty}$ for which $z_{2}=0$ is isomorphic to $\mathbb B\mu_{ar}$. The
restriction of the degree $(\beta,\frac{\delta}{r})$ from $C$ to $C_{e}$ is
equal to $(0,\frac{\delta(e)}{r})$, which is equivalent to:
$$deg(L_{j}|_{C_{e}})=0,\quad \textit{for}\quad \,1\leq j\leq k,\quad deg(N_{2}|_{C_{e}})=\frac{\delta(e)}{r}\ .$$
Recall that the inertia stack component $I_{g}Y$ of $I_{\mu}Y$ is isomorphic to
$$[AY^{ss}(\theta)^{g}/G]\ .$$
We define the edge moduli $\mathcal M_{e}$ to be
$$\sqrt[as\delta(e)]{L_{-\theta}/I_{g}Y}=\sqrt[as\delta(e)]{L_{-\theta}/[AY^{ss}(\theta)^{g}/G]}\ ,$$
which is the $as\delta(e)$th root gerbe over the inertia stack component
$I_{g}Y$ of $I_{\mu}Y$ by taking the $as\delta(e)$th root of the line bundle
$L_{-\theta}$.

The root gerbe $\sqrt[as\delta(e)]{L_{-\theta}/I_{g}Y}$ admits a representation
as a quotient stack:
\begin{equation*}\label{eq:edge3}
  [AY^{ss}(\theta)^{g} \!\times\! \C^{*} / (G \!\times\! \C^{*}_{w})],
\end{equation*}
where the (right) action is defined by:
$$(\vec{x},v)\cdot (g,w)=(\vec{x}g,\theta(g)^{-1}vw^{-as\delta(e)})\ ,$$
for all $(g,w)\in G \!\times\! \C^{*}_{w}$ and $(\vec{x},v)\in
AY^{ss}(\theta)^{g} \!\times\! \C^{*}$. For every character $\rho$ of $G$, we
can define a new character of $G \!\times\! \C^{*}_{w}$ by composing the
projection map $\mathrm{pr}_{G}:G \!\times\! \C^{*}_{w}\rightarrow G$, we will
still use $\rho$ to name the new character of $G \!\times\! \C^{*}_{w} $ by an
abuse of notation. Then $\rho$ will determines a line bundle
$L_{\rho}:=[(AY^{ss}(\theta)^{g}\!\times\! \C^{*} \!\times\! \mathbb C_{\rho})
/(G \!\times\!\C^{*}_{w})]$ on $\sqrt[as\delta(e)]{L_{-\theta}/I_{g}Y}$ by the
Borel construction.

By virtue of the universal property of root gerbe, on $\mathcal
M_{e}=\sqrt[as\delta(e)]{L_{-\theta}/I_{g}Y}$, there is a universal line bundle
$\mathcal R$ that is the $as\delta(e)$th root of the line bundle $L_{-\theta}$.
The root bundle $\mathcal R$ is determined by the character
$\mathrm{pr}_{\C^{*}}$:
$$\mathrm{pr}_{\C^{*}}:G \!\times\! \C^{*}_{w} \rightarrow \C^{*}_{w} \quad (g,w)\in G \!\times\!
\C^{*}_{w} \mapsto w\in \C^{*}_{w}\ .$$ We have the relation
$$L_{-\theta}=\mathcal R^{as\delta(e)}\ .$$
The coordinate functions $\vec{x}$ and $v$ of
$AY^{ss}(\theta)^{g}\times \C^{*}$ descents to be universal sections of line
bundles $\oplus_{\rho\in [n]}L_{\rho}$ and $L_{-\theta}\otimes \mathcal
R^{-\otimes as\delta(e)}$ over $\mathcal M_{e}$, respectively.

We will construct a universal family of $\C^{*}-$fixed twisted stable maps to $\P Y_{r,s}$ of degree
$(0, \frac{\delta(e)}{r})$ over $\mathcal M_e$:
\begin{equation*}
  \xymatrix{
    \mathcal C_e := \P_{ar,as}(\mathcal R \oplus \mathcal O_{\mathcal M_{e}}) \ar^-{f}[r]\ar_{\pi}[d] & \P Y_{r,s}\\
    \mathcal M_{e}:=\sqrt[as\delta(e)]{L_{-\theta}/I_{g}Y}.
  }
\end{equation*}
Then the universal curve $\mathcal C_{e}$ over
$\sqrt[as\delta(e)]{L_{-\theta}/I_{g}Y}$ can be represented as a quotient
stack:
$$\mathcal C_{e}=[(AY^{ss}(\theta)^{g} \!\times\! \C^{*} \!\times\! U) / (G \!\times\! \C^{*}_{w}
\!\times\! T)]\ ,$$ where $T=\{(t_{1},t_{2})\in (\C^{*})^{2}|\quad
t_{1}^{as}=t_{2}^{ar}\}$. The right action is defined by:
$$(\vec{x},v,x,y)\cdot (g,w,(t_{1},t_{2})) =(g\cdot \vec{x},\theta(g)^{-1}vw^{-as\delta(e)},wt_{1}x,t_{2}y)\ ,$$
for all $(g,w,(t_{1},t_{2}))\in G \!\times\! \C^{*}_{w} \!\times\! T$ and
$(\vec{x},v,(x,y))\in AY^{ss}(\theta)^g \!\times\! \C^{*} \!\times\! U$. Then
$\mathcal C_{e}$ is a family of orbifold $\P_{ar,as}$ parameterized by $\mathcal
M_{e}$.

There are two standard characters of $T$
$$\chi_{1}:(t_{1},t_{2})\in T\mapsto
t_{1}\in \C^{*} \quad \chi_{2}:(t_{1},t_{2})\in T\mapsto t_{2}\in \C^{*} \ ,$$
we can lift them to be characters of $G \!\times\! \C^{*}_{w} \!\times\! T$ by
composing the projection map $\mathrm{pr}_{T}:G\!\times\! \C^{*}_{w} \!\times\!
T\rightarrow T$. By an abuse of notation, we continue to use $\chi_{1},\chi_{2}$
to denote the new characters. These two new characters defines two line
bundles$$M_{1}:=(AY^{ss}(\theta)^g \!\times\! \C^{*} \!\times\! U)\times _{G \!\times\! \C^{*}_{w} \!\times\! T} \;\C_{\chi_{1}}$$ and
$$M_{2}:=(AY^{ss}(\theta)^g \!\times\! \C^{*} \!\times\! U)\times _{G \!\times\! \C^{*}_{w}
\!\times\! T}\; \C_{\chi_{2}}$$ on $\mathcal C_{e}$ by the Borel construction, respectively. We
have the relation $M_{1}^{\otimes as}=M_{2}^{\otimes ar}$ on $\mathcal C_{e}$.
The universal map $f$ from $\mathcal C_{e}$ to $\P Y_{r,s}$ can be described as
follows: Let
$$\tilde{f}: AY^{ss}(\theta)^g \!\times\! \C^{*} \!\times\! U\rightarrow \C^{*}
\!\times\! AY^{ss}(\theta) \!\times\! U$$ be the morphism defined by:
\begin{equation}
  \begin{split}
    (\vec{x},v,x,y)&\in AY^{ss}(\theta)^g \!\times\! \C^{*} \!\times\! U \mapsto\\
    &(v, (x_{1},\cdots,x_{n}),x^{a\delta(e)},y^{a\delta(e)})\in \C^{*}
    \!\times\! AY^{ss}(\theta) \!\times\! U\ .
  \end{split}
\end{equation}
Then $\tilde{f}$ is equivariant with respect to the group homomorphism from $G
\!\times\! \C^{*}_{w} \!\times\! T $ to $G \!\times\! \C^{*}_{\alpha} \!\times\!
\C^{*}_{t}$ defined by:
\begin{equation}
  \begin{split}
    (g,w,(t_{1},t_{2}))&\in G \!\times\! \C^{*}_{w} \!\times\! T \mapsto\\
    &\big(g\big((t^{-s}_{1}t_{2}^{r})^{p_{1}},\cdots,(t^{-s}_{1}t_{2}^{r})^{p_{k}}\big),(wt_{1})^{a\delta(e)},t_{2}^{a\delta(e)}\big)\in
    G \!\times\! \C^{*}_{\alpha} \!\times\! \C^{*}_{t}\ ,
  \end{split}
\end{equation}
where the tuple $(p_{1},\cdots, p_{k})\in \mathbb N^{k}$ satisfies that
$g=(\mu_{a}^{p_{1}},\cdots,\mu_{a}^{p_{k}})\in G$. Note $\tilde{f}$ is
well-defined for $\chi^{-s}_{1}\chi_{2}^{r}$ is a torsion character of $T$ of
order $a$. The above construction gives the universal morphism $f$ from
$\mathcal C_{e}$ to $\P Y_{r,s}$ by descent.

We will define a (quasi left) $\C^{*}-$action on $\mathcal C_{e}$ such that the
map $f$ constructed above is $\C^{*}-$equivariant. Define a $\C^{*}$-action on $\mathcal C_{e}$ induced by the $\C^{*}-$action on $AY^{ss}(\theta)^g \!\times\! \C^{*} \!\times\! U$:
  $$m:\C^{*}\times AY^{ss}(\theta)^g \!\times\! \C^{*} \!\times\! U\rightarrow AY^{ss}(\theta)^g \!\times\! \C^{*} \!\times\! U\ ,$$
  $$t\cdot (\vec{x},v,(x,y))= (\vec{x},v,(x,t^{\frac{-1}{ar\delta(e)}}y))\ .$$
Note the morphism $\pi$ is also $\C^{*}$-equivariant, where $\mathcal M_{e}$ is
equipped with trivial $\C^{*}$-action. By the universal property of the projectivized bundle $\mathcal C_e$ over $\mathcal
M_{e}$, one has a tautological section
\begin{equation}\label{eq:tautosec3}
 (x, y) \in H^0\big((M_{1} \otimes \mathcal \pi^{*}\mathcal
R) \oplus (M_{2}\otimes \C_{\frac{-\lambda}{ar\delta(e)}})\big)\ ,
\end{equation}  
which is also a $\C^{*}-$invariant section.

Now we can check that $f$ is a $\C^{*}-$equivariant morphism from $\mathcal
C_{e}$ to $\P Y_{r,s}$ with respect to the $\C^{*}-$actions
for $\mathcal C_{e}$ and $\P Y_{r,s}$. According to Remark
\ref{rmk:equivmap3}, $f$ is equivalent to the following data:
\begin{enumerate}\label{equivedge3}
\item $k+2$ $\C^{*}$-equivariant line
  bundles on $\mathcal C_{e}$:
  $$\mathcal L_{j} := \pi^*L_{\pi_{j}} \otimes (M^{-\otimes s}_{1}\otimes M_{2}^{\otimes r})^{p_{j}},
  1\leq j\leq k $$ and
  $$\mathcal N_{1}:=(M_{1}\otimes \pi^{*}\mathcal R)^{\otimes a\delta(e)}\,\quad \mathcal N_{2}:=M_{2}^{a\delta(e)}\otimes \mathbb C_{\frac{-\lambda}{r}}\ .$$
  Where $L_{\pi_{j}}$ are the standard $\C^{*}$-equivariant line bundles on
  $\mathcal M_{e}$ by the Borel construction, $M_{1},M_{2}$ are the standard
  $\C^{*}$-equivariant line bundles on $C_{e}$ by the Borel construction.
 \item a universal section
    \begin{equation}
    \begin{split}
      (u,\vec{x},(\zeta_1, \zeta_2)) :=
      &(v, x_{1},\cdots,x_{n},(x^{a\delta(e)}, y^{a\delta(e)})) \\
      &\in \Gamma\big( \big((\mathcal N^{\vee}_{1})^{\otimes s}\otimes \mathcal
      L_{-\theta}\otimes \mathcal N_{2}^{\otimes r}\otimes
      \C_{\lambda}\big)\oplus \bigoplus_{1\leq i\leq n} \mathcal
      L_{\rho_{i}}\oplus \mathcal N_{1}\oplus \mathcal N_{2}\big)^{\C^{*}}\ .
    \end{split}
  \end{equation}
  Here one only need to check $v\in\Gamma((\mathcal N^{\vee}_{1})^{\otimes
    s}\otimes \mathcal L_{-\theta}\otimes \mathcal N_{2}^{\otimes r}\otimes
  \C_{\lambda})$, which is easy to be verified.
\end{enumerate}

Now we compute the localization contribution from $\mathcal M_{e}$. Based on the
perfect obstruction theory for stable maps in $\mathcal K_{0,\vec{m}}(\P
Y_{r,s},(\beta,\frac{\delta}{r}))$, the restriction of the perfect obstruction
theory to $\mathcal M_{e}$ decomposes into three parts: (1) the deformation
theory of source curve $\mathcal C_{e}$; (2) the deformation theory of the lines
bundles $(\mathcal L_{i})_{1\leq j\leq k}$ and $\mathcal N$; (3) the deformation
theory for the section
$$(u, \vec{x},(\zeta_{1},\zeta_{2}))\in \Gamma\big(\big(\mathcal N_{1}^{-\otimes
  s}\otimes \mathcal L_{-\theta }\otimes \mathcal N_{2}^{\otimes r}\otimes
\mathbb C_{\lambda}\big)\oplus\bigoplus_{1\leq i\leq n} \mathcal L_{\rho_{i}}
\oplus \mathcal N_{1} \oplus \mathcal N_{2}\big)\ .$$

The $\C^{*}-$fixed part of three parts above will contribute to the virtual
cycle of $\mathcal M_{e}$, we will show that $[\mathcal
M_{e}]^{\mathrm{vir}}=[\mathcal M_{e}]$. The virtual normal bundle comes from
the $\C^{*}-$moving part of the above three parts.

First every fiber curve $C_{e}$ in $\mathcal C_{e}$ over a geometrical point in
$\mathcal M_{e}$ is isomorphic to $\P_{ar,as} $, which is rational. There are no
infinitesimal deformations/obstructions for $C_{e}$, line bundles
$L_{j}:=\mathcal L_{j}|_{C_{e}}$, $N_{1}:=\mathcal N_{1}|_{C_{e}}$ and
$N_{2}:=\mathcal N_{2}|_{C_{e}}$. Hence their contribution to the perfect
obstruction theory comes from infinitesimal automorphisms. The infinitesimal
automorphisms of $C_{e}$ come from the space of vector fields on $C_{e}$ that vanish on
special points. Thus the $\C^{*}-$fixed part of infinitesimal automorphisms of $C_{e}$
comes from the $1-$dimensional subspace of vector fields on $C_{e}$ which vanish
on the two ramification points. The
movable part of infinitisimial automorphisms of $C_{e}$ is nonzero only if one
of ramification points on $C_{e}$ is not a special point. by Remark
\ref{rmk:edge3*}, the ramifications on $C_{e}$ are both nontrivial stacky points
when $r$ and $s$ are sufficiently large, hence they must be special points. So
there is no movable part for infinitesimal automorphisms of $C_{e}$.

Now let's turn to the localizations from sections. First the
infinitesimal deformations of sections $(u,\vec{x})$ are fixed, which, together
with fixed part of infinitesimal automorphisms of $C_{e}$ and line bundles
$L_{j}, \;N_{1},\; N_{2}$, as well as fixed parts of infinitesimal deformations of sections $(z_{1},z_{2}):=(\zeta_{1},\zeta_{2})|_{C_{e}}$,
contribute to the virtual cycle $[\mathcal M_{e}]^{\mathrm{vir}}$, which is
equal to the fundamental class of $\mathcal M_{e}$. The localization
contribution from the infinitesimal deformations of sections $(z_{1},z_{2})$ to
the virtual normal bundle is:
$$(R^{\bullet}\pi_{*}(\mathcal N_{1}
\oplus \mathcal N_{2}))^{\mathrm{\mathrm{mov}}}\ .$$

We first come to the deformations of $z_{2}$, we continue to use the
tautological section $(x,y)$ as in \eqref{eq:tautosec3}. For each fiber $C_{e}$,
sections of $N_{2}$ is spanned by monomials
$(x^{asm}y^{n})|_{C_{e}}$ with $arm+n=a\delta(e)$ and $m,n\in \Z_{\geq 0}$. Note
$x^{asm}y^{n}$ may not be a global section of $\mathcal N_{2}$ but always a global
section of $\mathcal R^{\otimes asm}\otimes \mathcal N_{2}\otimes
\C_{\frac{m}{\delta(e)}\lambda}$. Then $R^{\bullet}\pi_{*}\mathcal N_{2}$ will
decompose as a direct sum of line bundles, each corresponds to the monomial
$x^{asm}y^{n}$, whose first chern class is
$$c_{1}(\mathcal R^{\otimes -asm} \bigotimes \mathbb C_{\frac{-m}{\delta(e)}\lambda})=\frac{m}{\delta(e)}(D_{\theta}-\lambda)\ .$$
So the total contribution is equal to
$$\prod_{m=0}^{\lfloor\frac{\delta(e)}{r}\rfloor}\bigg(\frac{m}{\delta(e)}(D_{\theta}-\lambda)\bigg)\ .$$
The factor for $m=0$ appearing in the above product is the $\C^{*}-$fixed part
of $R^{\bullet}\pi_{*}\mathcal N_{2}$, it will contribute to the virtual cycle
of $\mathcal M_{e}$. The rest contributes to the virtual normal bundle as
$$\prod_{m=1}^{\lfloor\frac{\delta(e)}{r}\rfloor}\bigg(\frac{m}{\delta(e)}(D_{\theta}-\lambda)\bigg)\ .$$
Note when $r$ is sufficiently large, the above product becomes 1.

For the deformations of $z_{1}$, arguing in the same way as $z_{2}$, the Euler class of $R^{\bullet}\pi_{*}\mathcal N_{1}$ is equal to
$$\prod_{n=0}^{\lfloor {\frac{\delta(e)}{s}} \rfloor}\bigg(\frac{n}{\delta(e)}(-D_{\theta}+\lambda)\bigg)\ .$$ 
The factor for $m=0$ appearing in the above product is the $\C^{*}-$fixed part of $R^{\bullet}\pi_{*}\mathcal N_{1}$, it will contribute to the virtual cycle of $\mathcal M_{e}$. The Euler
class of virtual normal bundle of $\mathcal M_{e}$ comes from the movable part
of deformations of section $z_{1}$ is:
$$\prod_{n=1}^{\lfloor {\frac{\delta(e)}{s}} \rfloor}\bigg(\frac{n}{\delta(e)}(-D_{\theta}+\lambda)\bigg)\ .$$
Note when $s$ is sufficiently large, the above product becomes 1.


\subsubsection{Node contributions}\label{subsubsec:node-cntr3}
The deformations in $\mathcal K_{0,\vec{m}}(\P
Y_{r,s},(\beta,\frac{\delta}{r}))$ smoothing a node contribute to the Euler
class of the virtual normal bundle as the first Chern class of the tensor
product of the two cotangent line bundles at the branches of the node. For nodes
at which a component $C_e$ meets a component $C_v$ over the vertex $0$, this
contribution is
\begin{equation*}
  \frac{\lambda - D_{\theta}}{as\delta(e)} - \frac{\bar{\psi}_v}{as}.
\end{equation*}
 
For nodes at which a component $C_e$ meets a component $C_v$ at the vertex over
$\infty$, this contribution is
\begin{equation*}
  \frac{-\lambda + D_{\theta }}{ar\delta(e)} -\frac{\bar{ \psi}_v}{ar}.
\end{equation*}
The type of node at which two edge component $C_e$ meets with a vertex $v$ over
$0$ or $\infty$ will not occur using a similar argument in \cite[Lemma
6]{Janda2017}.


As for the node contributions from the normalization exact sequence, each node
$q$ (specified by a vertex $v$) contributes the Euler class of
\begin{equation}
  \label{node3}
  \big(R^0\pi_*\mathcal N_{1}|_{q}\big)^{\mathrm{mov}} \oplus \big(R^0\pi_* \mathcal N_{2}|_{q}\big)^{\mathrm{mov}}
\end{equation}
to the virtual normal bundle. In the case where $j(v) = 0$, $z_{2}|_{q}\equiv 1$
gives a trivialization of the fiber $\mathcal N_{2}|_{q}$, note that $(\mathcal N^{\vee}_{1})^{\otimes s}\otimes
        \mathcal L_{-\theta}\otimes \mathcal N_{2}^{\otimes r}\otimes
        \C_{\lambda}\cong \mathbb C$ we have $\mathcal
N_{2}|_{q}\cong \C$ and $ \mathcal N_{1}|_{q}\cong
L_{-\theta}^{\frac{1}{s}}\otimes \mathbb C_{\frac{\lambda}{s}}$, this implies
that $(R^0\pi_* \mathcal
N_{2}|_{q})^{\mathrm{mov}}=0$ and $R^{0}\pi_{*}\mathcal N_{1}|_{q}=0$. The later
vanishes because of the nontrivial stacky structure of the
line bundle $\mathcal N_{1}$ at $q$ when $s$ is sufficiently large. Hence there
no localization contribution from the normalization at the node $q$ over $0$. Similarly, for each
node $q$ incident to a vertex $v$ with $j(v) = \infty$, there is no localization
contribution
from the normalization at the node over $\infty$.

\subsection{Total localization contributions}\label{subsec:total3}
For each decorated graph $\Gamma$, denote $F_{\Gamma}$ to be the
fiber product
$$\prod_{v:j(v)=0}\mathcal M_{v}\times_{\bar{I}_{\mu}\sqrt[s]{L_{-\theta}/Y}} \prod
  _{e\in E}\mathcal M_{e}\times_{\bar{I}_{\mu}\sqrt[r]{L_{\theta}/Y}} \prod_{v:j(v)=\infty} \mathcal M_{v}$$ of the following
diagram:
\[\xymatrixcolsep{5pc}\xymatrix{
  F_{\Gamma}\ar[r]\ar[d] &\prod\limits_{v:j(v)=0}\mathcal M_{v}\times \prod\limits_{e\in E}\mathcal
  M_{e}\times \prod\limits_{v:j(v)=\infty} \mathcal
  M_{v}
  \ar[d]^-{ev_{h_{e}},ev_{p^{e}_{0}}, ev_{p^{e}_{\infty}}, ev_{h'_{e}}}\\
  \prod\limits_{E}\bar{I}_{\mu}\sqrt[s]{L_{-\theta}/Y} \times
  \bar{I}_{\mu}\sqrt[r]{L_{\theta}/Y}\ar[r]^-{(\Delta^{\frac{1}{s}}\times
    \Delta^{\frac{1}{r}})^{E}}
  &\prod\limits_{E}\bigg((\bar{I}_{\mu}\sqrt[s]{L_{-\theta}/Y})^{2}\times
  (\bar{I}_{\mu}\sqrt[r]{L_{\theta }/Y})^{2}\bigg)\ , }
\]
where $\Delta^{\frac{1}{s}}=(id,\iota)$(resp. $\Delta^{\frac{1}{r}}=(id,\iota)$)
is the diagonal map of $\bar{I}_{\mu}\sqrt[s]{L_{-\theta}/Y}$ (resp.
$\bar{I}_{\mu}\sqrt[r]{L_{\theta}/Y}$) into
$\bar{I}_{\mu}\sqrt[s]{L_{-\theta}/Y}\times
\bar{I}_{\mu}\sqrt[s]{L_{-\theta}/Y}$ (resp.
$\bar{I}_{\mu}\sqrt[r]{L_{\theta}/Y}\times \bar{I}_{\mu}\sqrt[r]{L_{\theta}/Y}$). Here when $v$ is a stable
vertex, the vertex moduli $\mathcal M_{v}$ is described in
\ref{subsubsec:ver-contr2}; when $v$ is an unstable vertex over
$0$, we treat $\mathcal M_{v}:=\bar{I}_{m(h)}\sqrt[s]{L_{-\theta}/Y}$ with the
identical virtual cycle, where $m(h)$ is the multiplicity of the half-edge
incident to $v$; when $v$ is an unstable vertex over $\infty$, we treat $\mathcal
M_{v}:=\bar{I}_{m(h)}\sqrt[r]{L_{\theta }/Y}$ with the identical virtual cycle,
where $m(h)$ is the multiplicity of the half-edge incident to $v$.

We define that $[F_{\Gamma}]^{\mathrm{vir}}$ to be:
\begin{equation*}
\begin{split}
\prod_{v:j(v)=0}[\mathcal M_{v}]^{\mathrm{vir}}\times_{\bar{I}_{\mu}\sqrt[s]{L_{-\theta}/Y}} \prod
  _{e\in E}[\mathcal M_{e}]^{\mathrm{vir}}\times_{\bar{I}_{\mu}\sqrt[r]{L_{\theta}/Y}} \prod_{v:j(v)=\infty} [\mathcal M_{v}]^{\mathrm{vir}}\ , 
\end{split}
\end{equation*}
where the fiber product over $\bar{I}_{\mu}\sqrt[s]{L_{-\theta}/Y}$ and
$\bar{I}_{\mu}\sqrt[r]{L_{\theta}/Y}$ imposes that the evaluation maps at the
two branches of each node (here we adopt the convention that a node can link a
unstable vertex and an edge.) agree. Then the contribution of decorated graph
$\Gamma$ to the virtual localization is
is: 
\begin{equation}\label{eq:auto-loc2}
Cont_{\Gamma}=\frac{\prod_{e\in
    E}sa_{e}}{\text{Aut}(\Gamma)}(\iota_{\Gamma})_{*}\left(\frac{[F_{\Gamma}]^{\mathrm{vir}}}{e^{\C^{*}}(N^{\mathrm{vir}}_{\Gamma})}\right)\ .
\end{equation}
Here $\iota_{F}:F_{\Gamma}\rightarrow \mathcal K_{0,\vec{m}}(\P
Y_{r,s},(\beta,\frac{\delta}{r}))$ is a finite etale map of degree
$\frac{Aut(\Gamma)}{\prod_{e\in E}sa_{e}}$ into the corresponding
$\C^{*}$-fixed loci in $\mathcal K_{0,\vec{m}}(\P
Y_{r,s},(\beta,\frac{\delta}{r}))$. The virtual normal bundle
$e^{\C^{*}}(N^{\mathrm{vir}}_{\Gamma})$ is the
product of virtual normal bundles from vertex contributions, edge contributions
and node contributions. 
 
\section{Recursion relations from auxiliary cycles}\label{sec:rec-rel}
For any $\beta\in \mathrm{Eff}(W,G,\theta)$, for simplicity, we
will denote
$$\mathcal K_{0,\vec{m}}(\bullet,\beta)
:=\bigsqcup_{\substack{d\in \mathrm{Eff}(\bullet)\\
    (i_{\bullet})_{*}(d)=\beta}}\mathcal K_{0,\vec{m}}(\bullet,d),
$$
where $\bullet$ can be $Y$,$\sqrt[r]{L_{\theta}/Y}$ and
$\sqrt[s]{L_{-\theta}/Y}$, and $i_{\bullet}$ is the natural structure map from $\bullet$
to $\mathfrak X$ which factors through the inclusion $i_{\mathfrak Y}:\mathfrak
Y\rightarrow \mathfrak X$.

\subsection{Auxiliary cycle I}\label{subsec:intr1}
Fix a nonzero $\beta\in \operatorname{Eff(W, G,\theta)}$, and a positive
rational number $\epsilon$ and the tuple
$\boldsymbol{\varepsilon}=(\epsilon,\cdots,\epsilon)\in (\mathbb
Q_{>0})^{p}$ such that $\epsilon\beta(L_{\theta })+p\epsilon\leq 1$. Set
$\delta=\beta(L_{\theta})+p$. For simplicity, we will denote
$$Q^{\widetilde{\theta}}_{0,\star}(\P \mathfrak Y^{\frac{1}{r},p},(\beta,1^{p},\frac{\delta}{r}))
:=\bigsqcup_{\substack{d\in \mathrm{Eff}(AY,G,\theta)\\
    (i_{\mathfrak Y})_{*}(d)=\beta}}Q^{\widetilde{\theta}}_{0,1}(\P \mathfrak
Y^{\frac{1}{r},p},(d,1^{p},\frac{\delta}{r}))\cap
ev^{-1}_{1}(\bar{I}_{(g_{\beta},\frac{\delta}{r})}\P Y^{\frac{1}{r}})\ . $$
Recall that $g_{\beta}\in G$ is defined in \S\ref{sec:I-fun}. We
will always assume that $r$ is a sufficiently large prime in this subsection.

For any nonnegative integer $c$, we will first consider the following auxiliary
cycle:
\begin{equation}\label{eq:mainintegral1}
  \frac{1}{p!}(\widetilde{EV_{\star}})_{*}\left(e^{\C^{*}}\big(R^{1}\pi_{*}f^{*}L_{\infty}^{\vee}\big)\cap [Q^{\widetilde{\theta}}_{0,\star}(\P \mathfrak Y^{\frac{1}{r},p},(\beta,1^{p},\frac{\delta}{r}))]^{\mathrm{vir}}\cap \bar{\psi}_{\star}^{c}\cap \prod_{j=1}^{p}\hat{ev}_{j}^{*}(\hat{t})\right).
\end{equation}
Here an explanation of the notations is in order:
\begin{enumerate}
\item the morphism $$\pi:\mathcal C\rightarrow
  Q^{\widetilde{\theta}}_{0,\star}(\P \mathfrak
  Y^{\frac{1}{r},p},(\beta,1^{p},\frac{\delta}{r}))$$ is the universal curve and the
  morphism $$f:\mathcal C\rightarrow \P \mathfrak Y^{\frac{1}{r},p}$$ is the
  universal map;
\item the $\C^{*}$-equivariant line bundle $L_{\infty}$ corresponds to the
  invertible sheaf $\mathcal O(\mathcal D_{\infty})$ on $\P \mathfrak
  Y^{\frac{1}{r},p}$ with the $\C^{*}$-linearization that $\C^{*}$ acts on the
  fiber over $\mathcal D_{\infty}$ (given by $z_{2}=0$) with weight
  $-\frac{\lambda}{r}$ and $\C^{*}$ acts on the fiber over $\mathcal D_{0}$
  (given by $z_{1}=0$) with weight $0$. For every fiber curve $C$ of the
  universal curve $\mathcal C$ via $\pi$, the important observation here is that
  the restricted line bundle $f^{*}L^{\vee}_{\infty}|_{C}$ has negative degree
  as $\beta\neq 0$ and has non-positive degree on every irreducible components
  of $C$. Indeed, if the image of an irreducible component of $C$ via $f$ isn't
  contained in $\mathcal D_{\infty}$, the degree is obviously non-positive. If
  the image of an irreducible component of $C$ under $f$ is contained in
  $\mathcal D_{\infty}$, then using the fact that $L^{\vee}_{\infty}$ is
  isomorphic to $(L_{\theta}^{\frac{1}{r}})^{\vee}$ over
  $$\mathcal D_{\infty}\cong \sqrt[r]{L_{\theta}/Y}$$
  and a similar discussion as in remark \ref{rmk:vanishH0}, the degree is also
  non-positive. Thus by Lemma \ref{lem:vanishH0}, we have
  $R^{0}\pi_{*}f^{*}L_{\infty}^{\vee}=0$, which implies that
  $R^{1}\pi_{*}f^{*}L_{\infty}^{\vee}$ is a vector bundle (of rank 0);
\item the morphism $EV_{\star}$ is a composition of the following maps:
  $$\xymatrix{ Q^{\widetilde{\theta}}_{0,\star}(\P \mathfrak Y^{\frac{1}{r},p},(\beta,1^{p},\frac{\delta}{r}))\ar[r]^-{ev_{\star}} &\bar{I}_{\mu}\P
    Y^{\frac{1}{r}} \ar[r]^{\mathrm{pr}_{r}} &\bar{I}_{\mu}Y\ ,}$$ where
  $\mathrm{pr}_{r}: \bar{I}_{\mu}\P Y^{\frac{1}{r}}\rightarrow \bar{I}_{\mu} Y$
  is the morphism induced from the natural structure map from $\P
  Y^{\frac{1}{r}}$ to $Y$ forgetting $z_{1},z_{2}$.
  $(\widetilde{EV_{\star}})_{*}$ is defined by
  $$\iota_{*}(r_{\star}(EV_{\star})_{*})$$
  as in \eqref{tilde-ev}. Note here $r_{\star}$ is the order of the band from
  the gerbe structure of $\bar{I}_{\mu}Y$ but not $\bar{I}_{\mu} \P
  Y^{\frac{1}{r}} $.
\item recall that $\hat{ev}_{j}$ is defined in \eqref{eq:hatev}. The cohomology class $\hat{t}\in H^{*}(\mathfrak Y,\mathbb Q)[t_{1},\cdots,t_{l}]$ is of the form
  $\sum_{i=1}^{l}t_{i}u_{i}(c_{1}(L_{\pi_j}))$. where $t_{i}$ are formal
  variables and $u_{i}$ are polynomials in the first chern class of the line
  bundles $L_{\pi_{j}}$ associated to the standard characters $\pi_j$ of $G$.  
\end{enumerate}

Apply virtual localization to $Q^{\widetilde{\theta}}_{0,\star}(\P \mathfrak
Y^{\frac{1}{r},p},(\beta,1^{p},\frac{\delta}{r}))$, we first prove the following
vanishing result, where the idea is borrowed from~\cite{localp1}.
\begin{lemma}\label{lem:vanish1}
  The localization graph $\Gamma$ that has more than one vertex labeled by $\infty$ will
  contribute zero to \eqref{eq:mainintegral1}.
\end{lemma}
\begin{proof}
  Assume by contradiction, by the connectedness of the graph, there is at least
  one vertex at $0$ with valence at least two. Suppose $f:C\rightarrow \P
  \mathfrak Y^{\frac{1}{r},p}$ is $\C^{*}$-fixed. Assume that $C_{0}\cap C_{1}\cap
  C_{2}$ is part of curve $C$, where $C_{0}$ contracted by $f$ to $\mathcal
  D_{0}$ (given by $z_{1}=0$) and $C_{1},C_{2}$ are edges meeting with $C_{0}$
  at $b_{1}$ and $b_{2}$. Then in the normalization sequence for
  $R^{\bullet}\pi_{*}f^{*}(L_{\infty}^{\vee})$, it contains the part
  \begin{align*}\label{cd:vanishing}
    &H^{0}\left(C_{0}, f^{*}\left( L_{\infty}^{\vee}\right)\right) \oplus H^{0}\left(C_{1}, f^{*}\left(L_{\infty}^{\vee}\right)\right) \oplus H^{0}\left(C_{2}, f^{*}\left( L_{\infty}^{\vee}\right)\right)\\
    \rightarrow &H^{0}\left(b_{1}, f^{*}\left(
                  L_{\infty}^{\vee}\right)\right) \oplus H^{0}\left(b_{2},
                  f^{*}\left(L_{\infty}^{\vee}\right)\right)\\
    \rightarrow &H^{1}\left(C, f^{*}\left(L_{\infty}^{\vee}\right)\right).
  \end{align*}
  Hence there is one of the weight-0 pieces in $H^{0}\left(b_{1}, f^{*}\left(
      L_{\infty}^{\vee}\right)\right) \oplus H^{0}\left(b_{2},
    f^{*}\left(L_{\infty}^{\vee}\right)\right)$ that is canceled with a weight-0
  piece of $H^{0}\left(C_{0}, f^{*}\left( L_{\infty}^{\vee}\right)\right)$, and
  the other is mapped injectively into $H^{1}\left(C,
    f^{*}\left(L_{\infty}^{\vee}\right)\right)$. Thus $H^{1}\left(C,
    f^{*}\left(L_{\infty}^{\vee}\right)\right)$ contains a weight-0 piece with
  vanishing equivariant Euler class.
\end{proof}

Recall that we can write $\mathbb I(q,t,z)=\sum _{\beta,p}q^{\beta}\mathbb I_{\beta,p}$, where
$\mathbb I_{\beta,p}:=\frac{\boldsymbol{t}^{p}}{p!z^{p}}\mathbb I_{\beta}(z)$ is a Laurent polynomial in $z$ with coefficient in
degree $p$ multi-polynomial in $H^{*}(\bar{I}_{\mu}Y, \mathbb Q)[t_{0},\cdots,t_{l}]$. We will prove the following recursion relation by applying localization to \eqref{eq:mainintegral2}.
\begin{theorem}\label{thm:rec1}
  For any nonnegative integer $c$, $[z\mathbb I_{\beta,p}]_{z^{-c-1}}$ satisfies the
  following relation:
  \begin{equation}\label{eq:rec1}
    \begin{split}
      [z\mathbb I_{\beta,p}]_{z^{-c-1}}&=\bigg[\sum_{m=0}^{\infty}\sum_{\substack{\beta_{\star}+\beta_{1}+\cdots+\beta_{m}=\beta\\p_{1}+\cdots+p_{m}=p\\(\beta_{i},p_{i})\neq 0\, \text{for}\, 1\leq i\leq m}}\frac{1}{m!}(\widetilde{ev_{\star}})_{*}\bigg(\sum
      _{d=0}^{\infty}\epsilon_{*}\big(c_{d}(-R^{\bullet}\pi_{*}\mathcal
      L_{\theta}^{\frac{1}{r}})(\frac{\lambda}{r})^{-1+m-d}(-1)^{d}\\ &\cap [\mathcal
      K_{0,\vec{m}_{r}\cup
        \star}(\sqrt[r]{L_{\theta}/Y},\beta_{\star})]^{\mathrm{vir}}\big)\cap
      \prod_{i=1}^{m}\frac{ev_{i}^{*}(\frac{1}{\delta_{i}}(z \mathbb I_{\beta_{i},p_{i}}(z))|_{z=\frac{\lambda-D_{\theta}}{\delta_{i}}})}{\frac{\lambda-ev^{*}_{i}D_{\theta
          }}{r\delta_{i}}+\frac{\bar{\psi_{i}}}{r}}\cap
      \bar{\psi}_{\star}^{c}\bigg)\bigg]_{\lambda^{-1}}.
    \end{split}
  \end{equation}
  Here $\delta_{i}=\beta_{i}(L_{\theta})+p_{i}$ for $1\leq i\leq m$, $\epsilon:\mathcal K_{0,\vec{m}_{r}\cup
    \star}(\sqrt[r]{L_{\theta}/Y},\beta_{\star})\rightarrow \mathcal
  K_{0,\vec{m}\cup \star}(Y,\beta_{\star})$ is the natural structure morphism
  (c.f.\cite{tang16_quant_leray_hirsc_theor_banded_gerbes}), where
  $$\vec{m}_{r}\cup \star:=((g_{\beta_{i}}^{-1},\mu_{r}^{-\beta_{i}(L_{\theta})-p_{i}}):1\leq i\leq m)\times
    \{(g_{\beta},\mu_{r}^{\beta(L_{\theta})+p})\}\in
    (G\!\times\!\boldsymbol{\mu}_{r})^{m+1}\ ,$$
    and $\vec{m}\cup\star:=(g_{\beta_{i}}^{-1}:1\leq i\leq
    m)\!\times\!\{g_{\beta}\}\in G^{m+1}$.
\end{theorem}
\begin{proof}
  By Lemma \ref{lem:vanish1}, only decorated graph $\Gamma$, which has only one
  vertex $v$ labeled by $\infty$, may have nonzero localization contribution to the
  \eqref{eq:mainintegral1}. Note the only marking $q_{\star}$ can only be
  incident to the vertex labeled by $\infty$ due to the restriction on the
  multiplicity at $q_{\star}$. Furthermore, for such graph $\Gamma$, we claim
  there is no stable vertex labeled by $0$. Indeed, for any vertex $v$ over
  $0$, its degree $\beta(v)$ satisfies that $\beta(v)(L_{\theta})\leq
  \beta(L_{\theta})\leq \frac{1}{\epsilon}$, and it has valence 1 as no legs
  can attach to it and at most one edge is incident to it by Lemma
  \ref{lem:vanish1}, then the vertex $v$ must be unstable. So the
  decorated graph $\Gamma$ has only one vertex over $\infty$ with possible
  several edges (can be empty) attached, and each vertex labeled by $0$ corresponds to an edge in the graph
  $\Gamma$, which appears as an unmarked point (actually a base point as we will
  see). In the following, we analyze the localization contribution to
  \eqref{eq:mainintegral1} from the graph $\Gamma$ described above. We have two cases which
  depends on whether the only vertex labeled by $\infty$ on the graph $\Gamma$
  is stable or unstable.
  \begin{enumerate}\label{case:graph1}
  \item If the only vertex $v$ over $\infty$ is unstable, then it's a vertex
    with valence 2, i.e, it's incident to a leg and an edge. In this case
    the degree $(\beta,1^{p},\frac{\delta}{r})$ is concentrated on the single edge
    with the marked point $q_{\star}$ appearing as the ramification point over
    $\infty$ on the edge. Then it contributes
    \begin{equation*}\label{eq:edge-loc1}
      \frac{1}{\delta}(z\mathbb I_{\beta,p}(z))|_{z=\frac{\lambda-D_{\theta}}{\delta}}\cdot (\frac{\lambda-D_{\theta}}{\delta})^{c}
    \end{equation*}  
    to \eqref{eq:mainintegral1}. Here we use the fact that the restriction of
    $R^{1}\pi_{*}(f^{*}L^{\vee}_{\infty})$ to $F_{\Gamma}$ is a rank 0 vector
    bundle, so its equivariant Euler class is 1, and the restriction of
    $\bar{\psi}_{\star}$ to $\mathcal M_{e}$ is equal to
    $\frac{\lambda-D_{\theta}}{\delta}$.
  \item If the only vertex $v$ over $\infty$ is stable, then $v$ is incident to
    only one leg and possible several edges (can be none). We assume that the vertex
    $v$ has degree $(\beta_{*},\frac{\delta_{*}}{r})$ with
    $\delta_{*}=\beta_{*}(L_{\theta})$. If there is no edges in the graph
    $\Gamma$, which happens if and only if $\beta_{\star}=\beta$ and $p=0$, the corresponding graph
    has contribution
    \begin{equation}\label{eq:sub3}
      (\widetilde{ev_{\star}})_{*}\bigg(\sum
      _{d=0}^{\infty}\epsilon_{*}(c_{d}(-R^{\bullet}\pi_{*}\mathcal
      L_{\theta}^{\frac{1}{r}})(\frac{-\lambda}{r})^{-1-d} \cap [\mathcal K_{0,\star}(\sqrt[r]{L_{\theta}/Y},\beta_{\star})]^{\mathrm{vir}})\cap \bar{\psi}^{c}_{\star}\bigg).
    \end{equation}
    to the \eqref{eq:mainintegral1}. Otherwise we index all the edges attached to the vertex $v$ from $1$ to $m$ such that
    the edge $e_{i}$ corresponding to the index $i$ has degree
    $(\beta_{i},1^{J_{e_{i}}},\frac{\delta_{i}}{r})$. Since we assume that the total degree is
    $(\beta,1^{p},\frac{\delta}{r})=(\beta,1^{p},\frac{\beta(L_{\theta})}{r})$, and the
    degree on every edge satisfies the relation $\delta_{i}\geq
    \beta_{i}(L_{\theta})+p_{i}$ by Remark \ref{rmk:edge2}, where $p_{i}=|J_{e_{i}}|$. Therefore we must have
    $\delta_{i}=\beta_{i}(L_{\theta })+p_{i}$ for every edge $e_{i}$. It follows that
    all the edge has a base point and $(\beta_{i},p_{i})$ is nonzero.

    Equipped with these notations, by Remark \ref{rmk:edge1}, the vertex moduli $\mathcal M_{v}$ over
    $\infty$ is
    \begin{equation*}
      \begin{split}
        \mathcal K_{0,\vec{m}_{r}\cup
          \star}(\sqrt[r]{L_{\theta}/Y},\beta_{\star})&=\\
        \mathcal K_{0,[m]\cup
          \star}(\sqrt[r]{L_{\theta}/Y},\beta_{\star})\cap&\bigcap_{i=1}^{m}ev_{i}^{-1}(\bar{I}_{(g_{\beta_{i}}^{-1},\mu_{r}^{-\delta_{i}})}\sqrt[r]{L_{\theta}/Y})
        \cap
        ev_{\star}^{-1}(\bar{I}_{(g_{\beta},\mu_{r}^{\delta})}\sqrt[r]{L_{\theta}/Y}).
      \end{split}
    \end{equation*}
     
    Using the localization analysis in \S\ref{subsec:local-ana1} and
    the fact that $e^{\C^{*}}(R^{1}\pi_{*}f^{*}L_{\infty}^{\vee})=1$ as it's of rank zero on
    $F_{\Gamma}$. The localization contribution of such graph $\Gamma$ to
    \eqref{eq:mainintegral1} is equal to
    \begin{equation}\label{eq:sub1}
      \begin{split}
        \frac{1}{Aut(\Gamma)}\tilde{ev}_{*}\bigg(\sum
        _{d=0}^{\infty}\epsilon_{*}\big(c_{d}(-R^{\bullet}\pi_{*}\mathcal
        L_{\theta}^{\frac{1}{r}})(\frac{-\lambda}{r})^{-1+m-d}\cap [\mathcal K_{0,\vec{m}_{r}\cup\star}(\sqrt[r]{L_{\theta}/Y},\beta_{\star})]^{\mathrm{vir}}\big)\\
        \cap
        \prod_{i=1}^{m}\frac{ev_{i}^{*}(\frac{1}{\delta_{i}}(z \frac{t^{p_{i}}}{z^{p_{i}}}\mathbb I_{\beta_{i}}(q,z))|_{z=\frac{\lambda-D_{\theta}}{\delta_{i}}})}{-\frac{\lambda-ev^{*}_{i}D_{\theta
            }}{r\delta_{i}}-\frac{\bar{\psi_{i}}}{r}}\cap
        \bar{\psi}^{c}_{\star}\bigg) ,
      \end{split}
    \end{equation}
    where $t=\sum t_{i}u_{i}(c_{1}(L_{\tau_{ij}})+\beta(L_{\tau_{ij}})z)$ and $\epsilon:\mathcal K_{0,\vec{m}_{r}\cup
      \star}(\sqrt[r]{L_{\theta}/Y},\beta_{\star})\rightarrow \mathcal
    K_{0,\vec{m}\cup \star}(Y,\beta_{\star})$ is the natural structure map.
    Here $\vec{m}\cup \star=(g_{\beta_{i}}^{-1}:1\leq i\leq m)\times
    \{g_{\beta}\}\in G^{m+1}$. 
    
    Fix $\beta_{*}$ and $m$, the sum of \eqref{eq:sub1} coming from all possible
    graph which has $\infty-$vertex $v$ of degree $\beta_{\star}$ and $m$
    incident edges yields:
    \begin{equation}\label{eq:sub2}
      \begin{split}
        \sum_{\substack{\beta_{\star}+\beta_{1}+\cdots+\beta_{m}=\beta\\p_{1}+\cdots+p_{m}=p\\(\beta_{i},p_{i})\neq 0\, \text{for}\, 1\leq i\leq m}}\frac{1}{m!}&(\widetilde{ev_{\star}})_{*}\bigg(\sum
        _{d=0}^{\infty}\epsilon_{*}\big(c_{d}(-R^{\bullet}\pi_{*}\mathcal
        L_{\theta}^{\frac{1}{r}})(\frac{-\lambda}{r})^{-1+m-d}\\
        &\cap [\mathcal
        K_{0,\vec{m}_{r}\cup\star}(\sqrt[r]{L_{\theta}/Y},\beta_{\star})]^{\mathrm{vir}}\big)
        \cap
        \prod_{i=1}^{m}\frac{ev_{i}^{*}(\frac{1}{\delta_{i}}(z \mathbb I_{\beta_{i},p_{i}}(z))|_{z=\frac{\lambda-D_{\theta}}{\delta_{i}}})}{-\frac{\lambda-ev^{*}_{i}D_{\theta
            }}{r\delta_{i}}-\frac{\bar{\psi_{i}}}{r}}\cap
        \bar{\psi}_{\star}^{c}\bigg).
      \end{split}
    \end{equation}
  \end{enumerate}

  In summary, the auxiliary cycle \eqref{eq:mainintegral1} is equal to:
  \begin{equation}\label{eq:sum1}
    \begin{split}
      &\frac{1}{\delta}(z\mathbb I_{\beta,p}(z))|_{z=\frac{\lambda-D_{\theta}}{\delta}}\cdot (\frac{\lambda-D_{\theta}}{\delta})^{c}\\ 
      &+\sum_{m=0}^{\infty}\sum_{\substack{\beta_{\star}+\beta_{1}+\cdots+\beta_{m}=\beta\\p_{1}+\cdots+p_{m}=p\\(\beta_{i},p_{i})\neq 0\, \text{for}\, 1\leq i\leq m}}\frac{1}{m!}(\widetilde{ev_{\star}})_{*}\bigg(\sum
      _{d=0}^{\infty}\epsilon_{*}\big(c_{d}(-R^{\bullet}\pi_{*}\mathcal
      L_{\theta}^{\frac{1}{r}})(\frac{-\lambda}{r})^{-1+m-d}\\
            & \cap [\mathcal K_{0,\vec{m}_{r}\cup\star}(\sqrt[r]{L_{\theta}/Y},\beta_{\star})]^{\mathrm{vir}}\big)
      \cap
      \prod_{i=1}^{m}\frac{ev_{i}^{*}(\frac{1}{\delta_{i}}(z \mathbb I_{\beta_{i},p_{i}}(z))|_{z=\frac{\lambda-D_{\theta}}{\delta_{i}}})}{-\frac{\lambda-ev^{*}_{i}D_{\theta
          }}{r\delta_{i}}-\frac{\bar{\psi_{i}}}{r}}\cap
      \bar{\psi}_{\star}^{c}\bigg)\ .
          \end{split}
  \end{equation}

  Observe that \eqref{eq:mainintegral1} does not have negative $\lambda$ powers, then
  the $\lambda^{-1}$ coefficient in the equation \eqref{eq:sum1} is equal to zero.
  Note that the $\lambda^{-1}$ coefficient in \eqref{eq:sum1} is equal to
  \begin{equation}\label{eq:sum2}
    \begin{split}
      &[z\mathbb I_{\beta,p}(z)]_{z^{-c-1}}+\bigg[\sum_{m=0}^{\infty}\sum_{\substack{\beta_{\star}+\beta_{1}+\cdots+\beta_{m}=\beta\\p_{1}+\cdots+p_{m}=p\\(\beta_{i},p_{i})\neq 0\, \text{for}\, 1\leq i\leq m}}\frac{1}{m!}(\widetilde{ev_{\star}})_{*}\bigg(\sum
      _{d=0}^{\infty}\epsilon_{*}\big(c_{d}(-R^{\bullet}\pi_{*}\mathcal
      L_{\theta}^{\frac{1}{r}})(\frac{-\lambda}{r})^{-1+m-d}\\     
      &\cap [\mathcal
      K_{0,\vec{m}_{r}\cup\star}(\sqrt[r]{L_{\theta}/Y},\beta_{\star})]^{\mathrm{vir}}\big)\cap
      \prod_{i=1}^{m}\frac{ev_{i}^{*}(\frac{1}{\delta_{i}}(z \mathbb I_{\beta_{i},p_{i}}(z))|_{z=\frac{\lambda-D_{\theta}}{\delta_{i}}})}{-\frac{\lambda-ev^{*}_{i}D_{\theta
          }}{r\delta_{i}}-\frac{\bar{\psi_{i}}}{r}}\cap
      \bar{\psi}_{\star}^{c}\bigg)\bigg]_{\lambda^{-1}}\ .
    \end{split}
  \end{equation}
  Now \eqref{eq:sum2} immediately implies the formula \eqref{eq:rec1}. 
  
\end{proof}

\subsection{Auxiliary cycle II}\label{subsec:intr2}
In this section, for any $\beta\in \mathrm{Eff}(W,G,\theta)$, we will denote
$\mathcal K_{0,\vec{m}}(\P Y_{r,s},(\beta,\frac{\delta}{r}))$ to be
$$\bigsqcup_{\substack{d\in \mathrm{Eff}(AY,G,\theta)\\(i_{\mathfrak Y})_{*}(d)=\beta}}\mathcal
K_{0,\vec{m}}(\P Y_{r,s},(d,\frac{\delta}{r}))\ .$$

Fix $\beta,\delta$ as in \S\ref{subsec:intr1}. Assume that $r,s$ are
sufficiently large primes. We will also compare \eqref{eq:mainintegral1} to the
following auxiliary cycle:
\begin{equation}\label{eq:mainintegral2}
  \begin{split}
    \sum_{m=0}^{\infty}\sum_{\substack{\beta_{\star}+\beta_{1}+\cdots+\beta_{m}=\beta\\p_{1}+\cdots+p_{m}=p}}\frac{1}{m!}(\widetilde{EV_{\star}})_{*}\bigg([\mathcal
    K_{0,\vec{m}\cup\star}(\P Y_{r,s},(\beta_{\star},\frac{\delta}{r}))]^{\mathrm{vir}}\\
    \cap\prod_{i=1}^{m}ev_{i}^{*}\big(\mathrm{pr}_{r,s}^{*}(\mu_{\beta_{i},p_{i}}(-\bar{\psi}_{i}))\big)
    \cap \bar{\psi}_{\star}^{c}\cap e^{\C^{*}}(R^{1}\pi_{*}f^{*}
    L_{\infty}^{\vee})\bigg).
  \end{split}
\end{equation}
Here an explanation of the notations is in order:
\begin{enumerate}
\item the morphism $$\pi:\mathcal C\rightarrow \mathcal K_{0,\vec{m}\cup
    \star}(\P Y_{r,s},(\beta_{\star},\frac{\delta}{r}))$$ is the universal curve
  and the morphism
 $$f:\mathcal C\rightarrow \P Y_{r,s}$$ is the universal map; 
\item for any nonnegative integer $m$ and degrees
  $\beta_{\star},\beta_{1},\cdots, \beta_{m}$ in $\mathrm{Eff}(W,G,\theta)$.
  Here $\vec{m}=(m_{i}\in G \!\times\! \boldsymbol{\mu}_{s} \!\times\!
  \boldsymbol{\mu}_{r}:1\leq i \leq m)$, in which
  $m_{i}=(g^{-1}_{\beta_{i}},\mu_{s}^{\beta_{i}(L_{\theta })+p_{i}},1)$, and $m_{\star}=(g_{\beta},1,\mu_{r}^{\beta(L_{\theta })+p})\in
  G\times \boldsymbol{\mu}_{s}\times \boldsymbol{\mu}_{r}$. So $\mathcal
  K_{0,\vec{m}\cup\star}(\P Y_{r,s},(\beta_{\star},\frac{\delta}{r}))$ is
  defined to be:
  \begin{equation*}
    \mathcal
    K_{0,m+1}(\P Y_{r,s},(\beta_{\star},\frac{\delta}{r}))\cap \bigcap_{i=1}^{m}ev^{-1}_{i}(\bar{I}_{m_{i}}\P Y_{r,s})\cap ev^{-1}_{m+1}(\bar{I}_{m_{\star}}\P Y_{r,s} );
  \end{equation*}
\item the line bundle $L_{\infty}$ corresponds to the invertible sheaf $\mathcal
  O(\mathcal D_{\infty})$ with the $\C^{*}$-linearization such that $\C^{*}$
  acts on the fiber over $\mathcal D_{\infty}$ with weight $-\frac{\lambda}{r}$
  and acts on the fiber over $\mathcal D_{0}$ with weight zero; Using the same
  reasoning in the case of auxiliary cycle I, we have
  $R^{0}\pi_{*}f^{*}L_{\infty}^{\vee}=0$ and
  $R^{1}\pi_{*}f^{*}L_{\infty}^{\vee}$ is a vector bundle (of rank $0$);
\item the morphism $EV_{\star}$ is a composition of the following maps:
  $$\xymatrix{ \mathcal K_{0,\vec{m}\cup \star}(\P
    Y_{r,s},(\beta_{\star},\frac{\delta}{r}))\ar[r]^-{ev_{\star}}
    &\bar{I}_{\mu}\P Y_{r,s} \ar[r]^{\mathrm{pr}_{r,s}} &\bar{I}_{\mu}Y\
    ,}$$where $\mathrm{pr}_{r,s}: \bar{I}_{\mu}\P Y_{r,s}\rightarrow
  \bar{I}_{\mu} Y$ is the morphism induced from the natural structure map from
  $\P Y_{r,s}$ to $Y$ forgetting $u$ and $z_{1},z_{2}$, and
  $(\widetilde{EV_{\star}})_{*}$ is defined by
  $$\iota_{*}(r_{\star}(EV_{\star})_{*})$$
  as in \ref{tilde-ev}. Note here $r_{\star}$ is the order of the band from the
  gerbe structure of $\bar{I}_{\mu}Y$ but not $\bar{I}_{\mu}\P Y_{r,s}$.
\end{enumerate}
First we have a similar vanishing result as Lemma \ref{lem:vanish1} by an
analogous argument.
\begin{lemma}\label{lem:vanish2}
  The localization graph $\Gamma$ which has more than one vertex labeled by $\infty$
  contributes zero to \eqref{eq:mainintegral2}.
\end{lemma}

We will prove the following recursion relation by applying localization to \eqref{eq:mainintegral2}.
\begin{theorem}\label{thm:rec2}
  Assume $r,s$ are sufficiently large and prime. For any nonnegative integer
  $c$, the summation
  $$\sum_{m=0}^{\infty}\sum_{\substack{\beta_{\star}+\beta_{1}+\cdots+\beta_{m}=\beta\\p_1+\cdots+p_m=p}}\frac{1}{m!}\phi^{\alpha}\langle\mu_{\beta_{1},p_{1}}(-\bar{\psi}_{1}),\cdots,\mu_{\beta_{m,p_{m}}}(-\bar{\psi}_{m}),\phi_{\alpha}\bar{\psi}^{c}_{\star}
  \rangle_{0,[m]\cup \star,\beta_{\star}}$$ satisfies the following
  relation:
  \begin{equation}\label{eq:rec2}
    \begin{split}
      &\sum_{m=0}^{\infty}\sum_{\substack{\beta_{\star}+\beta_{1}+\cdots+\beta_{m}=\beta\\p_1+\cdots+p_m=p}}\frac{1}{m!}\phi^{\alpha}\langle{\mu_{\beta_{1}}(-\bar{\psi}_{1}),\cdots,\mu_{\beta_{m}}(-\bar{\psi}_{m}),\phi_{\alpha}\bar{\psi}^{c}_{\star}}
      \rangle_{0,[m]\cup \star,\beta_{\star}}\\&=
\bigg[\sum_{m=0}^{\infty}\sum_{\substack{\beta_{\star}+\beta_{1}+\cdots+\beta_{m}=\beta\\p_1+\cdots+p_m=p\\(\beta_{i},p_{i})\neq 0\, \text{for}\, 1\leq i\leq m}}\frac{1}{m!}(\tilde{ev_{\star}})_{*}\bigg(\sum
      _{d=0}^{\infty}\epsilon_{*}\big(c_{d}(-R^{\bullet}\pi_{*}\mathcal
      L_{\theta}^{\frac{1}{r}})(\frac{\lambda}{r})^{-1+m-d}(-1)^{d}\\
      &\cap [\mathcal K_{0,\vec{m}_{r}\cup
        \star}(\sqrt[r]{L_{\theta}/Y},\beta_{\star})]^{\mathrm{vir}}\big) \cap
      \prod_{i=1}^{m}\frac{ev_{i}^{*}(\frac{1}{\delta_{i}}f_{\beta_{i},p_{i}}(z)|_{z=\frac{\lambda-D_{\theta}}{\delta_{i}}})}{\frac{\lambda-ev^{*}_{i}D_{\theta
          }}{r\delta_{i}}+\frac{\bar{\psi_{i}}}{r}}\cap
      \bar{\psi}_{\star}^{c}\bigg)\bigg]_{\lambda^{-1}}.
    \end{split}
  \end{equation}
  Here $f_{\beta_{i},p_{i}}(z)$ is defined as follows:
  \begin{equation}\label{eq:fbeta}
    \begin{split}
      f_{\beta_{i},p_{i}}(z)&:=\mu_{\beta_{i},p_{i}}(z)+
      \sum_{l=0}^{\infty}\sum_{\substack{\beta_{0}+\beta_{1}+\cdots+\beta_{l}=\beta_{i}\\p_{1}+\cdots+p_{l}=p_{i}}}
      \frac{1}{l!}\\&(\widetilde{ev_{0}})_{*}\bigg( [\mathcal
      K_{0,[l]\cup\{0\}}(Y,\beta_{0})]^{\mathrm{vir}}\cap
      \bigcap_{j=1}^{l}ev_{j}^{*}(\mu_{\beta_{j},p_{j}}(-\bar{\psi}_{j}))\cap
      \frac{1}{z-\bar{\psi}_{0}}\bigg)\ ,
    \end{split}
  \end{equation}
  $\delta_{i}=\beta_{i}(L_{\theta})+p_{i}$,
  $\vec{m}_{r}\cup\star:=((g^{-1}_{\beta_{i}},\mu_{r}^{-\beta_{i}(L_{\theta})-p_{i}}):1\leq
  i\leq m)\cup \{g_{\beta},\mu_{r}^{\beta(L_{\theta})+p}\}$, $\vec{m}_{r}\cup\star:=(g^{-1}_{\beta_{i}}:1\leq
  i\leq m)\cup \{g_{\beta}\}$ and $\epsilon:\mathcal K_{0,\vec{m}_{r}\cup
    \star}(\sqrt[r]{L_{\theta}/Y},\beta_{\star})\rightarrow \mathcal
  K_{0,\vec{m}\cup \star}(Y,\beta_{\star})$ is the natural structure morphism.
\end{theorem}
\begin{proof}
  By Lemma \ref{lem:vanish2}, only decorated graph $\Gamma$ that has only one
  vertex $v$ labeled by $\infty$ may have nonzero localization contribution to the
  \eqref{eq:mainintegral2}. Let's denote by $\beta_{\star}$ the degree of the
  unique vertex $v$ labeled by $\infty$ coming from $\vec{x}$. Note the marked point $q_{\star}$
  must lie on a vertex labeled by $\infty$ due to the choice of multiplicity at
  the marking $q_{\star}$. Thus the vertex $v$ can't be a node linking
  two edges. Note one can assume all the other legs appear in vertexes labeled
  by $0$ due to the restriction of multiplicity on the other legs and the fact $\mu_{0}=0$. Hence there are only two types of graph $\Gamma$ depending on
  whether $v$ is a stable or unstable vertex.
  \begin{enumerate}\label{case:graph2}
  \item If the only vertex $v$ over $\infty$ in $\Gamma$ is unstable, in the case,
    $v$ is of valence $2$, i.e. it's incident to an edge and an leg
    corresponding to the marking
    $q_{\star}$. Then $\Gamma$ has only one edge whose degree
    $(0,\frac{\delta}{r})$, and has only one vertex over $0$, which is incident
    to the edge. The vertex over $0$ can be stable or unstable. If the vertex
    over $0$ is unstable, it must be a marked point with input $\mu_{\beta,p}$,
    then the graph $\Gamma$ contributes
    $$\frac{\mu_{\beta,p}(\frac{\lambda-D_{\theta}}{\delta})}{\delta}\cdot(\frac{\lambda-D_{\theta}}{\delta})^{c}$$
    to \eqref{eq:mainintegral2}. If the vertex over $0$ is stable, then this
    type of graphs contributes
    \begin{equation*}
      \begin{split}
        &\sum_{m=0}^{\infty}\sum_{\substack{\beta_{0}+\cdots+\beta_{m}=\beta\\p_{1}+\cdots+p_{m}=p}}\frac{1}{m!}(\widetilde{ev_{0}})_{*}\bigg(\sum_{d=0}^{\infty}\epsilon'_{*}\big(c_{d}(-R^{\bullet}\pi_{*}\mathcal L_{-\theta}^{\frac{1}{s}})(\frac{\lambda}{s})^{-d}\\
        &\cap [\mathcal
        K_{0,\vec{m}_{s}\cup\{0\}}(\sqrt[s]{L_{-\theta}/Y},\beta_{0})]^{\mathrm{vir}}\big)\cap \bigcap_{i=1}^{m}ev_{i}^{*}(\mu_{\beta_{i},p_{i}}(-\bar{\psi}_{i}))\cap
        \frac{\frac{1}{\delta}(\frac{\lambda-ev^*_0D_{\theta}}{\delta})^{c}}{\frac{\lambda-ev^*_0D_{\theta}}{s\delta}-\frac{\bar{\psi}_{0}}{s}}\bigg)
      \end{split}
    \end{equation*} 
    to \eqref{eq:mainintegral2}, where
    $\vec{m}_{s}\cup\{0\}=((g^{-1}_{\beta_{1}},\mu_{s}^{\beta_{1}(L_{\theta
      })+p_{1}}),\cdots, (g^{-1}_{\beta_{m}},\mu_{s}^{\beta_{m}(L_{\theta
      })+p_{m}}))\cup \{(g_{\beta},\mu_{s}^{-\beta(L_{\theta})-p})\}$, $\vec{m}\cup\{0\}=(g^{-1}_{\beta_{1}},\cdots,g^{-1}_{\beta_{m}})\cup\{g_{\beta}\}$ and  $\epsilon':\mathcal K_{0,\vec{m}_{s}\cup\{0\}}(\sqrt[s]{L_{-\theta}/Y},\beta_{0})\rightarrow \mathcal
    K_{0,\vec{m}\cup \{0\}}(Y,\beta_{0})$ is the natural structure morphism (c.f.\cite{tang16_quant_leray_hirsc_theor_banded_gerbes}). By Lemma
    \ref{lem:simp-fbeta} proved below, the above formula is equal to
    \begin{equation*}
      \sum_{m=0}^{\infty}\sum_{\substack{\beta_{0}+\cdots+\beta_{m}=\beta\\p_1+\cdots+p_m=p}}\frac{1}{m!}\phi^{\alpha}\langle{\mu_{\beta_{1},p_{1}}(-\bar{\psi}_{1}),\cdots,\mu_{\beta_{m}.p_{m}}(-\bar{\psi}_{m}),\frac{\frac{1}{\delta}(\frac{\lambda-D_{\theta}}{\delta})^{c}\phi_{\alpha}}{\frac{\lambda-D_{\theta}}{\delta}-\bar{\psi}_{0}}}
      \rangle_{0,[m]\cup \{0\},\beta_{0}}\ .
    \end{equation*}
    
    In summary, the localization contribution from the decorated graphs of which
    the vertex over $\infty$ is unstable contributes
    \begin{equation}\label{eq:red2*}
      \begin{split}
        &\mu_{\beta,p}(\frac{\lambda-D_{\theta}}{\delta})\cdot(\frac{\lambda-D_{\theta}}{\delta})^{c}\\
        &+\sum_{m=0}^{\infty}\sum_{\substack{\beta_{0}+\beta_{1}+\cdots+\beta_{m}=\beta\\p_{1}+\cdots+p_{m}=p}}\frac{1}{m!}\phi^{\alpha}\langle{\mu_{\beta_{1},p_{1}}(-\bar{\psi}_{1}),\cdots,\mu_{\beta_{m},p_{m}}(-\bar{\psi}_{m}),\frac{\frac{1}{\delta}(\frac{\lambda-D_{\theta}}{\delta})^{c}\phi_{\alpha}}{\frac{\lambda-D_{\theta}}{\delta}-\bar{\psi}_{0}}}\rangle_{0,[m]\cup
          \{0\},\beta_{0}}
      \end{split}
    \end{equation}to the \eqref{eq:mainintegral2}. Here we use the fact
    $R^{1}\pi_{*}(f^{*}\mathcal L^{\vee}_{\infty})$ is of rank 0 over
    $F_{\Gamma}$, so its Euler class is 1.
  \item If the only vertex $v$ over $\infty$ in $\Gamma$ is stable, then $v$ is incident to
    only one leg (corresponding to the marking $q_{\star}$) and possible several edges (can be none). Let's assume
    that $v$ has degree $(\beta_{\star},\frac{\delta_{\star}}{r})$ with
    $\delta_{\star}=\beta_{\star}(L_{\theta})$. If there is no edges in the
    graph $\Gamma$, which happens if and only if $\beta_{\star}=\beta$ and
    $p=0$. Then this has contribution:
    \begin{equation}\label{eq:sub3*}
      (\widetilde{ev_{\star}})_{*}\bigg(\sum
      _{d=0}^{\infty}\epsilon_{*}\big(c_{d}(-R^{\bullet}\pi_{*}(f^{*}\mathcal
      L_{\theta}^{\frac{1}{r}}))(\frac{-\lambda}{r})^{-1-d} \cap [\mathcal K_{0,\star}(\sqrt[r]{L_{\theta}/Y},\beta)]^{\mathrm{vir}}\big)\cap \bar{\psi}^{c}_{\star}\bigg)
    \end{equation}
    to \eqref{eq:mainintegral2}. Otherwise, there are $m$ ($m\geq 1$) edges
    attached to the vertex $v$, let's index them by $[m]:=\{1,\cdots,m\}$, with
    degree $(0,\frac{\delta_{i}}{r})$ on the $i$th edge $e_{i}$ for
    $\delta_{i}\in \Z_{>0}$. On each edge $e_{i}$ there is exactly one vertex $v_{i}$
    over $0$ incident to it, which can't be a unstable vertex of valence 1 (see
    Remark \ref{rmk:edge3*}) or a node linking two edges by Lemma
    \ref{lem:vanish2}. So $v_{i}$ is either a marking or a stable vertex with
    only one node incident to the edge $e_{i}$ and possible $l$ marked points
    ($l$ can be zero) on it, let's label the legs incident to $v_{i}$ by
    $\{i1,\cdots,il\}\subset [n]$ ($n$ is the total number of legs on $\Gamma$). 

    Assume that the vertex $v_{i}$ has degree $\beta_{i0}$. Since the
    insertion at the marking $q_{ij}$ on the curve $C_{v_{i}}$ corresponding to
    $v_{i}$ is of the form
    $\mu_{\beta_{ij},p_{ij}}(-\bar{\psi}_{ij})$ in \eqref{eq:mainintegral2}, let's say
    the leg for $q_{ij}$ has \emph{virtual degree} $(\beta_{ij},p_{ij})$ contribution to the vertex $v_{i}$, denote
    $\beta_{i}$ to be summation of $\beta_{i0}$ and all the virtual degrees from
    the markings on $C_{v_{i}}$, similarly for $p_{i}$. We call $(\beta_{i},p_{i})$ the \emph{total degree} at the vertex
    $v_{i}$. From the \eqref{eq:mainintegral2}, one has
    $$\beta_{\star}+\beta_{1}+\cdots+\beta_{m}=\beta,\; \;p_1+\cdots+p_m=p\ .$$

    Note to ensure such a graph $\Gamma$ exists, one must have
    \begin{equation}\label{eq:adm-graph}
     \beta_{i}(L_{\theta })+p_{i}=\delta_{i}\ .
    \end{equation}
    Indeed, by Riemann-Roch Theorem, one has
    $$deg(N_{1}|_{C_{v_{i}}})=-\frac{\beta_{i0}(L_{\theta})}{s}=
    (1-\frac{\delta_{i}}{s})+\sum_{j=1}^{l}\frac{\beta_{ij}(L_{\theta})+p_{ij}}{s}\quad
    mod \quad \Z\ .$$ Here the first term on the left hand is the age of $N_{1}$
    at the node of $C_{v_{i}}$, and the second term on the
    right is the sum of the ages of $N_{1}$ at the marked points on $C_{v_{i}}$. As $s$ is sufficiently large, one must have
    $$\frac{\delta_{i}}{s}=\frac{\beta_{i0}(L_{\theta})}{s}+\sum_{j=1}^{l}\frac{\beta_{ij}(L_{\theta})+p_{ij}}{s}\ ,$$
    which implies that $\beta_{i}(L_{\theta})+p_{i}=\delta_{i}$.

    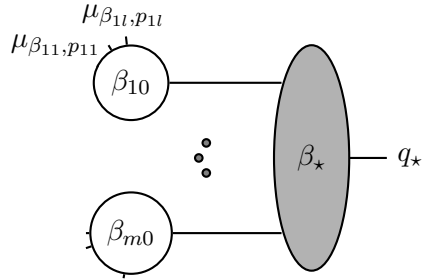
\begin{figure}[htbp]
      \centering
      \begin{tikzpicture}[baseline=3pt,->,>=,node distance=1.3cm,thick,main
        node/.style={circle,draw,font=\Large,scale=0.2}]
        \draw[fill=gray!60!white] (0,0) ellipse (0.5 and 1.5); \node at (0,0)
        (0){$\beta_\star$}; \node at (-2.4,1) [circle, draw] (01)
        {$\beta_{10}$};\node at (1,0)[right] (star){$q_{\star}$}; \node at (-2.4,-1) [circle,draw] (0k){$\beta_{m0}$};
        \node at (-3.4,1.5) (1){\textup{$\mu_{\beta_{11},p_{11}}$}}; \node at
        (-2.5,1.9) (l){\textup{$\mu_{\beta_{1l},p_{1l}}$}}; \draw [-] (-0.4,1) to (01);
        \draw[fill=gray!100!white] (-1.4,0.2) circle (0.05cm);
        \draw[fill=gray!100!white] (-1.5,0) circle (0.05cm);
        \draw[fill=gray!100!white] (-1.4,-0.2) circle (0.05cm); \draw [-]
        (-2.7,1.5) to (01); \draw [-] (-0.4,-1) to (0k); \draw [-] (-3,-1) to
        (0k); \draw [-] (-3,-1.2) to (0k); \draw [-] (-2.5,-1.6) to (0k); \draw
        [-] (1) to (01); \draw [-] (l) to (01);\draw [-] (1,0) to (0.5,0);
      \end{tikzpicture}
      \caption{The ellipse dubbed gray on the right means the vertex labeled by
        $\infty$ with a leg attached, and the
        two big circles on the left mean vertexes labeled by $0$. The text inside the vertex means the
        decorated degree for this vertex. On the upper left vertex, texts near the legs mean the
        insertion terms. The three grey dots in the middle mean the other edges
        (together with its incident vertexes and legs on them)
        besides edges indexed by $1$ and $m$.}
      \label{fig:typeB}
    \end{figure}

    We call a decorated graph $\Gamma$ admissible if $\Gamma$ has only one
    stable vertex $v$ over
    $\infty$ of degree $\beta_{\star}$ and $m$ ($m\geq 1$) edges \emph{labeled
      by $[m]$} incident to
    $v$ such that the degree for each vertex over $0$ satisfies
    \eqref{eq:adm-graph}. Note our definition of the admissible decorated graph has
    more decorations than the decorated graph introduced in Section
    \ref{sec:masterspace-inf} as we also label the edges. Then the automorphism
    group of an admissible decorated
    graph $\Gamma$ is identity, which is usually smaller than the automorphism group of the 
    corresponding decorated graph without labeling the edges. If we want to use
    admissible decorated graphs to compute the localization contribution, we need to divide $m!$ to
    offset the labeling as shown below.
    
    Now we can group the admissible decorated graphs by the triple
    $$(m, \beta_{\star},((\beta_{1},p_{1}),\cdots,(\beta_{m},p_{m})))\ .$$
    Denote by $\Lambda_{(m, \beta_{\star},((\beta_{1},p_{1}),\cdots,(\beta_{m},p_{m}))}$ the
    collection of all the admissible decorated graphs such that the vertex incident to the edge labeled by $i$ has total degree
    $(\beta_{i},p_{i})$.
        
    Now using the locaization formula in \S\ref{subsec:total3} to compute
    the contribution from $\Lambda_{(m,
      \beta_{\star},((\beta_{1},p_{1}),\cdots,(\beta_{m},p_{m}))}$ to \eqref{eq:mainintegral2}.
    For any admissible decorated graph $\Gamma$ in $\Lambda_{(m,
      \beta_{\star},((\beta_{1},p_{1}),\cdots,(\beta_{m},p_{m}))}$, they all have the same
    localization contribution for vertex and nodes over $\infty$
    and edges, as well as the same automorphism factor which comes from the gerbe
    structures of the edge moduli. Then the contribution from the vertex $v_{i}$ together with node at
    $v_{i}$ has localization contribution (after pushing forward to
    $\bar{I}_{\mu}Y$ along $\iota\circ (ev_{h_{i}})_{*}$, where $h_{i}$ is the
    node on $v_{i}$ incident to the edge $e_{i}$):
    \begin{equation*}
      \begin{split}
        &\mu_{\beta_{i},p_{i}}(\frac{\lambda-D_{\theta}}{\delta_{i}})+
        \sum_{l=0}^{\infty}\sum_{\substack{\beta_{0}+\beta_{1}+\cdots+\beta_{l}=\beta_{i}\\p_{1}+\cdots+p_{l}=p_{i}}}\frac{1}{l!}(\widetilde{ev_{0}})_{*}\bigg(\sum_{d=0}^{\infty}\epsilon'_{*}\big(c_{d}(-R^{\bullet}\pi_{*}\mathcal L_{-\theta}^{\frac{1}{s}})(\frac{\lambda}{s})^{-d}\\
        &\cap [\mathcal K_{0,\vec{l}\cup\{0\}}(\sqrt[s]{L_{-\theta}/Y},\beta_{0})
        ]^{\mathrm{vir}}\big)\cap \bigcap_{j=1}^{l}ev_{j}^{*}(\mu_{(\beta_{j},p_{j})}(-\bar{\psi}_{j}))\cap
        \frac{1}{\frac{\lambda-ev^*_0D_{\theta}}{\delta_{j}s}-\frac{\bar{\psi}_{0}}{s}}\bigg),
      \end{split}
    \end{equation*}
    which, by Lemma \ref{lem:simp-fbeta} below, is equal to
    $f_{\beta_{i},p_{i}}(z)|_{\frac{\lambda-D_{\theta}}{\delta_{i}}}$ by our
    definition of $f_{\beta_{i},p_{i}}(z)$. Put contributions from vertexes, edges
    together and nodes together, it yields
    \begin{equation*}
      \begin{split}
        &\frac{1}{m!}(\widetilde{ev_{\star}})_{*}\bigg(\sum
        _{d=0}^{\infty}\epsilon_{*}\big(c_{d}(-R^{\bullet}\pi_{*}\mathcal
        L_{\theta}^{\frac{1}{r}})(\frac{-\lambda}{r})^{-1+m-d} \cap [\mathcal K_{0,\vec{m}_{r}\cup\star}(\sqrt[r]{L_{\theta}/Y},\beta_{\star})]^{\mathrm{vir}}\big)\\
        &\cap
        \prod_{i=1}^{m}\frac{ev_{i}^{*}(\frac{1}{\delta_{i}}(f_{\beta_{i},p_{i}}(z)|_{z=\frac{\lambda-D_{\theta}}{\delta_{i}}})}{-\frac{\lambda-ev^{*}_{i}D_{\theta
            }}{r\delta_{i}}-\frac{\bar{\psi_{i}}}{r}}\cap
        \bar{\psi}_{\star}^{c}\bigg)\ .
      \end{split}
    \end{equation*}
    Now go over all possible triples $(m,
    \beta_{\star},((\beta_{1},p_{1}),\cdots,(\beta_{m},p_{m})))$, it yields the summation:
    \begin{equation*}
      \begin{split}
        \sum_{m=1}^{\infty}\sum_{\substack{\beta_{\star}+\beta_{1}+\cdots+\beta_{m}=\beta\\p_{1}+\cdots+p_{m}=p}}&\frac{1}{m!}(\widetilde{ev_{\star}})_{*}\bigg(\sum
        _{d=0}^{\infty}\epsilon_{*}\big(c_{d}(-R^{\bullet}\pi_{*}\mathcal
        L_{\theta}^{\frac{1}{r}})(\frac{-\lambda}{r})^{-1+m-d} \cap [\mathcal
        K_{0,\vec{m}_{r}\;\cup \star}(\sqrt[r]{L_{\theta}/Y},\beta_{\star})]^{\mathrm{vir}}\big)\\
        &\cap
        \prod_{i=1}^{m}\frac{ev_{i}^{*}(\frac{1}{\delta_{i}}(f_{\beta_{i},p_{i}}(z)|_{z=\frac{\lambda-D_{\theta}}{\delta_{i}}})}{-\frac{\lambda-ev^{*}_{i}D_{\theta
            }}{r\delta_{i}}-\frac{\bar{\psi_{i}}}{r}}\cap
        \bar{\psi}_{\star}^{c}\bigg)\ .
      \end{split}
    \end{equation*}
  \end{enumerate}
  By the discussion above, we can write \eqref{eq:mainintegral2} as the
  following:
  \begin{equation}\label{eq:rec2*}
    \begin{split}
      &\frac{\mu_{\beta}(\frac{\lambda-D_{\theta}}{\delta})}{\delta}\cdot(\frac{\lambda-D_{\theta}}{\delta})^{c}
      +\sum_{m=0}^{\infty}\sum_{\substack{\beta_{0}+\beta_{1}+\cdots+\beta_{m}=\beta\\p_{1}+\cdots+p_{m}=p}}\frac{1}{m!}(\widetilde{ev_{0}})_{*}\\
      &\bigg([\mathcal K_{0,[m]\cup\{0\}}(Y,\beta_{0})]^{\mathrm{vir}}\cap
      \bigcap_{i=1}^{m}ev_{i}^{*}(\mu_{\beta_{i},p_{i}}(-\bar{\psi}_{i}))\bigcap
      \frac{\frac{1}{\delta_{0}}(\frac{\lambda-ev^*_0D_{\theta}}{\delta_{0}})^{c}}{\frac{\lambda-ev^*_0D_{\theta}}{\delta_{0}}-\bar{\psi}_{0}}\bigg)\\ 
      &-\sum_{m=0}^{\infty}\sum_{\substack{\beta_{\star}+\beta_{1}+\cdots+\beta_{m}\\p_{1}+\cdots+p_{m}=p}}\frac{1}{m!}(\widetilde{ev_{\star}})_{*}\bigg(\sum
      _{d=0}^{\infty}\epsilon_{*}(c_{d}(-R^{\bullet}\pi_{*}\mathcal
      L_{\theta}^{\frac{1}{r}})(\frac{\lambda}{r})^{-1+m-d}(-1)^{d}\\
      &\cap [\mathcal
      K_{0,\vec{m}_{r}\cup\star}(\sqrt[r]{L_{\theta}/Y},\beta_{\star})]^{\mathrm{vir}})\cap
      \prod_{i=1}^{m}\frac{ev_{i}^{*}(\frac{1}{\delta_{i}}(f_{\beta_{i},p_{i}}(z)|_{z=\frac{\lambda-D_{\theta}}{\delta_{i}}})}{\frac{\lambda-ev^{*}_{i}D_{\theta
          }}{r\delta_{i}}+\frac{\bar{\psi_{i}}}{r}}\cap
      \bar{\psi}_{\star}^{c}\bigg).
    \end{split}
  \end{equation}

  As \eqref{eq:mainintegral2} lies in $H^{*}(\bar{I}_{\mu}Y, \mathbb Q)[\lambda]$,
  the coefficient of $\lambda^{-1}$ term in \eqref{eq:rec2*} must vanish. Note
  that the coefficients before $\lambda^{-1}$ in the first two terms in
  \eqref{eq:rec2*} yields (after replacing the index $0$ by $\star$)
  \begin{equation*}
    \sum_{m=0}^{\infty}\sum_{\substack{\beta_{\star}+\beta_{1}+\cdots+\beta_{m}=\beta\\p_1+\cdots+p_m=p}}\frac{1}{m!}\phi^{\alpha}\langle{\mu_{\beta_{1},p_{1}}(-\bar{\psi}_{1}),\cdots,\mu_{\beta_{m},p_{m}}(-\bar{\psi}_{m}),\phi_{\alpha}\bar{\psi}^{c}_{\star}}
    \rangle_{0,[m]\cup \star,\beta_{\star}},
  \end{equation*} 
  which is the left hand side of equality in \eqref{eq:rec2}. Then we extract the
  coefficient of the $\lambda^{-1}$ term in the third term in \eqref{eq:rec2*}, note
  $f_{\beta_{i},p_{i}}=0$ when $(\beta_{i},p_{i})=0$, this yields the term on the right hand side of
  \eqref{eq:rec2} up to a minus sign. This completes the proof of
  \eqref{eq:rec2}.
\end{proof}
\begin{lemma}\label{lem:simp-fbeta}
  when $s$ is sufficiently large, one has
  \begin{equation}\label{eq:simp-fbeta}
    \begin{split}
      &(\widetilde
      {ev_{0}})_{*}\bigg(\epsilon'_{*}\big(\sum_{d=0}^{\infty}c_{d}(-R^{\bullet}\pi_{*}\mathcal
      L_{-\theta}^{\frac{1}{s}})(\frac{\lambda}{s})^{-d} \\ &\cap [\mathcal K_{0,\vec{m}_{s}\cup\{0\}}(\sqrt[s]{L_{-\theta}/Y},\beta_{0})]^{\mathrm{vir}}\big)\cap \gamma
      \cap \bigcap_{i=1}^{m}ev_{i}^{*}(\mu_{\beta_{i},p_{i}}(-\bar{\psi}_{i}))\bigg)\\
      &=\frac{1}{s}(\widetilde{ev_{0}})_{*}\big([\mathcal
      K_{0,[m]\cup\{0\}}(Y,\beta_{0})]^{\mathrm{vir}}\cap\gamma \cap\bigcap_{i=1}^{m}ev_{i}^{*}(\mu_{\beta_{i},p_{i}}(-\bar{\psi}_{i}))\big)\
      ,
    \end{split}
  \end{equation} 
  Here $\epsilon': \mathcal K_{0,\vec{m}_{s}\cup \{0\}}(\sqrt[s]{L_{-\theta
    }/Y},\beta_{0})\rightarrow \mathcal K_{0,\vec{m}\cup\{0\}}(Y,\beta_{0})$ is
  the natural structure map, where $\vec{m}_{s}\cup \{0\}:=((g^{-1}_{\beta_{i}},\mu_{s}^{\beta_{i}(L_{\theta})+p_{i}}):1\leq
  i\leq m)\cup \{(g_{\beta},\mu_{s}^{-\beta(L_{\theta})-p})\}$, $\vec{m}\cup\{0\}:=(g^{-1}_{\beta_{i}}:1\leq
  i\leq m)\cup \{g_{\beta}\}$, $\beta=\sum_{i=0}^{m}\beta_{i}$,
  $p=\sum_{i=1}^{m}p_{i}$ and $\gamma$ is any cycle in $A_{*}(\mathcal K_{0,\vec{m}\cup \{0\}}(\sqrt[s]{L_{-\theta}/Y},
\beta))_{\mathbb Q}$.
\end{lemma}
\begin{proof}
  As $\mu_{\beta_{i},p_{i}}$ belongs to the twist sector
  $H^{*}(\bar{I}_{g^{-1}_{\beta_{i}}}Y,\mathbb Q)$, by Riemann-Roch theorem,
  it's equivalent to prove the equality by replacing the right hand side of
  \eqref{eq:simp-fbeta} by
   $$\frac{1}{s}(\widetilde{ev_{0}})_{*}\bigg([\mathcal K_{0,
        \vec{m}\cup\{0\}}(Y,\beta_{0})]^{\mathrm{vir}}\cap \gamma \cap \bigcap_{i=1}^{m}ev_{i}^{*}(\mu_{\beta_{i},p_{i}}(-\bar{\psi}_{i}))\bigg)\ .$$
  We will first show that $R^{0}\pi_{*}\mathcal L^{\frac{1}{s}}_{-\theta}=0$ on $\mathcal K_{0,\vec{m}\cup\{0\}}(\sqrt[s]{L_{-\theta}/Y},\beta_{0})$, which
  implies that
  $R^{1}\pi_{*}\mathcal
  L^{\frac{1}{s}}_{-\theta}=0$ as $R^{\bullet}\pi_{*}\mathcal
  L^{\frac{1}{s}}_{-\theta}$ has virtual rank $0$ when $s$ is sufficiently
  large. By Remark \ref{rmk:vanishH0}, when $\beta_{0}\neq 0$, we have
  $R^{0}\pi_{*}\mathcal L^{\frac{1}{s}}_{-\theta}=0$. So it remains to prove
  the case when $\beta_{0}=0$. Assume now that $\beta_{0}=0$, as the vertex $v$
  over $0$ is
  stable, there must be some marked points on $C_{v}$. Assume $q_{i}$ is one of the marked
  points with insertion $\mu_{\beta_{i},p_{i}}$. Without loss of generality, we can assume $(\beta_{i},p_{i})\neq 0$ for all
  $i$ as $\mu_{0}(z)=0$ by the very definition. Note we have
  $$age_{q_{i}}((\mathcal L^{\frac{1}{s}}_{-\theta})|_{C_{v}})=\frac{\beta_{i}(L_{\theta})+p_{i}}{s}\neq
  0\ ,$$ then the restricted line bundle $L^{\frac{1}{s}}_{-\theta}:=(\mathcal
  L^{\frac{1}{s}}_{-\theta})|_{C_{v}}$ can't have any nonzero section on
  $C_{v}$. Indeed the degree of the restriction of $L^{\frac{1}{s}}_{-\theta}$
  to every irreducible component is zero by Lemma \ref{lem:2.3} as the total
  degree $\beta_{0}$ is zero, then a nonzero section of $L^{\frac{1}{s}}_{-\theta}$ will
  trivialize the line bundle $L^{\frac{1}{s}}_{-\theta}$, this contradicts the
  fact that $L^{\frac{1}{s}}_{-\theta}$ has nontrivial stacky structure at
  $q_{i}$.
 
  Now as $-R^{\bullet}\pi_{*}\mathcal
  L^{\frac{1}{s}}_{-\theta}=R^{1}\pi_{*}\mathcal L^{\frac{1}{s}}_{-\theta}=0$,
  \eqref{eq:simp-fbeta} follows immediately from the identity
   $$\epsilon'_{*}([\mathcal K_{0,\vec{m}_{s}\cup \{0\}}(\sqrt[s]{L_{-\theta
     }}/Y,\beta_{0})]^{\mathrm{vir}})=\frac{1}{s}[\mathcal K_{0,\vec{m}\cup
     \{0\}}(Y,\beta_{0})]^{\mathrm{vir}}\ ,$$ which is proved in \cite[Theorem
   5.16]{tang16_quant_leray_hirsc_theor_banded_gerbes}.
 \end{proof}

\subsection{Proof of Main Theorem}
Using the notation in the introduction, now we prove the main theorem \ref{thm:main}:
\begin{proof}
According to the analysis in the introduction, it suffices to prove the
following:
\begin{equation}\label{eq:coewc3*}
\begin{split}
  &[z\mathbb I_{\beta,p}(z)]_{z^{-c-1}}=\\
  &\sum_{m=0}^{\infty}\sum_{\substack{\beta_{0}+\beta_{1}+\cdots+\beta_{m}=\beta\\p_{1}+\cdots+p_{m}=p}}\frac{1}{m!}\phi^{\alpha}\langle\mu_{\beta_{1},p_{1}}(-\bar{\psi}_{1}),\cdots,\mu_{\beta_{m},p_{m}}(-\bar{\psi}_{m}),\phi_{\alpha}\bar{\psi}_{\star}^{c}
  \rangle_{0,[m]\cup \star,\beta_{0}}\ ,
\end{split}
  \end{equation}
for any nonnegative integer $c$ and degree $(\beta,p)$. 
Let's assume that \eqref{eq:coewc3*} is proved for all degrees $(\beta',p')\in
\mathrm{Eff}(W,G,\theta)\!\times\!\mathbb N$ with
$\beta'(L_{\theta})+p'<\beta(L_{\theta})+p$. Then $f_{\beta_{i},p_{i}}(z)$ in
\eqref{eq:rec2} is equal to $z \mathbb I_{\beta_{i},p_{i}}(z)$ by induction (note one can
assume that $(\beta_{i},p_{i})\neq 0$). Indeed, first
notice that, in \eqref{eq:rec2}, $f_{\beta_{i},p_{i}}(z)$ appears only for the graph
$\Gamma$ which
has stable vertex over $\infty$ with degree $(\beta_{\star},\frac{\delta_{\star}}{r})$. When $\beta_{\star}$ is nonzero,
it immediately follows that $\beta_{i}(L_{\theta})<\beta(L_{\theta})$, hence $\beta_{i}(L_{\theta})+p_{i}<\beta(L_{\theta})+p$; 
otherwise, when $\beta_{\star}=0$, as the unique vertex $v$ over $\infty$ is a stable vertex, then there
are at least two edges in $\Gamma$, which implies that
$\beta_{i}(L_{\theta})+p_{i}<\beta(L_{\theta})+p$ as each vertex over $0$ has nonzero
total degree as it's equal to the degree $\delta(e)$ of the edge incident to the
vertex by \eqref{eq:adm-graph}.  
Then \eqref{eq:coewc3*} immediately follows from Theorem \ref{thm:rec1} and
\ref{thm:rec2}.   
\end{proof}

\section{An example}\label{sec:corti-example}
In this section, we will recover the quantum product computation by Corti for a
cubic hypersurface $Y$ which is cut off by the polynomial $x_{1}^{3}+x_{2}^{3}+x_{3}^{3}+x_{4}x_{1}$ in $\P(1,1,1,2)$. The following is the table for (small) quantum
product of $Y$ obtained by Corti (see~\cite{guest08_orbif_quant_d_modul_assoc}):

\renewcommand{\AA}{\tfrac16 q^{\tfrac {A^A}A} \mathbbm{1}_{\tfrac2{A_A}}}
\[
\begin{array}{c|ccccc}
\vphantom{\AA}
 & \ \ \mathbbm {1}\ \  & \ \ p\ \  &\ \  p^2\ \  & \ \ \mathbbm{1}_{\frac12}\ \ 
\\
\hline
\vphantom{\AA}
\mathbbm{1}  & \mathbbm{1} & p & p^2 & \mathbbm{1}_{\frac12}
 \\
\vphantom{\AA}
p  &  & p^2+12r^2+3r \mathbbm{1}_{\frac12}  & 12r^2p & r p
\\
\vphantom{\AA}
p^2  &  &  & 108r^4+36r^3\mathbbm{1}_{\frac12} & 12r^3
\\
\vphantom{\AA}
\mathbbm{1}_{\frac12}  &  &  &  & \tfrac13 p^2 - 3r\mathbbm{1}_{\frac12}
\\
\end{array}
 \]
Here $r=\frac{1}{2}q$ \footnote{In~\cite{guest08_orbif_quant_d_modul_assoc},
  they use $r=\frac{1}{2}q^{\frac{1}{2}}$, their $q^{\frac{1}{2}}$ corresponds
  our $q$ here.}, $p$ is the hyperplane class of $Y$ and $\mathbbm
1_{\frac{1}{2}}$ is the fundamental class of the unique nontrivial twisted
sector of $H^{*}(\bar{I}_{\mu}Y,\mathbb Q)$. Due to the discussion
in~\cite{guest08_orbif_quant_d_modul_assoc}, the usual twisted $I-$function of
$\P(1,1,1,2)$ only recovers the first two rows, and the rest two rows rely on Corti's
key calculation

\begin{equation}\label{eq:corti}(\mathbbm{1}_{\frac{1}{2}}\circ\mathbbm{1}_{\frac{1}{2}},\mathbbm{1}_{\frac{1}{2}} )= \frac{-3}{2}r\ .\end{equation}

In the following, we will recover Corti's key calculation using the $I-$function
by choosing a different GIT presentation of $\P(1,1,1,2)$.

Choose the matrix
\[
  \rho =
  \begin{pmatrix}
    1 & 1  & 1 & 2 & 0 \\
    0 & 0 & 0 & 1 & 1\\
  \end{pmatrix}
  \colon \mathbb Z^2 \to \mathbb Z^5\ ,
\]
which gives the action of $G:=\C^{*}_{t}\times \C^{*}_{z}$ on $W:=\mathbb C^{5}$ so that the GIT (stack)
quotient is still $\P(1,1,1,2)$ (with the choice of stability condition
$\theta=t^{2}z^{3}$, this also corresponds to the S-extended data
$S=\{\frac{1}{2}\}$ in the sense of~\cite{coates2019some, Coates_2015}). Consider the polynomial
$zx^{3}_{1}+zx_{2}^{3}+zx_{3}^{3}+x_{4}x_{1}$, then it cuts off a hypersurface in the
new GIT stack quotient $[W^{ss}(\theta)/G]$, which is isomorphic to $Y$. Note this hypersurface doesn't come from a semi-positive line
bundle on $[W/G]$ in the sense of \ref{def:pos-lin}, in this case $Y$ comes from
the line bundle $L_{t^{3}z}$ on $[W/G]$. 

The monoid $\mathrm{Eff}(W,G,\theta)$ is generated by
$\beta_{1},\beta_{2}\in Hom(\chi(G),\mathbb Q)$ such that
\[
  \begin{pmatrix}
    \beta_{1}(L_{t})&\beta_{1}(L_{z})\\
    \beta_{2}(L_{t})&\beta_{2}(L_{z})\\
  \end{pmatrix}
  =
  \begin{pmatrix}
    \frac{1}{2}&0\\
    -\frac{1}{2}&1\\
  \end{pmatrix}\ .
\]
Then we can think $q:=q^{\beta_{1}}$ generates the semigroup of degrees of \emph{stable maps} to the hypersurface $Y$ and $x:=q^{\beta_{2}}$ is a formal variable.

By \S \ref{subsec:special-case}, the small $I-$function of
$Y$ using this new GIT presentation of $\P(1,1,1,2)$ is 

\begin{equation}
\begin{split}
I(q,x,z) = &\sum_{\substack{(l,k) \in \mathbb{N}^2\\\frac{3l-k}{2}\geq 0}} \frac{q^l x^{k}} {z^{k} k!} \frac{\prod_{i<0}(p+(\frac{l-k}{2}-i)z)^{3}}{\prod_{i<\frac{l-k}{2}}(p+(\frac{l-k}{2}-i)z)^{3}}\\
  &\frac{1}{ \prod_{0\leq i<l} (2p+(l-i)z) } \prod_{0\leq i <\frac{3l-k}{2} }\bigg(3p+(\frac{3l-k}{2}-i)z\bigg) \mathbbm 1_{\frac{k-l}{2}}\\
&+\sum_{\substack{(l,k) \in \mathbb{N}^2\\\frac{3l-k}{2}\in \Z_{<0}}} \frac{q^l x^{k}} {z^{k} k!} \frac{\prod_{\frac{l-k}{2}<i<0}(p+(\frac{l-k}{2}-i)z)^{3}}{\prod_{0\leq i<l}(2p+(l-i)z)}\\
  &\frac{1}{\prod_{\frac{3l-k}{2}<i<0}\bigg(3p+(\frac{3l-k}{2}-i)z\bigg)}\frac{1}{3}p^{2}\\
&+\sum_{\substack{(l,k) \in \mathbb{N}^2\\\frac{3l-k}{2}\in Q_{<0}\backslash\Z_{<0}}} \frac{q^l x^{k}} {z^{k} k!} \frac{\prod_{\frac{l-k}{2}<i<0}(p+(\frac{l-k}{2}-i)z)^{3}}{\prod_{0\leq i<l}(2p+(l-i)z)}\\
  &\frac{1}{\prod_{\frac{3l-k}{2}<i<0}\bigg(3p+(\frac{3l-k}{2}-i)z\bigg)}\mathbbm 1_{\frac{k-l}{2}}\ .
\end{split}
\end{equation}
where $\mathbbm 1_{\frac{k-l}{2}}=\mathbbm 1_{\frac{1}{2}}$ if $k-l$ is odd,
otherwise $\mathbbm 1_{\frac{k-l}{2}}=\mathbbm 1$. We can show the following fact about $I(q,x,z)$:
\begin{equation}\label{eq:1.3}I(q,x,z)=\mathbbm{1}+\frac{x \mathbbm 1_{\frac{1}{2}}+qx \mathbbm 1}{z}+\mathcal O(x^{3})+\mathcal O(\frac{1}{z^{2}})\ ,\end{equation}
\begin{equation}\label{eq:1.4}\frac{\partial I(q,x,z)}{\partial x}=\frac{\mathbbm 1_{\frac{1}{2}}+q\mathbbm
  1}{z}+\frac{x(q^{2} \mathbbm 1+\frac{1}{3}p^{2}+\frac{q \mathbbm
    1_{\frac{1}{2}}}{2})}{z^{2}}+\mathcal O(x^{2})+\mathcal O(\frac{1}{z^{3}})\ ,\end{equation}
and
\begin{equation}\label{eq:1.4*}\frac{\partial^{2} I(q,x,z)}{\partial^{2} x}=
  \frac{q^{2} \mathbbm 1+\frac{1}{3}p^{2}+\frac{q \mathbbm
    1_{\frac{1}{2}}}{2}}{z^{2}}+\mathcal O(x)+\mathcal O(\frac{1}{z^{3}})\ .\end{equation}

Since $-ze^{\frac{qx}{z}}I(q,x,-z)$ is a slice on the
Givental's cone by string flow. Note we have the asympotic expansion
\begin{equation}\label{eq:1.5}ze^{\frac{-qx \mathbbm 1}{z}}I(q,x,z)=z \mathbbm 1+x \mathbbm 1_{\frac{1}{2}}+\mathcal
O(x^{3})+\mathcal O(\frac{1}{z})\ .\end{equation}
Then
\begin{equation}\label{eq:1.6}ze^{\frac{-qx \mathbbm 1}{z}}I(q,x,z)=J^{Giv}(q,x \mathbbm 1_{\frac{1}{2}},z)+\mathcal O
(x^{3})\ ,\end{equation}
where $J^{Giv}(q,t,z)$ is Givental's $J-$function which has an asymptotic
expansion
\begin{equation}z \mathbbm 1+t+\mathcal O(\frac{1}{z})\ ,\end{equation}
and $t=\sum t^{\alpha}\phi_{\alpha}\in H^{*}(\bar{I}_{\mu}Y, \mathbb
Q)$. We have the following standard fact about Givental's $J-$function (c.f.~\cite{Givental_2004}):
\begin{equation}\label{eq:1.8} z\frac{\partial }{\partial t_{\alpha}}\frac{\partial}{\partial
  t_{\beta}}J^{Giv}(q,t,z)=\phi_{\alpha}\star_{t}\phi_{\beta}+\mathcal
O(z^{-1})\ .\end{equation}

Now consider the function
\begin{equation}\label{eq:star}
z\frac{\partial^{2}}{\partial^{2}x}\bigg(ze^{\frac{-qx \mathbbm 1}{z}}I(q,x,z)\bigg)\ , 
\end{equation}
a direction computation using product rule yields:
\begin{equation}\label{eq:1.9}q^{2}e^{\frac{-qx \mathbbm 1}{z}}I(q,x,z)-2zqe^{\frac{-qx \mathbbm 1}{z}}\frac{\partial}{\partial
x}I(q,x,z)+z^{2}e^{\frac{-qx \mathbbm 1}{z}}\frac{\partial^{2}}{\partial^{2}x}I(q,x,z) \ .\end{equation}
Apply \eqref{eq:1.3}, \eqref{eq:1.4}, \eqref{eq:1.4*} to the first term, second
term and third term in \eqref{eq:1.9}, respectively, we have the following asymptotic expansion of \eqref{eq:star}:
\begin{equation}\label{eq:asy1}
q^{2}\mathbbm 1-2q(\mathbbm 1_{\frac{1}{2}}+q\mathbbm 1)+(q^{2} \mathbbm 1+\frac{1}{3}p^{2}+\frac{q \mathbbm
    1_{\frac{1}{2}}}{2})+ \mathcal O(x)+\mathcal O(z^{-1})\ .
\end{equation}
On the other hand, using equation \eqref{eq:1.6}, \eqref{eq:1.8}, one has another asymptotic
expansion about \eqref{eq:star}:
\begin{equation}\label{eq:asy2}
\mathbbm 1_{ \frac{1}{2}}\star_{x}\mathbbm 1_{\frac{1}{2}}+\mathcal O(x)+\mathcal O(z^{-1})\ .
\end{equation}
Compare \eqref{eq:asy1} and \eqref{eq:asy2}, after evaluating $x=0$ and ignoring
all negative $z$ powers, we have
$$\mathbbm 1_{ \frac{1}{2}}\circ \mathbbm
1_{\frac{1}{2}}=\frac{1}{3}p^{2}-\frac{3}{2}q \mathbbm 1_{\frac{1}{2}} \ ,$$
which recovers Corti's calculation\footnote{Here the small quantum product
  $\circ$ is defined by the specialization of the big quantum product
  $\star_{t}$ to $t=0$.} \eqref{eq:corti}!

\bibliographystyle{amsxport}
\bibliography{references}

\end{document}